\documentclass[a4paper,fleqn,leqno]{article}

\usepackage{latexsym,amssymb,amsthm,amsmath,mathtools,epsfig,color,epsf}
\usepackage[english]{babel}
\usepackage[normalem]{ulem}

\usepackage{hyperref}
\usepackage[utf8]{inputenc}
\usepackage{graphicx, tikz, pgfplots,pgfplotstable}
\usepackage{ifthen,tkz-euclide,txfonts}
\usepackage[ampersand]{easylist}
\pgfplotsset{compat=1.5}
\usepackage{caption}
\usepackage{subcaption}

 \makeatletter
 \@addtoreset{equation}{section}
 \makeatother

 \newcounter{extralabel}[section]

 \newtheorem{ittheorem}{Theorem}
 \newtheorem{itlemma}{Lemma}
 \newtheorem{itproposition}{Proposition}
 \newtheorem{itdefinition}{Definition}
 \newtheorem{itcorollary}{Corollary}
 \newtheorem{itconjecture}{Conjecture}
 \newtheorem{itremark}{Remark}
 \newtheorem{itassumption}{Assumption}

 \newenvironment{theorem}{\addtocounter{extralabel}{1}
 \begin{ittheorem}}{\end{ittheorem}}

 \newenvironment{lemma}{\addtocounter{extralabel}{1}
 \begin{itlemma}}{\end{itlemma}}

 \newenvironment{proposition}{\addtocounter{extralabel}{1}
 \begin{itproposition}}{\end{itproposition}}

 \newenvironment{definition}{\addtocounter{extralabel}{1}
 \begin{itdefinition}}{\end{itdefinition}}

 \newenvironment{corollary}{\addtocounter{extralabel}{1}
 \begin{itcorollary}}{\end{itcorollary}}

 \newenvironment{remark}{\addtocounter{extralabel}{1}
 \begin{itremark}}{\end{itremark}}

\newcommand{\Rand}[1]{\marginpar{#1}}

\newcommand{\be}[1]{\Rand{\vspace{0,6cm}\tt #1}\begin{equation}\label{#1}}
\newcommand{\beL}[1]{\Rand{\vspace{0,6cm}\tt #1}\begin{lemma}\label{#1}}
\newcommand{\belC}[2]{\Rand{\vspace{0,6cm}\tt #1}\begin{lemma}[#2]\label{#1}}
\newcommand{\beP}[1]{\Rand{\vspace{0,6cm}\tt #1}\begin{proposition}\label{#1}}
\newcommand{\bePC}[2]{\Rand{\vspace{0,6cm}\tt #1}\begin{proposition}[#2]\label{#1}}
\newcommand{\beD}[1]{\Rand{\vspace{0,6cm}\tt #1}\begin{definition}\label{#1}}
\newcommand{\beT}[1]{\Rand{\vspace{0,6cm}\tt #1}\begin{theorem}\label{#1}}
\newcommand{\beC}[1]{\Rand{\vspace{0,6cm}\tt #1}\begin{corollary}\label{#1}}
\newcommand{\bePr}[1]{\Rand{\vspace{0,6cm}\tt #1}\begin{proof}\label{#1}}

\newcommand{\bea}[1]{\Rand{\vspace{0,7cm}\tt #1\vspace{-0,7cm}}\begin{eqnarray}\label{#1}}

\renewcommand{\d}{{\rm d}}
\newcommand{\e}{{\rm e}}

\newcommand{\E}{\mathbb{E}}
\renewcommand{\P}{\mathbb{P}}

 \def\1{{\mathchoice {1\mskip-4mu\mathrm l} 
{1\mskip-4mu\mathrm l}
{1\mskip-4.5mu\mathrm l} {1\mskip-5mu\mathrm l}}}

\def\CB{\mathcal{B}}
\def\CC{\mathcal{C}}

\def\CE{\mathcal{E}}
\def\CF{\mathcal{F}}
\def\CG{\mathcal{G}}

\def\CL{\mathcal{L}}
\def\CI{\mathcal{I}}

\def\CT{\mathcal{T}}

\def\CR{\mathcal{R}}
\def\CP{\mathcal{P}}

\def\E{\mathbb{E}}
\def\G{\mathbb{G}}

\def\N{\mathbb{N}}
\def\P{\mathbb{P}}

\def\R{\mathbb{R}}
\def\S{\mathbb{S}}

\def\Z{\mathbb{Z}}
\def\Im{\mathrm{i}}

\DeclareMathSymbol{\varNu}{\mathord}{letters}{78}

\newcommand{\ee}{\end{equation}}
\newcommand{\eea}{\end{eqnarray}}
\newcommand{\bean}{\begin{eqnarray*}}
\newcommand{\eean}{\end{eqnarray*}}

\newcommand{\sign}{\mathrm{sgn}\,}

\definecolor{rood}{rgb}{1,0,0}
\definecolor{rood1}{rgb}{0.9,0,0.1}
\definecolor{rood2}{rgb}{0.8,0,0.2}
\definecolor{rood3}{rgb}{0.7,0,0.3}
\definecolor{rood4}{rgb}{0.6,0,0.4}
\definecolor{rood5}{rgb}{0.5,0,0.5}
\definecolor{rood6}{rgb}{0.4,0,0.6}
\definecolor{rood7}{rgb}{0.3,0,0.7}
\definecolor{rood8}{rgb}{0.2,0,0.8}
\definecolor{rood9}{rgb}{0.1,0,0.9}
\definecolor{rood9}{rgb}{0,0,1}

\usepackage{color}
\usepackage{colortbl}

\newtheorem{xx}{\bf xxx}

\newtheorem{zz}{\bf zzz}

\newtheorem{yy}{\bf yyy}

\definecolor{groen}{rgb}{0,0.5,0.2}

\begin{document}

%%%%%%%%% TITLE PAGE %%%%%%%%%%%%%%%%%%%%%%%%%

\title{Spatial populations with seed-bank:\\
well-posedness, duality and equilibrium}

\author{Andreas Greven$^1$, Frank den Hollander$^2$, Margriet Oomen$^3$}

\date{April 27, 2020}

\maketitle

\begin{abstract}
We consider a system of interacting Fisher-Wright diffusions with seed-bank. Individuals live in colonies and are subject to resampling and migration as long as they are {\em active}. Each colony has a structured seed-bank into which individuals can retreat to become {\em dormant}, suspending their resampling and migration until they become active again. As geographic space labelling the colonies we consider a countable Abelian group $\G$ endowed with the discrete topology. The key example of interest is the Euclidean lattice $\G=\Z^d$, $d \in \N$. Our goal is to \emph{classify} the long-time behaviour of the system in terms of the underlying model parameters. In particular, we want to understand in what way the seed-bank enhances genetic diversity. 

We introduce three models of increasing generality, namely, individuals become dormant:  (1) in the seed-bank of their colony; (2) in the seed-bank of their colony while adopting a \emph{random colour} that determines their wake-up time; (3) in the seed-bank of a \emph{random colony} while adopting a \emph{random colour}. The extension in (2) allows us to model wake-up times with fat tails while preserving the Markov property of the evolution. The extension in (3) allows us to place individuals in different colony when they become dormant. For each of the three models we show that the system of continuum stochastic differential equations, describing the population in the large-colony-size limit, has a unique strong solution. We also show that the system converges to a unique equilibrium depending on a single \emph{density parameter} that is determined by the initial state, and exhibits a dichotomy of \emph{coexistence} (= locally multi-type equilibrium) versus \emph{clustering} (= locally mono-type equilibrium) depending on the parameters controlling the migration and the seed-bank. 

The seed-bank slows down the loss of genetic diversity. In model (1), the dichotomy between clustering and coexistence is determined by migration only. In particular, clustering occurs for recurrent migration and coexistence occurs for transient migration, as for the system without seed-bank. In models (2) and (3), an interesting interplay between migration and seed-bank occurs. In particular, the dichotomy is affected by the seed-bank when the wake-up time has infinite mean. For instance, for \emph{critically recurrent migration} the system exhibits clustering for finite mean wake-up time and coexistence for infinite mean wake-up time. Hence, at the \emph{critical dimension} for the system without seed-bank, \emph{new universality classes} appear when the seed-bank is added. If the wake-up time has a sufficiently fat tail, then the seed-bank determines the dichotomy and migration has no effect at all. 

The presence of the seed-bank makes the proof of convergence to a unique equilibrium a conceptually delicate issue. By combining duality arguments with coupling techniques, we show that our results also hold when we replace the Fisher-Wright diffusion function by a more general diffusion function, drawn from an appropriate class. 

\medskip\noindent
\emph{Keywords:} 
Fisher-Wright diffusion, resampling, migration, seed-bank, duality, equilibrium, coexistence 
versus clustering, 

\medskip\noindent
\emph{MSC 2010:} 
Primary 
60J70, % Applications of Brownian motions and diffusion theory 
            % (population genetics, absorption problems, etc.) 
60K35; % Interacting random processes; statistical mechanics type models; percolation theory 
Secondary 
92D25. % Population dynamics (general). 

\medskip\noindent 
\emph{Acknowledgements:} 
AG was supported by the Deutsche Forschungsgemeinschaft  (through grant DFG-GR 876/16-2 of SPP-1590), FdH and MO were supported by the Netherlands Organisation for Scientific Research (through NWO Gravitation Grant NETWORKS-024.002.003). FdH was also supported by the Alexander von Humboldt Foundation (during extended visits to Bonn and Erlangen in the Fall of 2019).  
\end{abstract}

\bigskip

\footnoterule
\noindent
\hspace*{0.3cm} {\footnotesize $^{1)}$ 
Department Mathematik, Universit\"at Erlangen-N\"urnberg, Cauerstrasse 11,
D-91058 Erlangen, Germany\\
greven@mi.uni-erlangen.de}\\
\hspace*{0.3cm} {\footnotesize $^{2)}$ 
Mathematisch Instituut, Universiteit Leiden, Niels Bohrweg 1, 2333 CA  Leiden, NL\\
denholla@math.leidenuniv.nl}\\
\hspace*{0.3cm} {\footnotesize $^{3)}$ 
Mathematisch Instituut, Universiteit Leiden, Niels Bohrweg 1, 2333 CA  Leiden, NL\\
m.oomen@math.leidenuniv.nl}

%%%%%%%%%%%%%%%%%%%%%%%%%%%%%%%%

\tableofcontents

\newpage

%%%%%%%%%%%%%%%%%%%%%%%%%%%%%%%%%

\section{Background and outline}
\label{s.introduct}

%%%

\subsection{Background and goals}
\label{ss.back}

In populations with a seed-bank, individuals can become dormant and stop reproducing themselves, until they can become active and start reproducing themselves again. In \cite{BCEK15} and \cite{BCKW16}, the evolution of a population evolving according to the Fisher-Wright model with a {\em seed-bank} was studied. In this model individuals are subject to \emph{resampling} and can move in and out of a seed-bank. While in the seed-bank they suspend resampling, i.e., the seed-bank acts as a repository for the genetic information of the population. Individuals that do not reside in the seed-bank are called \emph{active}, those that do are called \emph{dormant}. In the present paper we extend the single-colony Fisher-Wright model with seed-bank introduced in \cite{BCKW16} to a multi-colony setting in which individuals live in different colonies and move between colonies. In other words, we introduce \emph{spatialness}. 

Seed-banks are observed in many taxa, including plants, bacteria and other micro-organisms. Typically, they arise as a response to unfavourable environmental conditions. The dormant state of an individual is characterised by low metabolic activity and interruption of phenotypic development (see e.g.\ Lennon and Jones~\cite{LJ11}). After a varying and possibly large number of generations, dormant individuals can be resuscitated under more favourable conditions and reprise reproduction after having become active. This strategy is known to have important implications for population persistence, maintenance of genetic variability and stability of ecosystems. It acts as a \emph{buffer} against evolutionary forces such as genetic drift, selection and environmental variability. The importance of this evolutionary trait has led to several attempts to model seed-banks from a mathematical perspective, see e.g.\ \cite{KKL01}, \cite{BCKS13}, \cite{CAESKSB14}, \cite{BBGCWB19}. In \cite{BCKW16} it was shown that the continuum model obtained by taking the large-colony-size limit of the individual-based model with seed-bank is the Fisher-Wright diffusion with seed-bank.  Also the long-time behaviour and the genealogy of the continuum model with seed-bank were analysed in \cite{BCKW16}.  
 
In the present paper we consider a {\em spatial} version of the continuum model with seed-bank, in which individuals live in colonies, each with their own seed-bank, and are allowed to \emph{migrate} between colonies. Our goal is to understand the change in behaviour compared to the spatial model without seed-bank. The latter has been the object of intense study. A sample of relevant papers and overviews is \cite{S80}, \cite{D93}, \cite{DG93b}, \cite{DGV95}, \cite{DG96}, \cite{EF96}, \cite{dHS98}, \cite{dH06}, \cite{DGHSS08}, \cite{GHKK14}. We expect the presence of the seed-bank to affect the long-time behaviour of the system not only quantitatively but also qualitatively. To understand how this comes about, we must find ways to deal with the \emph{richer behaviour} of the population caused by the motion in and out of the seed-bank. Earlier work on a  spatial model with seed-bank, migration and mutation was carried out in \cite{HP16}, where the probability to be identical by decent for two individuals drawn randomly from two colonies was computed as a function of the distance between the colonies. 

It has been recognised that qualitatively different behaviour may occur when the wake-up time in the seed-bank changes from having a thin tail to having a fat tail \cite{LJ11}. One challenge in modelling seed-banks has been that fat tails destroy the Markov property for the evolution of the system. A key idea of the present paper is that we can \emph{enrich the seed-bank with internal states} -- which we call colours -- to allow for fat tails and still preserve the Markov property for the evolution. We will see that fat tails induce \emph{new universality classes}.

The main goals of the present paper are the following: 
\begin{itemize}
\item[(1)] 
Identify the typical features of the long-time behaviour of populations with a seed-bank. In particular, prove convergence to equilibrium, and identify the parameter regimes for \emph{clustering} (= convergence towards locally mono-type equilibria)
and \emph{coexistence} (= convergence towards locally multi-type equilibria). 
\item[(2)]
Identify the role of finite versus infinite mean wake-up time. Identify the \emph{critical dimension} in case the geographic space is $\Z^d$, $d \geq 1$, i.e., the dimension at which the crossover between clustering and coexistence occurs for migration with finite variance. 
\begin{itemize}
\item[(2a)]
Show that if the wake-up time has \emph{finite mean}, then the dichotomy between coexistence and clustering is controlled by the migration only and the seed-bank has no effect. In particular, clustering prevails when the symmetrised migration kernel is recurrent while coexistence prevails when it is transient. This is the classical dichotomy for populations without seed-bank \cite{CG94}. The critical dimension is $d=2$.
\item[(2b)]
Show that if the wake-up time has \emph{infinite mean} with \emph{moderately fat tails}, then the dichotomy is controlled by both the migration and the seed-bank. In particular, the parameter regimes for clustering and coexistence reveal an interesting interplay between rates for migration and rates for exchange with the seed-bank. The critical dimension is $1 < d < 2$.
\item[(2c)]
Show that if the wake-up time has \emph{infinite mean} with \emph{very fat tails}, then the dichotomy is controlled by the seed-bank only and the migration has no effect. The critical dimension is $d=1$. 
\end{itemize}
\end{itemize}

We focus on the situation where the individuals can be of two types. The extension to infinitely many types, called the Fleming-Viot measure-valued diffusion, only requires standard adaptations and will not be considered here (see \cite{DGV95}). Also, instead of Fisher-Wright resampling we will allow for state-dependent resampling, i.e., the rate of resampling in a colony depends on the fractions of the two types in that colony.  In what follows we only work with \emph{continuum} models, in which the components represent \emph{type frequencies} in the colonies labelled by a discrete geographic space.  

The techniques of proof that we use include duality, moment relations, semigroup comparisons and coupling. These techniques are standard, but have to be adapted to the fact that individuals move into and out of seed-banks. Since there is no resampling and no migration in the seed-bank, the motion of ancestral lineages in the dual process looses part of the random-walk structure that is crucial in models without seed-bank. Moreover, for seed-banks with infinite mean wake-up times, we encounter \emph{fat-tailed} wake-up time distributions in the dual process, and we need to deal with lineages that are dormant most of the time and therefore are much slower to coalesce. The coupling arguments also change. Already in a single colony, if the seed-bank has infinitely many internal states, then we are dealing with an infinite system in which the manipulation of Lyapunov functions and the construction of successful couplings from general classes of initial states is hard. In the multi-colony setting this becomes even harder, and conceptually challenging issues arise. 

%%%

\subsection{Outline}
\label{ss.outline}

In Section~\ref{s.basics} we introduce three models of increasing generality, establish their well-posedness via a martingale problem, and introduce their dual processes, which play a crucial role in the analysis. In Section~\ref{s.scaling} we state our main results.  
We focus on the long-time behaviour, prove convergence to equilibrium, and establish a \emph{dichotomy} between clustering and coexistence. We show that this dichotomy is \emph{affected by the presence of the seed-bank}, namely, the dichotomy depends not only on the migration rates, but can also depend on the relative sizes of the active and the dormant population and their rates of exchange. In particular, if the dormant population is much larger than the active population, then the residence time in the seed-bank has a fat tail that \emph{enhances genetic diversity significantly}. 

Sections~\ref{s.wpdualpr}--\ref{s.model3} are devoted to the proofs of the theorems stated in Sections~\ref{s.basics}--\ref{s.scaling}. In Appendix \ref{appA} we give the derivation of the single-colony continuum model from the single-colony individual-based Fisher-Wright model in the large-colony-size limit. In the individual-based model active individuals \emph{exchange} with dormant individuals, i.e., for each active individual that becomes dormant a dormant individual becomes active. In Appendix \ref{appB} we look at the continuum limit of the single-colony individual-based Moran model in which active and dormant individuals no longer exchange state but rather change state independently. We show that change instead of exchange does not affect the long-time behaviour. Appendices \ref{appC} and \ref{appe} contain the proof of technical lemmas that are needed in the proof of the convergence to equilibrium. 

In three upcoming companion papers \cite{GdHOpr2}, \cite{GdHOpr3}, \cite{GdHOpr4} we deal with three further aspects: 
\begin{itemize}
\item[(I)]
In \cite{GdHOpr2} we establish the \emph{finite-systems scheme}, i.e.,  we identify in the coexistence regime how a finite truncation of the system behaves as both the time and the truncation level tend to infinity, properly tuned together. This underlines the relevance of systems with an infinite geographic space and a seed-bank with infinitely many colours for the description of systems with a large finite geographic space and a seed-bank with a large finite number of colours. We show that there is a \emph{single universality class} for the scaling limit, represented by a Fisher-Wright diffusion whose volatility constant is {\em reduced} by the seed-bank. We show that if the wake-up time has finite mean, then the scaling time is \emph{proportional} to the geographical volume of the system, while if the wake-up time has infinite mean, then the scaling time grows \emph{faster} than the geographical volume of the system. We also investigate what happens for systems with a large finite geographic space and a seed-bank with infinitely many colours, where the behaviour turns out to be different.
\item[(II)] 
In \cite{GdHOpr3} we consider the special case where the colonies are organised in a \emph{hierarchical fashion}, i.e., the geographic space is the hierarchical group $\Omega_N$ of order $N$. We identify the parameter regime for clustering for all $N<\infty$, and analyse the \emph{multi-scale behaviour} of the system in the hierarchical mean-field limit $N \to \infty$ by looking at block averages on successive hierarchical space-time scales. Playing with the migration kernel, we can choose the migration to be close to critically recurrent in the sense of potential theory. By letting $N\to\infty$ we can approach the \emph{critical dimension}, so that the migration becomes similar to migration on the two-dimensional Euclidean geographic space. With the help of \emph{renormalisation arguments} we show that, close to the critical dimension, the scaling behaviour on large space-time scales is \emph{universal}.
\item[(III)] 
In \cite{GdHOpr4} we identify the \emph{pattern of cluster formation} in the clustering regime (= how fast mono-type clusters grow in time) and describe the \emph{genealogy} of the population. The latter provides further insight into how the seed-bank enhances genetic diversity. 
\end{itemize}  
In these papers too we will see that the seed-bank can cause not only quantitative but also qualitative changes in the scaling behaviour of the system.

%%%%%%%%%%% SECTION 2 %%%%%%%%%%%%%%%%%%%%%%%%%%%%%%

\section{Introduction of the three models and their basic properties}
\label{s.basics}

In Section ~\ref{ss.process} we give a formal definition of the three models of increasing generality. In Section~\ref{ss.comments} we comment on their biological significance. In Section~\ref{ss.wellpos} we establish their well-posedness via a martingale problem (Theorem~\ref{T.wellp}). In Section~\ref{ss.duality} we introduce the associated  dual processes and state the relevant duality relations (Theorems~\ref{T.dual1}, \ref{T.dual2} and \ref{T.dual3}). In Section~\ref{ss.dichcrit} we use these duality relations to formulate a criterion for clustering versus coexistence (Theorems~\ref{T.dichcrit.model1} and \ref{T.dichcrit.model2}).

%%%

\subsection{Migration, resampling and seed-bank: three models}
\label{ss.process}

In this section we extend the model for a population with seed-bank from \cite{BCKW16} to three models of increasing generality for  spatial populations with seed-bank. In each of the three models, we consider populations of individuals of two types -- either $\heartsuit$ or $\diamondsuit$ -- located in a {\em geographic space} $\G$ that is a countable Abelian group endowed with the discrete topology. In each of the three models, the population in a colony consist of an \emph{active} part and a \emph{dormant} part. The \emph{repository} of the dormant population at colony $i\in\G$ is called the seed-bank at $i\in\G$. Individuals in the active part of a colony $i\in\G$ can resample, migrate and exchange with a dormant population. Individuals in the dormant part of a colony $i\in\G$ only exchange with the active population. An active individual that resamples chooses uniformly at random another individual from its colony and adopts its type. (Alternatively, resampling may be viewed as the active individual being replaced by a copy of the active individual chosen. Because individuals carry a type and not a label, this gives the same model.) When an active individual at colony $i\in\G$ migrates, it chooses a parent from another colony $j\in\G$ and adopts its type. In each of the three models the migration is described by a {\em migration kernel} $a(\cdot,\cdot)$, which is  an irreducible $\G \times \G$ matrix of transition rates satisfying
\begin{equation}
\label{gh10}
a(i,j) = a(0,j-i) \quad \forall\,i,j \in \G, \qquad \sum_{i \in \G} a(0,i) < \infty.
\end{equation}
Here, $a(i,j)$ is to be interpreted as the rate at which an active individual at colony $i\in\G$ chooses a parent in the active part of colony $j\in\G$ and adopts its type. An active individual that becomes dormant \emph{exchanges} with a randomly chosen dormant individual that becomes active. The three models we discuss in the present paper differ in the way the active population exchanges with the dormant population. However, in each of the three models the exchange mechanism guarantees that the sizes of the active and the dormant population stay fixed over time. The dormant part of the population only evolves due to exchange of individuals with the active part of the population. 
 
Since we look at continuum models obtained from individual-based models, we are interested in the frequencies of type $\heartsuit$ in the different colonies. In Appendix~\ref{appA} we discuss the individual-based models underlying the continuum models described below.
 
\begin{remark}{{\bf [Notation]}}
\label{notation}
{\rm Throughout the paper we use lower case letters for \emph{components} and upper case letters for \emph{systems of components}.} \hfill $\Box$
\end{remark}

\paragraph{Model 1: single-layer seed-bank.}

Each colony $i \in \G$ has an active part $A$ and a dormant part $D$. Therefore we say that \emph{the effective geographic space} is given by $\G\times{\{A,D\}}$. For $i\in\G$ and $t \geq 0$, let $x_i(t)$ denote the fraction of individuals in colony $i$ of type $\heartsuit$ that are active at time $t$, and $y_i(t)$ the fraction of individuals in colony $i$ of type $\heartsuit$ that are dormant at time $t$. Then the system is described by the process 
\begin{equation}
\label{e409}
(Z(t))_{t\geq 0}, \qquad Z(t) = \big(z_i(t)\big)_{i\in\G}, \qquad z_i(t) = (x_i(t),y_i(t)),
\end{equation} 
on the state space 
\begin{equation}
\label{ss1}
E=\left([0,1]\times[0,1]\right)^\G,
\end{equation} 
and $(Z(t))_{t\geq0}$ evolves according to the following SSDE: 
\begin{eqnarray}
\label{gh1}
\d x_i(t) &=& \sum_{j \in \G} a(i,j)\, [x_j (t) - x_i(t)]\,\d t
+ \sqrt{d x_i(t)[1-x_i(t)]}\,\d w_i(t)\\ \nonumber
&&+\,Ke\,[y_i(t) - x_i(t)]\,\d t,\\[0.2cm]
\label{gh2}
\d y_i (t) &=& e\,[x_i(t) - y_i(t)]\,\d t,\qquad  i\in\G,
\end{eqnarray}
where $(w_i(t))_{t \geq 0}$, $i \in \G$, are independent standard Brownian motions. As initial state $Z(0)=z$ we may pick any $z\in E$. The first term in \eqref{gh1} describes the {\em migration} of active individuals at rate $a(i,j)$. The second term in \eqref{gh1} describes the {\em resampling} of individuals at rate $d \in (0,\infty)$. The third term in \eqref{gh1} together with the term in \eqref{gh2} describe the \emph{exchange} of active and dormant individuals at rate $e \in (0,\infty)$. 

%%%%%%%%%%%%%%%%%%%%%%%%%%%%%%%%%%%%%%%%%%%%%%%%%%%%
\medskip\noindent
\vspace{-0.5cm}
\begin{figure}[htbp]
\begin{center}
\begin{tikzpicture}
\draw [fill=red!20!blue!20!,ultra thick] (0,0) rectangle (2,4);
\draw [fill=red!20!blue!20!,ultra thick] (4,0) rectangle (6.5,4);
\node  at (1,2) {\text{\Large A}};
\node  at (5.25,2) {\text{\Large D}};
\draw[ultra thick,<-](2.25,2)--(3.75,2);
\draw[ultra thick,->](2.25,3)--(3.75,3);
\draw[ultra thick,<-](-1.5,2)--(0,2);
\draw[ultra thick,->](-1.5,3)--(0,3);
\node at (3,4.3) {\text{exchange}};
\node at (1,5.3) {\text{resampling}};
\node at (-1,4.3) {\text{migration}};
\draw [ultra thick] (1,3.75)to [out=180,in=180](1, 4.75) ;
\draw [ultra thick,->] (1,4.75)to [out=0,in=60](1.1, 3.75) ;
\node[above]  at (3,3) {$Ke$};
\node[above]  at (3,2) {$e$};
\node[above]  at (-0.75,3) {};
\node[above]  at (-0.75,2) {};
\node[left]  at (1.7,4.5) {$d$};
\end{tikzpicture}
\vspace{0.2cm}
\caption{\small The evolution in model 1. Individuals are subject to migration, resampling 
and exchange with the seed-bank.}
\label{fig:Model1}
\end{center}
\end{figure}
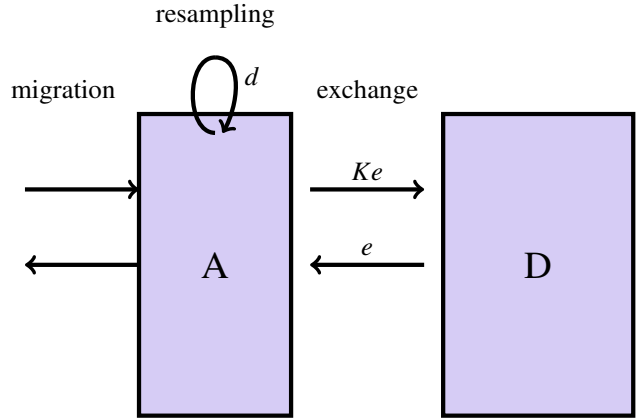
\vspace{-0.5cm}
%%%%%%%%%%%%%%%%%%%%%%%%%%%%%%%%%%%%%%%%%%%%%%%%%%%%

The factor $K \in (0,\infty)$ is defined by 
\begin{equation}
\label{ratio1}
K = \frac{\text{size dormant population}}{\text{size active population}}, 
\end{equation}
and is the same for all colonies $i\in\G$. The factor $K$ turns up in the scaling limit of the individual-based model when there is an {\em asymmetry} between the sizes of the active and the dormant population (see Appendix \ref{appA}). In Fig.~\ref{fig:Model1} we give a schematic illustration of the process \eqref{gh1}--\eqref{gh2}. A detailed description of the underlying individual-based model, as well as a derivation of the continuum limit \eqref{gh1}--\eqref{gh2} from the individual-based model following \cite{BCKW16}, can be found in Appendix \ref{appA}. The continuum limit is also referred to as the frequency limit or the diffusion limit. 

\begin{remark}{\bf [Interpretation of the state space.]}
\label{rm1ss}
{\rm Note that the \emph{state space} of the system can also be written as
\begin{equation}
\label{Edef}
E= [0,1]^\S, \qquad \S=\G\times\{A,D\}, 
\end{equation} 
where $A$ denotes the reservoir of the active population and $D$ the repository of the dormant population. With that interpretation, the process is denoted by
\begin{equation}
(Z(t))_{t \geq 0}, \qquad  Z(t) = \big(z_{u}(t)\big)_{u \in \S}
\end{equation} 
with $z_{u}(t)=x_i(t)$ if $u=(i,A)$ and $z_{u}(t) = y_i(t)$ if $u=(i,D).$ To analyse the system we need both interpretations of the state space.} \hfill$\Box$
\end{remark}

\paragraph{Model 2: multi-layer seed-bank.}

%%%%%%%%%%%%%%%%%%%%%%%%%%%%%%%%%%%%%%%%%%%%%%%%%%%%
\begin{figure}[htbp]
\begin{center}
\begin{tikzpicture}
\draw [fill=red!20!blue!20!,ultra thick] (0,0) rectangle (2,4);
\draw [fill=red!20!blue!20!,ultra thick] (4,0) rectangle (6.5,4);
\draw [fill=yellow!50!, ultra thick] (4.25,3.15) rectangle (6.25,3.85); 
\draw [fill=red!50!, ultra thick] (4.25,2.15) rectangle (6.25,2.85); 
\draw [fill=green!50!, ultra thick] (4.25,0.55) rectangle (6.25,1.25); 
\node  at (1,2) {\text{\Large A}};
\node  at (5.25,3.45) {\text{\Large $D_0$}};
\node  at (5.25,2.45) {\text{\Large $D_1$}};
\node  at (5.25,0.85) {\text{\Large $D_m$}};
\draw[ultra thick,<-](2.25,3.4)--(3.75,3.4);
\draw[ultra thick,->](2.25,3.7)--(3.75,3.7);
\draw[ultra thick,<-](2.25,2.4)--(3.75,2.4);
\draw[ultra thick,->](2.25,2.7)--(3.75,2.7);
\draw[ultra thick,<-](2.25,0.8)--(3.75,0.8);
\draw[ultra thick,->](2.25,1.1)--(3.75,1.1);
\node at (3,4.3) {\text{exchange}};
\node at (1,5.3) {\text{resampling}};
\draw [ultra thick] (1,3.75)to [out=180,in=180](1, 4.75) ;
\draw [ultra thick,->] (1,4.75)to [out=0,in=60](1.1, 3.75) ;
\foreach \x in {1.9,1.7,1.5,0.35,0.15}
\draw[fill] (5.25,\x) circle [radius=0.05];
\node[above]  at (3,3.6) {$K_0 e_0$};
\node[below]  at (3,3.49) {$e_0$};
\node[above]  at (3,2.6) {$K_1 e_1$};
\node[below]  at (3,2.49) {$e_1$};
\node[above]  at (3,1.0) {$K_me_m$};
\node[below]  at (3,0.89) {$e_m$};
\node[left]  at (1.7,4.5) {$d$};
\node at (-1,4.3) {\text{migration}};
\draw[ultra thick,<-](-1.5,2)--(0,2);
\draw[ultra thick,->](-1.5,3)--(0,3);	
\end{tikzpicture}
\vspace{.3cm}
\caption{\small The evolution in model 2. Individuals are subject to migration, resampling and 
exchange with the seed-bank, as in model 1. Additionally, when individuals become dormant 
they get a colour and when they become active they loose their colour.}
\label{fig:Model2}
\end{center}
\end{figure}
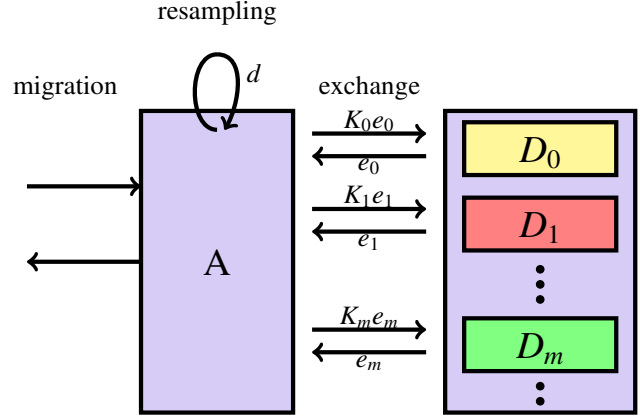
%%%%%%%%%%%%%%%%%%%%%%%%%%%%%%%%%%%%%%%%%%%%%%%%%%%%

In this model we give the seed-bank an internal structure by colouring the dormant individuals with countably many colours  $m\in\N_0$. Active individuals that become dormant are assigned a colour $m$ that is drawn randomly from an infinite sequence of colours labeled by $\N_0$ (see Fig.~\ref{fig:Model2} for an illustration). As will be explained in Section~\ref{ss.comments}, this captures the  \emph{different ways in which individuals can enter into the seed-bank}. In Section \ref{ss.duality} we will show how this internal structure allows for fat tails in the wake-up times of individuals while preserving the Markov property.

For each $i\in\G$ a colony now consists of an active part $A$ and a whole sequence $(D_m)_{m\in\N_0}$ of dormant parts, labeled by their colour $m\in\N_0$. Therefore in this model the \emph{effective geographic space} is given by $\G\times\{A,(D_m)_{m\in\N_0}\}$.

As before, for $i\in\G$, let $x_i(t)$ denote the fraction of individuals in colony $i$ of type $\heartsuit$ that are active 
at time $t$, but now let $y_{i,m}(t)$ denote the fraction of individuals in colony $i$ of type $\heartsuit$ that are dormant with colour $m$ at time $t$. Then the system is described by the process
\begin{equation}
\label{e510}
(Z(t))_{t\geq 0}, \qquad Z(t) = \big(z_i(t)\big)_{i \in \G}, \qquad z_i(t) = (x_i(t),(y_{i,m}(t))_{m \in \N_0}),
\end{equation}
on the state space 
\begin{equation}
\label{ss2}
E = ([0,1]\times[0,1]^{\N_0})^\G.
\end{equation}
Suppose that active individuals exchange with dormant individuals with colour $m$ at rate $e_m \in (0,\infty)$, and let the factor $K_m \in (0,\infty)$ capture the asymmetry between the size of the active population and the $m$-dormant population, i.e., similarly as in \eqref{ratio1}, 
\begin{equation}
\label{ratio2}
K_m = \frac{\text{size $m$-dormant population}}{\text{size active population}},
\qquad m\in\N_0, 
\end{equation}
where $K_m \in (0,\infty)$ is the same for all colonies. Then the process $(Z(t))_{t \geq 0}$ evolves according to the SSDE
\begin{eqnarray}
\label{gh1*}
\d x_i(t) &=& \sum_{j \in \G} a(i,j)\,[x_j (t) - x_i(t)]\,\d t
+ \sqrt{d x_i(t)[1-x_i(t)]}\,\d w_i(t)\\ \nonumber
&&+\,\sum_{m\in\N_0} K_me_m\, [y_{i,m}(t) - x_i(t)]\,\d t,\\[0.2cm]
\label{gh2*}
\d y_{i,m} (t) &=& e_m\,[x_i(t) - y_{i,m}(t)]\,\d t, \qquad m\in\N_0,\ i\in\G,
\end{eqnarray}
where we have to assume that 
\begin{equation}
\label{defchi}
\sum_{m\in\N_0} K_me_m<\infty,
\end{equation}
since otherwise active individuals become dormant instantly. Comparing \eqref{gh1*}--\eqref{gh2*} with the SSDE of model 1 in \eqref{gh1}--\eqref{gh2}, we see that active individuals migrate (the first term in \eqref{gh1*}), resample (the second term in \eqref{gh1*}), but now interact with a whole sequence of dormant populations (the third term in \eqref{gh1*} and the term in \eqref{gh2*}). As initial state $Z(0)=z$ we may again take any $z \in E$.
 
\begin{remark}{\bf [Interpretation of the state space.]}
\label{rm2ss}
{\rm Note that, like in Remark \ref{rm1ss}, the \emph{state space} of the system can also be written as
\begin{equation}
\label{Edef2}
E= [0,1]^\S, \qquad \S=\G\times\{A,(D_m)_{m\in\N_0}\}. 
\end{equation} 
With this interpretation, the process is denoted by
\begin{equation}
\label{e511}
(Z(t))_{t\geq 0}, \qquad Z(t) = \big(z_{u}(t)\big)_{u \in \S},
\end{equation} 
with $z_{u}(t)=x_i(t)$ if $u=(i,A)$ and $z_{u}(t) = y_{i,m}(t)$ if $u=(i,D_m)$ for $m\in\N_0$.} \hfill$\Box$
\end{remark}

\paragraph{Model 3: multi-layer seed-bank with displaced seeds.}

We can extend the mechanism of model 2 by allowing active individuals that become dormant to do so in a randomly chosen colony. This amounts to introducing a sequence of irreducible \emph{displacement kernels} $a_m(\cdot,\cdot)$, $m \in \N_0$, satisfying 
\begin{equation}
\label{pmdef}
a_m(i,j) = a_m(0,j-i) \quad \forall\,i,j \in \G,
\qquad \sum_{i \in \G} a_m(0,i) = 1 \quad \forall\, m \in\N_0,
\end{equation}  
and replacing \eqref{gh1*}--\eqref{gh2*} by
\begin{eqnarray}
\label{gh1**}
\d x_i(t) &=& \sum_{j \in \G} a(i,j)\,
[x_j (t) - x_i(t)]\,\d t
+ \sqrt{d x_i(t)[1-x_i(t)]}\,\d w_i(t)\\ \nonumber
&&+\, \sum_{j\in\G} \sum_{m\in\N_0} K_me_m\,a_m(j,i)\,[y_{j,m}(t) - x_i(t)]\,\d t,\\[0.2cm]
\label{gh2**}
\d y_{i,m} (t) &=& \sum_{j \in \G} e_m\,a_m(i,j)\,[x_j(t) - y_{i,m}(t)]\,\d t, \qquad m\in\N_0,\ i\in\G.
\end{eqnarray}
Here, the third term in \eqref{gh1**} together with the term in \eqref{gh2**} describe the \emph{switch of colony} when individuals exchange between active and dormant. Namely, with probability $a_m(i,j)$ simultaneously an active individual in colony $i$ becomes dormant with colour $m$ in colony $j$ and a randomly chosen dormant individual with colour $m$ in colony $j$ becomes active in colony $i$. The state space $E$ is the same as in \eqref{ss2}. Also \eqref{e510}, \eqref{ratio2}, \eqref{defchi} and \eqref{e511} remain the same.

\paragraph{Two key quantities.}

In models 2 and 3 we must assume that
\begin{equation}
\label{emcond}
\chi = \sum_{m\in\N_0} K_me_m < \infty
\end{equation}
in order to make sure that active individuals do not become dormant instantly. Define
\begin{equation}
\label{rhodef}
\rho=\sum_{m\in\N_0} K_m=\frac{\text{size dormant population}}{\text{size active population}}.
\end{equation}
It turns out that $\rho$ and $\chi$ are two key quantities of our system. In particular, we will see that the long-time behaviour of model 2 and model 3 is different for $\rho<\infty$ and $\rho=\infty$.  

%%%
 
\subsection{Comments}
\label{ss.comments} 

\begin{itemize}

\item[(1)]
Models 1--3 are increasingly more general. Model 2 is the special case of model 3 when $a_m(0,0)=1$ for all $m\in\N_0$, while model 1 is the special case of model 2 when $e_0=e$, $K_0=K$ and $e_m=K_m=0$ for all $m \in \N$. Nonetheless, in what follows we prefer to state our main theorems for each model separately, in order to exhibit the increasing level of complexity. In Appendix~\ref{appA} we explain how \eqref{gh1}--\eqref{gh2}, \eqref{gh1*}--\eqref{gh2*} and \eqref{gh1**}--\eqref{gh2**} arise as the large-colony-size limit of individual-based Fisher-Wright models.  

\item[(2)]
As geographic space $\G$ we allow \emph{any} countable Abelian group endowed with the discrete topology. Key examples are the Euclidean lattice $\G=\Z^d$, $d\in\N$, and the hierarchical lattice $\G=\Omega_N$, $N\in\N$. In this paper we will focus $\G=\Z^d$. The case $\G=\Omega_N$ will be considered in more detail in \cite{GdHOpr3}.  

\item[(3)]
In model 1, each colony has a seed-bank that serves as a \emph{repository} for the genetic information (type $\heartsuit$ or $\diamondsuit$) carried by the individuals. Because the active and the dormant population exchange individuals, the genetic information can be temporarily stored in the seed-bank and thereby be withdrawn from the resampling. We may think of dormant individuals as seeds that drop into the soil and preserve their type until they come to the surface again and grow into a plant. 

In model 2, the seed-bank is a repository for seeds with one of infinitely many colours. \emph{The colours provide us with a tool to model different distributions for the time an individual stays dormant without loosing the Markov property for the evolution of the system}. Tuning the parameters $K_m$ and $e_m$ properly and subsequently forgetting about the colours, we can mimic different distributions for the time an individual stays dormant. This is of biological significance, especially in colonies of bacteria, where individuals stay dormant for random times whose distribution is fat-tailed (see \cite{LJ11}).

In model 3, the seed may even be blown elsewhere. Individuals that displace before becoming dormant are observed in plant-species as well as in bacteria populations (see \cite{LJ11}).

\item[(4)]
In Appendix~\ref{appB} we comment on what happens when the rates to become active or dormant are decoupled, i.e., individuals are no longer subject to exchange but move in and out of the seed-bank independently. This leads to a Moran model where the sizes of the active and the dormant population can fluctuate. We will show that, modulo a change of variables and a short transient period in which the sizes of the active and the dormant population establish equilibrium, this model has the same behaviour as the model with exchange. 

\item[(5)]
In \eqref{gh1}, \eqref{gh1*} and \eqref{gh1**} we may replace the diffusion functions $dg_{\text{FW}}$, $d \in (0,\infty)$, where 
\begin{equation}
\label{FWg}
g_{\text{FW}}(x) = x(1-x), \qquad x \in [0,1],
\end{equation}
is the Fisher-Wright diffusion function, by a \emph{general diffusion function} in the class $\CG$ defined by 
\begin{equation}
\label{gh6}
\CG = \Big\{g\colon\,[0,1] \to [0,\infty)\colon\,g(0) = g(1) = 0, \, g(x) >0 \,\,\forall\,x \in (0,1), 
\, g \mbox{ Lipschitz}\Big\}.
\end{equation}
This class is appropriate because a diffusion with a diffusion function $g \in \CG$ stays confined to $[0,1]$, yet can go everywhere in [0,1] (Breiman~\cite[Chapter 16, Section 7]{B68}). Picking $g\neq g_{\text{FW}}$ amounts to allowing the resampling rate to be \emph{state-dependent}, i.e., the resampling rate in state $x$ equals $g(x)/x(1-x)$, $x \in (0,1)$. An example is the Kimura-Ohta diffusion function $g(x)=[x(1-x)]^2$, $x \in [0,1]$, for which the resampling rate is equal to the genetic diversity of the colony. \emph{In the sequel we allow for general diffusion functions $g \in \CG$ in all three models, unless stated otherwise.} 
\end{itemize}

%%%

\subsection{Well-posedness}
\label{ss.wellpos}

For every law on $E$, with $E$ depending on the choice of model, we want the SSDE for models 1, 2 and 3 to define a Borel Markov process, i.e., the law of the path is a Borel measurable function of the initial state for every starting point in the state space \cite[p.62]{D93}. We use a martingale problem, in the sense of \cite[p.173]{EK86}, to characterize the SSDE. Let 
\begin{equation}
\label{eq401}
\begin{aligned}
\CF= \Big\{&f \in C_b (E, \R) \colon\, f \mbox{ depends on finitely many components}\\[-0.2cm] 
&\mbox{and is twice continuously differentiable in each component}\Big\}.
\end{aligned}
\end{equation}
The generator $G$ of the process acting on $\CF$ reads for model 1 (\eqref{gh1}--\eqref{gh2}),
\begin{equation}
\label{eq402}
\begin{aligned}
G &= \sum_{i \in \G} \Bigg( \Bigg[ \sum_{j \in \G} 
a(i,j)(x_j - x_i)\Bigg] \frac{\partial}{\partial x_i}
+ g(x_i) \frac{\partial^2}{\partial x_i^2}\\
& \qquad\qquad\qquad \qquad + Ke\,(y_i-x_i) \frac{\partial}{\partial x_i}
+ e\,(x_i -y_i) \frac{\partial}{\partial y_i}\Bigg),
\end{aligned}
\end{equation}
for model 2 (\eqref{gh1*}--\eqref{gh2*}),
\begin{equation}
\label{eq403}
\begin{aligned}
G &= \sum_{i \in \G} \Bigg( \Bigg[ \sum_{j \in \G} 
a(i,j)(x_j - x_i)\Bigg] \frac{\partial}{\partial x_i}
+ g(x_i) \frac{\partial^2}{\partial x_i^2}\\
&\qquad\qquad\qquad \qquad + \sum_{m\in\N_0} \Bigg[K_m e_m\,(y_{i,m}-x_i) \frac{\partial}{\partial x_i}
+ e_m\,(x_i -y_{i,m}) \frac{\partial}{\partial y_{i,m}}\Bigg]\Bigg),
\end{aligned}
\end{equation}
while for model 3 (\eqref{gh1**}--\eqref{gh2**}) the last term in the right-hand side of \eqref{eq403} is to be replaced by
\begin{equation}
\label{eq404}
\sum_{i,j\in\G} \sum_{m\in\N_0} \left[K_m e_m\,a_m(j,i)\,(y_{j,m}-x_i)\frac{\partial}{\partial x_i}
+e_m\,a_m(i,j)\,(x_j-y_{i,m})\frac{\partial}{\partial y_{i,m}}\right].
\end{equation} 

\begin{theorem}{{\bf [Well-posedness: models 1--3]}}
\label{T.wellp}
For each of the three models the following holds:\\ 
(a) The SSDE has a unique strong solution in $C([0,\infty),E)$. Its law is the unique solution of the $(G,\CF,\delta_u)$-martingale problem for all $u \in E$.\\ 
(b) The process starting in $u\in E$ is Feller and strong Markov. Consequently, the SSDE defines a unique Borel Markov process starting from any initial law on $E$.
\end{theorem}

%%%%%%%%%%%%%%%%%%%%%%%%%%%%%%%%%%%%%%%%

\subsection{Duality}
\label{ss.duality}

For $g= dg_{\text{FW}}$ the three models have a tractable dual, which will be seen to play a crucial role in the analysis of their long-time behaviour. For $g \neq dg_{\text{FW}}$ the three models do not have a tractable dual. However, we  compare them with models that do and determine their long-time behaviour. In \cite[Sections 2.2 and 3]{BCKW16} it was shown that the non-spatial Fisher-Wright diffusion with seed-bank is dual to the so-called \emph{block-counting process} of a seed-bank coalescent. The latter describes the evolution of the number of partition elements in a partition of $n\in\N$ individuals, sampled from the current population, into subgroups of individuals with the same ancestor (i.e., individuals that are identical by descent). The enriched dual generates the ancestral lineages of the individuals evolving according to a Fisher-Wright diffusion with seed-bank, i.e., generates their full genealogy. The corresponding block-counting process counts the number of ancestral lineages left when traveling backwards in time. In this section we will extend the duality results in \cite{BCKW16} to the spatial setting.   

\paragraph{Model 1.}

%%%%%%%%%%%%%%%%%%%%%%%%%%%%%%%%%%%%%%%%%%%%%
\begin{figure}
	\begin{center}
		\begin{tikzpicture}[scale=0.4]
		\draw [ultra thick,->] (-1.5,1)to [out=90,in=90](1.5,1);
		\draw [ultra thick,black!50!,->] (3.5,1)to [out=90,in=180](5,2.5); 
		\node[right] at (5,2.5){coalescence};
		\node[] at (0,3){migration};
		\node[left] at (-2.5,0) {$Ke$};
		\node[right] at (2.5,0) {$e$};
		\node[right] at (5,-2) {dormant};
		\node[right] at (5,1) {active};
		\draw[ultra thick, ->] (-2.5,1)--(-2.5,-2);
		\draw[ultra thick, <-] (2.4,1)--(2.4,-2);
		\draw[ultra thick] (-5,1)--(5,1);
		\draw[dashed, ultra thick](-5,-2)--(5,-2);
		\end{tikzpicture}
	\end{center}
	\caption{\small Transition scheme for an ancestral lineage in the dual, which moves according 
		to the transition kernel $b(\cdot,\cdot)$ in \eqref{mrw}. Two active ancestral lineages that are 
		at the same colony coalesce at rate $d$.}
	\label{fig:dualrw}
\end{figure}
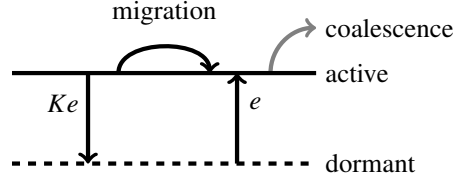
%%%%%%%%%%%%%%%%%%%%%%%%%%%%%%%%%%%%%%%%%%%%%%%%%

%%%%%%%%%%%%% Figure %%%%%%%%%%%%%%%%%%
\begin{figure}[htbp]
	\begin{center}
		\vspace{.2cm}
		\setlength{\unitlength}{.5cm}
		%%%
		\begin{tikzpicture}[scale=0.4]
		\draw [fill=blue!10!] (0,-0.5) rectangle (1,4.5);
		\draw [fill=blue!10!] (2,-0.5) rectangle (3,4.5);
		\draw [fill=blue!10!] (4,-0.5) rectangle (5,4.5);
		\draw [fill=blue!10!] (6,-0.5) rectangle (7,4.5);
		\draw [fill=blue!10!] (8,-0.5) rectangle (9,4.5);
		\draw [fill=black, ultra thick] (10,2) circle [radius=0.05] ;
		\draw [fill=black, ultra thick] (11,2) circle [radius=0.05] ;
		\draw [fill=black, ultra thick] (12,2) circle [radius=0.05] ;
		\draw [fill=blue!10!] (13,-0.5) rectangle (14,4.5);
		\draw [ultra thick, ->] (-1,-.5) -- (-1,5);
		\node[left] at (-1.5,2.5) {$t$};
		\draw [fill=blue] (.25,-.5) circle [radius=0.1] ;
		\draw [fill=blue] (.75,-.5) circle [radius=0.1] ;
		\draw[blue,ultra thick] (.25,-.5)--(.25,2);	
		\draw[blue,ultra thick] (.75,-.5)--(.75,2)--(.25,2)--(.25,4.5);
		\draw [fill=blue] (2.25,-.5) circle [radius=0.1] ;
		\draw[blue,ultra thick] (2.25,-.5)--(2.25,1.5);
		\draw[red,ultra thick] (2.25,1.5)--(2.25,3.5);
		\draw[blue,ultra thick] (2.25,3.5)--(2.25,4)--(.75,4)--(.75, 4.5);
		\draw [fill=red] (4.25,-.5) circle [radius=0.1] ;
		\draw[red,ultra thick] (4.25,-.5)--(4.25,4.5);	
		\draw [fill=blue] (6.25,-.5) circle [radius=0.1] ;
		\draw[blue,ultra thick] (6.25,-.5)--(6.25,4.5);
		\draw [fill=red] (6.75,-.5) circle [radius=0.1] ;
		\draw[red,ultra thick] (6.75,-.5)--(6.75,1);
		\draw[blue,ultra thick] (6.75,1)--(6.75,3)--(8.75,3)--(8.75,4.5);
		\draw [fill=blue] (13.25,-.5) circle [radius=0.1] ;
		\draw[blue,ultra thick] (13.25,-.5)--(13.25,4.5);	
		\end{tikzpicture}
		%%%
\vspace{0.2cm}
\caption{\small Picture of the evolution of lineages in the spatial coalescent. The purple blocks depict the colonies, the blue lines the active lineages, and the red lines the dormant lineages. Blue lineages can migrate and become dormant, (i.e., become red lineages). Two blue lineages can coalesce when they are at the same colony. Red dormant lineages first have to become active (blue) before they can coalesce with other blue and active lineages or migrate. Note that the dual runs \emph{backwards in time}. The collection of all lineages determines the genealogy of the system.}
\label{fig:Duality}
	\end{center}
\end{figure}
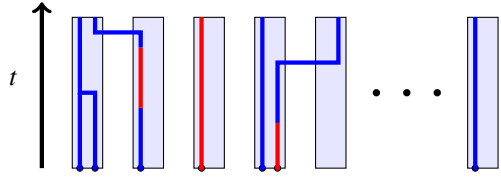
%%%%%%%%%%%%%%%%%%%%%%%%%%%%%%%%%%%%%

Recall that for model 1, $\S=\G\times\{A,D\}$ is the effective geographic space. For $n\in\N$ the state space of the $n$-spatial seed-bank coalescent is the set of partitions of $\{1,\ldots,n\}$, where the partition elements are marked with a position vector giving their location. A state is written as $\pi$, where
\begin{equation}
\label{e402}
\begin{aligned}
&\pi = ((\pi_1,\eta_1), \ldots, (\pi_{\bar{n}},\eta_{\bar{n}})), \quad \bar{n} = |\pi|,\\
&\pi_\ell\subset \{1,\ldots, n\},  \text{ and }\{\pi_1,\cdots \pi_{\bar{n}}\}\ \mbox{  is a partition of } \{1,\ldots, n\},\\ 
&\eta_\ell \in \S, \quad \ell \in \{1,\ldots, \bar{n}\}, \quad 1\leq \bar{n}\leq n.
\end{aligned}
\end{equation}
A marked partition element $(\pi_\ell,\eta_\ell)$ is called active if $\eta_\ell=(j,A)$ and called dormant if $\eta_\ell=(j,D)$ for some $j\in\G$. The $n$-spatial seed-bank coalescent is denoted by
\begin{equation}
\label{e403}
(\CC^{(n)}(t))_{t \geqslant 0},
\end{equation}
and starts from
\begin{equation}
\label{e403alt}
\CC^{(n)}(0) = \pi(0), \qquad  \pi(0) = \{(\{1\},\eta_{\ell_1}),\ldots,(\{n\},\eta_{\ell_n})\}, \qquad
\eta_{\ell_1},\ldots,\eta_{\ell_n} \in \S.
\end{equation}

The $n$-spatial seed-bank coalescent is a Markov process that evolves according to the following two rules.
\begin{enumerate}
\item 
Each partition element moves independently of all other partition elements according the kernel 
\begin{equation}
\label{mrw}
b^{(1)}((i,R_i), (j, R_j)) = \left\{ \begin{array}{ll}
a (i, j), &\text{ if } R_i = R_j =A,\\
Ke, &\text{ if } i = j,\ R_i=A,\ R_j = D, \\
e, &\text{ if } i = j,\ R_i=D,\ R_j = A,\\
0, &\mbox{otherwise},
\end{array}
\right.
\end{equation}
where $a(\cdot,\cdot)$ is the migration kernel defined in \eqref{gh10}, $K$ is the relative size of the dormant population defined in \eqref{ratio1}, and $e$ is the rate of exchange between the active and the dormant population shown in \eqref{gh1}--\eqref{gh2}. Therefore an active partition element migrates according to the transition kernel $a(\cdot,\cdot)$ and becomes dormant at rate $Ke$, while a dormant partition element can only become active and does so at rate $e$. In \eqref{mrw}, the notation $b^{(1)}$ marks that the kernel refers to model 1. Later we will use the notation $b^{(2)}$ for model 2 and $b^{(3)}$ for model 3.  
\item 
Independently of all other partition elements, two partition elements that are at the same colony and are both active coalesce with rate $d$, i.e., the two partition elements merge into one partition element.
\end{enumerate}

\noindent
The spatial seed-bank coalescent $(\CC(t))_{t\geq 0}$ is defined as the projective limit of the $n$-spatial seed-bank coalescents $(\CC^{(n)}(t))_{t\geq 0}$ as $n\to\infty$. This object is well-defined by Kolmogorov's extension theorem (see \cite[Section 3]{BCKW16}).

For $n\in\N$ we define the block-counting process $(L(t))_{t\geq 0}$ corresponding to the $n$-spatial seed-bank coalescent as the process that counts at each site $(i,R_i)\in\G\times\{A,D\}$ the number of partition elements of $\CC^{(n)}(t)$,  i.e., 
\begin{equation}
\label{blctpr}
\begin{aligned}
&L(t)=\big(L_{(i,A)}(t),L_{(i,D)}(t)\big)_{i\in\G},\\
&L_{(i,A)}(t)=L_{(i,A)}(\CC^{(n)}(t))
=\sum_{\ell=1}^{\bar{n}}1_{\{\eta_\ell(t)=(i,A)\}},\qquad L_{(i,D)}(t)=L_{(i,D)}(\CC^{(n)}(t))
=\sum_{\ell=1}^{\bar{n}}1_{\{\eta_\ell(t)=(i,D)\}}.
\end{aligned}
\end{equation}
Therefore $(L(t))_{t\geq 0}$ has state space $E'=(\N_0\times\N_0)^\G$. We denote the elements of $E^\prime$ by sequences $(m_i,n_i)_{i\in\G}$, and define $\delta_{(j,R_j)}\in E^\prime$ to be the element of $E^\prime$ that is $0$ at all sites $(i,R_i)\in\G\times\{A,D\}\backslash (j,R_j)$, but $1$ at the site $(j,R_j)$. From the evolution of $\CC^{(n)}(t)$ described below \eqref{e403} we see that the block-counting process has the following transition kernel:
\begin{eqnarray}
\label{blockproc}
(m_i,n_i)_{i\in\G} \rightarrow
\begin{cases}
({m}_i,{n}_i)_{i\in\G}-\delta_{(j,A)}+\delta_{(k,A)},\ 
&\text{at rate } m_j a(j,k) \text{ for } j,k\in\G,\\
({m}_i,{n}_i)_{i\in\G}-\delta_{(j,A)},\ 
&\text{at rate } d{m_j \choose 2} \text{ for } j\in\G,\\
({m}_i,{n}_i)_{i\in\G}-\delta_{(j,A)}+\delta_{(j,D)},\ 
&\text{at rate } m_jKe \text{ for } j\in\G, \\
({m}_i,{n}_i)_{i\in\G}+\delta_{(j,A)}-\delta_{(j,D)},\ 
&\text{at rate } n_je \text{ for } j\in\G.\\
\end{cases}
\end{eqnarray}

The process $(Z(t))_{t\geq 0}$ defined in \eqref{gh1}--\eqref{gh2} is dual to the block-counting process $(L(t))_{t\geq 0}$. The duality function $H\colon\,E \times E^\prime\to \R$ is defined by
\begin{equation}
\label{hpoldef}
H\Big((x_i,y_i)_{i\in\G},(m_i,n_i)_{i\in\G}\Big) = \prod_{i\in\G} x_i^{m_i}y_i^{n_i}.
\end{equation}
The {\em duality relation} reads as follows.

\begin{theorem}{{\bf [Duality relation: model 1]}}. Let $H$ be defined as in \eqref{hpoldef}. Then for all $(x_i,y_i)_{i\in\G}\in E$ and $(m_i,n_i)_{i\in\G}\in E^\prime$,
\label{T.dual1}
\begin{equation}
\label{e401}
\E_{(x_i,y_i)_{i\in\G}}\Big[H\Big((x_i(t),y_i(t))_{i\in\G},(m_i,n_i)_{i\in\G}\Big)\Big]
=\E_{(m_i,n_i)_{i\in\G}}\Big[H\Big((x_i,y_i)_{i\in\G},(L_{(i,A)}(t),L_{(i,D)}(t))_{i\in\G}\Big)\Big]
\end{equation}
with $\E$ the generic symbol for expectation (on the left over the original process, on the right over the dual process).
\end{theorem}

\noindent
Since the duality function $H$ gives all the mixed moments of $(Z(t))_{t\geq 0}$, the duality relation in Theorem~\ref{T.dual1} is called a \emph{moment dual}.

\begin{remark}{\bf [Duality relation in terms of the effective geographic space]}
\label{dualeff}
{\rm Interpreting $(Z(t))_{t\geq 0}$ as a process on the effective geographic space $\S$, recall Remark \ref{rm1ss}, we can rewrite the duality relation. Let the block-counting process $(L(t))_{t\geq 0}=(L(\CC(t))_{t \geq 0}$ count at each site $u\in\S$ the number of partition elements of $\CC(t)$, i.e.,
\begin{equation}
\begin{aligned}
L(t)&=(L_u(t))_{u\in\S},\\
L_u(t)&=L_u(\CC(t)) =\sum_{\ell=1}^{\bar{n}}1_{\{\eta_\ell(t)=u\}},
\end{aligned}
\end{equation}
and rewrite the duality function $H$ in \eqref{hpoldef} as
\begin{equation}
H((z_u,l_u)_{u\in\S})=\prod_{u\in\S}z_u^{l_u}.
\end{equation}
Then, for $z\in\E$ and $l\in\E^\prime$, the {\em duality relation} reads 
\begin{equation}
\label{e401b}
\E\big[H(z_u(t),l_u)\big]=\E\big[H(z_u,L_u(t))\big]. 
\end{equation}
Interpreting the duality relation in terms of the effective geographic space $\S$, we see that each ancestral lineage in the dual is a Markov chain that moves according to the transition kernel $b^{(1)}(\cdot,\cdot)$. Interpreting the duality relation in terms of the geographic space $\G$, we see that an ancestral lineage is a random walk moving on $\G$, with internal states $A$ and $D$. Both interpretations turn out to be useful in analysing the long-time behaviour of the system.}
\hfill$\Box$
\end{remark}

\begin{remark}{\bf [Wake-up times]}
\label{wakeup}
{\rm Define (see Fig.~\ref{fig:dualrw})
\begin{equation}
\label{defwake}
\begin{aligned}
\sigma&= \text{ typical time spent by an ancestral lineage in state } A \text{ before switching to state } D,\\
\tau&= \text{ typical time spent by an ancestral lineage in state } D \text{ before switching to state } A.
\end{aligned}
\end{equation}
(Here, the word typical refers to what happens to an ancestral lineage each time it switches state at some geographic location. For a more precise definition we refer to Section~\ref{ss.clustcase} and Fig.~\ref{fig:periods}.) It follows from \eqref{mrw} that
\begin{equation}
\begin{aligned}
\P(\sigma>t) &= \e^{-Ke t},\\
\P(\tau>t) &= \e^{-e t}.
\end{aligned}
\end{equation}
An ancestral lineage in the dual of the spatial seed-bank process behaves as an ancestral lineage in the dual of a spatial Fisher-Wright diffusion without seed-bank (see e.g.\ \cite{FG96}), but becomes dormant every once in a while. On the long run we expect an ancestral lineage to be active only a fraction $\frac{1}{1+K}$ of the time. We will see in Section \ref{s.model1} that the effect of the seed-bank on the long-time behaviour of the ancestral lineages in the dual is a slow down by a factor $\frac{1}{1+K}$ compared to the long-time behaviour of the ancestral lineages in the dual of interacting Fisher-Wright diffusions without seed-bank.
} \hfill$\Box$
\end{remark}

\paragraph{Model 2.} 

The dual for model 2 arises naturally from the dual for model 1 by adding internal states to the seed-bank and adapting the rates of becoming active and dormant accordingly. Recall that for model 2 the effective geographic space is $\S=\G \times \{A, (D_m)_{m\in\N_0}\}$. Migration and coalescence are as before, but at every colony switches between an active copy $A$ and a dormant copy $D_m$ now occur at rates $e_m$, respectively, $K_m\,e_m$. The spatial coalescent $(\CC(t))_{t\geq 0}$ in \eqref{e403} starts from an initial configuration like \eqref{e403alt} and evolves according to the same two rules, but the transition kernel $b(\cdot,\cdot)$ must be replaced by
\begin{equation}
\label{mrw2}
b^{(2)}((i,R_i), (j, R_j)) = \left\{ \begin{array}{ll}
a (i, j), & R_i = R_j  =A,\\
K_me_m, & i = j,\ R_i=A,\ R_j = D_m, \text{ for } m\in\N_0,\\
e_m, & i = j,\ R_i=D_m,\ R_j = A, \text{ for } m\in\N_0,\\
0, &\mbox{otherwise}.
\end{array}
\right.
\end{equation} 

The corresponding block-counting process becomes
\begin{equation}
\begin{aligned}
&L(t)
=\left(L_{(i,A)}(t),\left(L_{(i,D_m)}(t)\right)_{m\in\N_0}\right)_{i\in\G},\\
&L_{(i,A)}(t)
= L_{(i,A)}(\CC(t)) =\sum_{\ell=1}^{\bar{n}}1_{\{\eta_\ell(t)=(i,A)\}},\qquad L_{(i,D_m)}(t)=L_{(i,D_m)}(\CC(t))
= \sum_{\ell=1}^{\bar{n}}1_{\{\eta_\ell(t)=(i,D_m)\}},\  m\in\N_0.
\end{aligned}
\end{equation}
The state space is now given by $E^\prime=\left(\N_0\times\N_0^{\N_0}\right)^\G$, and the transition kernel becomes
\begin{eqnarray}
\label{blockproc2}
(m_i,(n_{i,m})_{m\in\N_0})_{i\in\G} \rightarrow
\begin{cases}
({m}_i,({n}_{i,m})_{m\in\N_0})_{i\in\G}-\delta_{(j,A)}+\delta_{(k,A)},\ 
&\text{at rate } m_j a(j,k) \text{ for } j,k\in\G,\\
(m_i,(n_{i,m})_{m\in\N_0})_{i\in\G}-\delta_{(j,A)},\ 
&\text{at rate } d{m_j \choose 2} \text{ for } j\in\G,\\
(m_i,(n_{i,m})_{m\in\N_0})_{i\in\G}-\delta_{(j,A)}+\delta_{(j,D_m)},\ 
&\text{at rate } m_jK_me_m \text{ for } j\in\G, \\
(m_i,(n_{i,m})_{m\in\N_0})_{i\in\G}+\delta_{(j,A)}-\delta_{(j,D_m)},\ 
&\text{at rate } n_{j,m}e_m \text{ for } j\in\G.\\
\end{cases}
\end{eqnarray}

The duality function $H\colon\,E \times E^\prime \to\R$ is defined by 
\begin{eqnarray}
H\Big((x_i,y_{i,m})_{i\in\G,m\in\N_0},(m_i,n_{i,m})_{i\in\G,m\in\N_0}\Big)
=\prod_{i\in\G}\prod_{m\in\N_0}x_i^{m_i}y_{i,m}^{n_{i,m}}.
\label{dualmod2alt}
\end{eqnarray}

\begin{theorem}{{\bf [Duality relation: model 2]}} 
\label{T.dual2} 
For $(x_i,y_{i,m})_{i\in\G,m\in\N_0}\in E$ and $(m_i, n_{i,m})_{i\in\G,m\in\N_0}\in E^\prime$,
\begin{equation}
\label{e401alt}
\begin{aligned}
&\E_{(x_i,y_{i,m})_{i\in\G,m\in\N_0}}\Big[H\Big((x_i(t),y_{i,m}(t))_{i\in\G,m\in\N_0},(m_i,n_{i,m})_{i\in\G,m\in\N_0}\Big)\Big]\\
&\qquad = \E_{(m_i,n_{i,m})_{i\in\G,m\in\N_0}}\Big[H\Big((x_i,y_{i,m})_{i\in\G,m\in\N_0},(L_{(i,A)}(t),L_{(i,D_m)}(t))_{i\in\G,m\in\N_0}\Big)\Big].
\end{aligned}
\end{equation}
\end{theorem}

\noindent
By rewriting the block-counting process as in Remark~\ref{dualeff}, the duality function can be rewritten as
\begin{equation}
H((z_u,l_u)_{u\in\S})=\prod_{u\in\S}z_u^{l_u}
\end{equation}
and the {\em duality relation} reads 
\begin{equation}
\label{e401b2}
\E\Big[H\Big((z_u(t))_{u\in\S},(l_u)_{u\in\S}\Big)\Big]
= \E\Big[H\Big((z_u)_{u\in\S},(L_u(t))_{u\in\S})\Big)\Big].
\end{equation}
 
\begin{remark}{\bf[Fat-tailed wake-up times.]}
\label{fattails}
{\rm Recall the definition of $\chi$ in \eqref{emcond} and the definition of $\rho$ in \eqref{rhodef}. Define
\begin{equation}
\label{defwakealt}
\begin{aligned}
\sigma&= \text{ typical time spent by an ancestral lineage in the active state } A\\
&\quad\text{ before switching to a dormant state }\cup_{m\in\N_0} D_m,\\
\tau&= \text{ typical time spent by an ancestral lineage in the dormant state } \cup_{m\in\N_0} D_m\\
&\quad\text{ before switching to the active state }A.
\end{aligned}
\end{equation}
Note that $\tau$ does not look at the colour of the dormant state. It follows from \eqref{mrw2} that
\begin{equation}
\label{longtimeprob}
\begin{aligned}
\P(\sigma>t) &=\e^{-\chi t},\\
\P(\tau>t) &= \sum_{m\in\N_0}\frac{K_me_m}{\chi}\e^{-e_m t},
\end{aligned}
\end{equation}
independently of the colony $i\in\G$. Hence 
\begin{equation}
\label{exptau}
\E\left[\tau\right]=\frac{\rho}{\chi}.
\end{equation}
If $\rho<\infty$, then we invoke the seed-bank colours and use the balance equations for recurrent Markov chains to see that each ancestral lineage in the dual in the long run spends a fraction $\frac{\rho}{1+\rho}$ of the time in the dormant state. Like in model 1, an ancestral lineage in the dual behaves like an ancestral lineage in the dual of interacting Fisher-Wright diffusions, but is slowed down by a factor $\frac{\rho}{1+\rho}$. However, if $\rho=\infty$, then \eqref{mrw2} together with \eqref{exptau} imply that each ancestral lineage in the dual behaves like a null-recurrent Markov chain on $\{A,(D_m)_{m\in\N_0}\}$, and consequently the probability to be active tends to 0 as $t\to\infty$. Therefore we may expect that the long-time behaviour of the system is affected by the seed-bank. In particular, choosing
\begin{equation}
\label{assalta}
\begin{aligned}
&K_m \sim A\,m^{-\alpha}, \quad e_m \sim B\,m^{-\beta}, \quad m \to \infty,\\
&A,B \in (0,\infty), \quad \alpha,\beta \in \R\colon\,\alpha \leq 1 < \alpha +\beta,
\end{aligned}
\end{equation}
we see that \eqref{longtimeprob} implies
\begin{equation}\label{ft}
\P(\tau>t)\sim C t^{-\gamma}, \qquad t \to\infty,
\end{equation}
with $\gamma = \frac{\alpha+\beta-1}{\beta}$ and $C = \frac{A}{\chi\beta}\,B^{1-\gamma}\,\Gamma(\gamma)$, where $\Gamma$ is the Gamma-function. The conditions on $\alpha,\beta$ guarantee that $\rho=\infty$, $\chi<\infty$ (recall \eqref{emcond} and \eqref{rhodef}). Examples are: $\alpha=0$, $\beta >1$ and $\alpha \in (0,1)$, $\beta > 1-\alpha$. Thus, for $\rho=\infty$ we can model individuals with a fat-tailed wake-up time simply by not taking their colours into account.  \emph{The internal structure of the seed-bank captured by the colours allows us to model fat-tailed wake-up times without loosing the Markov property for the evolution.} 
} \hfill$\Box$	
\end{remark} 
 
\paragraph{Model 3.}

The effective geographic space is again $\S=\G \times \{A, (D_m)_{m\in\N_0}\}$. On top of migration and coalescence, each switch from $A$ to $D_m$ and vice versa is accompanied by a displacement according to the displacement kernel $a_m(\cdot,\cdot)$ defined in \eqref{pmdef}.    Therefore each lineage in the dual evolves according to 
\begin{equation}
\label{mrw3}
b^{(3)}((i,R_i), (j, R_j)) = \left\{ \begin{array}{ll}
a (i, j), & R_i = R_j  =A,\\
K_me_ma_m(j,i), &\ R_i=A,\ R_j = D_m, \text{ for } m\in\N_0,\\
e_ma_m(i,j), &\ R_i=D_m,\ R_j = A, \text{ for } m\in\N_0.\\
\end{array}
\right.
\end{equation}
Again, when two ancestral lineages are active at the same site they coalesce at rate $1$ and the corresponding block-counting process evolves according to the transition kernel
\begin{equation}
\label{blockproc3}
(m_i,(n_{i,m})_{m\in\N_0})_{i\in\G} \rightarrow
\begin{cases}
({m}_i,({n}_{i,m})_{m\in\N_0})_{i\in\G}-\delta_{(j,A)}+\delta_{(k,A)},\ 
&\text{at rate } m_j a(j,k) \text{ for } j,k\in\G,\\
(m_i,(n_{i,m})_{m\in\N_0})_{i\in\G}-\delta_{(j,A)},\ 
&\text{at rate } d{m_j \choose 2} \text{ for } j\in\G,\\
(m_i,(n_{i,m})_{m\in\N_0})_{i\in\G}-\delta_{(j,A)}+\delta_{(k,D_m)},\ 
&\text{at rate } m_jK_me_ma_m(k,j) \text{ for } j\in\G, \\
(m_i,(n_{i,m})_{m\in\N_0})_{i\in\G}+\delta_{(k,A)}-\delta_{(j,D_m)},\ 
&\text{at rate } n_{j,m}e_ma_m(j,k) \text{ for } j\in\G.\\
\end{cases}
\end{equation}

\begin{theorem}{{\bf [Duality relation: model 3]}}
\label{T.dual3}
The same duality relation holds as in \eqref{e401alt}, where now the dual dynamics includes not only the exchange between active and dormant but also the accompanying displacement in space.  
\end{theorem}

%%%
\subsection{Dichotomy criterion }
\label{ss.dichcrit}

For $g=dg_{\mathrm{FW}}$ the duality relations in Theorems~\ref{T.dual1}, \ref{T.dual2} and \ref{T.dual3} provide us with the following criterion to characterise the long-term behaviour. If, in the limit as $t\to\infty$, locally only one type survives in the population, then we say that the system exhibits \emph{clustering}. If, in the limit as $t\to\infty$, locally both types survive in the population, then we say that the system exhibits \emph{coexistence}. For model 1 the criterion reads as follows.

\begin{theorem}{\bf [Dichotomy criterion: model 1]}
\label{T.dichcrit.model1} 
Suppose that $\mu(0)$ is invariant and ergodic under translations. Let $d\in(0,\infty)$. Then the system with $g=dg_{\mathrm{FW}}$ clusters if and only if in the dual two partition elements coalesce with probability $1$. 
\end{theorem}

The idea behind Theorem \ref{T.dichcrit.model1} is as follows. If in the dual two partition elements coalesce with probability 1, then a random sample of $n$ individuals drawn from the current population has a common ancestor some finite time backwards in time. Since individuals inherit their type from their parent individuals, this means that all $n$ individuals have the same type. A formal proof will be given in Section \ref{ss.dichotomy}.

For model 2--3 we need an extra assumption on $\mu(0)$ when $\rho=\infty$.  

\begin{definition}{\bf [Colour regular initial measures]}
\label{D.regular}
{\rm We say that $\mu(0)$ is \emph{colour regular} when
\begin{equation}
\label{covcond1}
\lim_{N\to\infty} \E_{\mu(0)}[y_{0,N}] \quad \text{ exists},
\end{equation}	
i.e., $\mu(0)$ has asymptotically converging colour means.} \hfill$\Box$
\end{definition}

\noindent
Thus, colour regularity is a condition on the deep seed-banks (where deep means $m\to\infty$). This condition is needed because as time proceeds lineages starting from deeper and deeper seed-banks become active for the first time, and bring new types into the active population. Without control on the initial states of the deep seed-banks, there may be no convergence to equilibrium.

\begin{theorem}{{\bf [Dichotomy criterion: models 2--3]}}
\label{T.dichcrit.model2} 
The same as in Theorem~\ref{T.dichcrit.model1} is true for $\rho<\infty$, but for $\rho=\infty$ additionally requires that $\mu(0)$ is colour regular.
\end{theorem}

\begin{remark}{\rm\textbf{[Clustering criterion general $g\in\CG$]}
\label{rmgeng}
In Section \ref{s.scaling} we will see that the dichotomy criterion in Theorems \ref{T.dichcrit.model1} and \ref{T.dichcrit.model2} for $g=dg_{\mathrm{FW}}$ does not depend on $d$, the rate of resampling. We will use \emph{duality comparison arguments} to carry over the dichotomy criterion in Theorems \ref{T.dichcrit.model1} and \ref{T.dichcrit.model2} to $g\in\CG$. We will see later that for all three models the system with $g$ exhibits clustering if and only if the system with $g_{\mathrm{FW}}$ exhibits clustering.}\hfill$\Box$
\end{remark}

\begin{remark}{\bf [Liggett conditions]}
\label{ligcon}
{\rm We will see in Section \ref{ss.cosalt} that, for model 2 with $\rho=\infty$, if an initial measure $\mu$ is invariant and ergodic under translations and is colour regular, then the Markov chain evolving according to $b^{(2)}(\cdot,\cdot)$ satisfies the following two conditions:
\begin{enumerate}
\item[{\rm (1)}] 
$\lim_{t\to \infty} \sum_{(k,R_k)\in \G\times\{A,(D_m)_{m\in\N_0}\}} b^{(2)}_t\big((i,R_i),(k,R_k)\big)\, 
\E_{\mu}[z_{(k,R_k)}] =\theta$,
\item[{\rm (2)}] 
$\lim_{t\to \infty}\sum_{(k,R_k),(l,R_l)\in \G\times\{A,(D_m)_{m\in\N_0}\}} 
b^{(2)}_t\big((i,R_i),(k,R_k)\big)\,b^{(2)}_t\big((j,R_j),(l,R_l)\big)\,\E_\mu[z_{(k,R_k)}z_{(l,R_l)}]=\theta^2$.
\end{enumerate}
These are precisely the conditions in \cite[Chapter V.1]{Lig85} necessary to determine the dichotomy in the long-time behaviour of the voter model. We show that (1) and (2) imply convergence to a unique equilibrium that is invariant and ergodic under translations. It is difficult to identify exactly which initial measures $\mu$ satisfy (1) and (2). This is the reason why we work with sufficient conditions and need the notion of colour regularity. 

For model 2 with $\rho<\infty$, conditions (1) and (2) are satisfied when $\mu(0)$ is invariant and ergodic under translations, and colour regularity is not needed. The same holds for model 1, once the state space is replaced by $\G\times\{A,D\}$ and $b^{(2)}(\cdot,\cdot)$ is replaced by $b^{(1)}(\cdot,\cdot)$. Also for model 3 conditions (1) and (2) hold after replacing $b^{(2)}(\cdot,\cdot)$ by $b^{(3)}(\cdot,\cdot)$. If $\rho=\infty$ in model 3 we need to assume colour regularity, if $\rho<\infty$, this is not needed.  } \hfill$\Box$
\end{remark}

%%%%%%%%%%%% SECTION 3 %%%%%%%%%%%%%%%%%%%%%%%%%

\section{Long-time behaviour}
\label{s.scaling}

In this section we study the long-time behaviour of models 1--3. In  Sections~\ref{ss.scal1}--\ref{ss.scal3} we prove convergence to a unique equilibrium measure, establish the dichotomy between clustering and coexistence, and identify which of the two occurs in terms of the migration kernel and the rates governing the exchange with the seed-bank (Theorems~\ref{T.ltb1}--\ref{T.ltb3}). 

Throughout the sequel, $g$ is a general diffusion function from the class $\CG$ defined in \eqref{gh6}. Special cases are the multiples of the standard Fisher-Wright diffusion function: $g=dg_{\mathrm{FW}}$, $d \in (0,\infty)$, with $g_{\mathrm{FW}}(x)=x(1-x)$, $x \in [0,1]$. We use the following notation (with $\CP(E)$ denotes the set of probability measures on $E$): 
\begin{equation}
\begin{aligned}
\CT &= \big\{\mu\in\CP(E)\colon\, \mu \text{ is invariant under translations in } \G\big\},\\
\CT^{\mathrm{erg}} &= \big\{\mu\in\CT\colon\, \mu \text{ is ergodic under translations in } \G\big\},\\
\CI &= \big\{\mu\in\CT\colon\, \mu \text{ is invariant under the evolution}\big\}. 
\end{aligned}
\end{equation}

%%%%%%%%%%

\subsection{Long-time behaviour of Model 1}
\label{ss.scal1}

Let $a(\cdot,\cdot)$ be as in \eqref{gh10}. Define the \emph{symmetrized} migration kernel
\begin{equation}
\label{gh11}
\hat a(i,j) = \tfrac{1}{2} [a (i,j) + a(j,i)], \qquad i,j \in \G, 
\end{equation}
which describes the difference of two independent copies of the migration each driven by $a(\cdot,\cdot)$. Let $\hat{a}_t(0,0)$ denote the time-$t$ transition kernel of the random walk with migration kernel $\hat{a}(\cdot,\cdot)$, and suppose that 
\begin{equation}
\label{ass2}
t \mapsto \hat{a}_t(0,0) \text{ is regularly varying at infinity}.
\end{equation} 
(Examples can be found in \cite[Chapter 3]{H95}.)  Define
\begin{equation}
\label{Ihatpdef}
I_{\hat{a}} = \int_1^\infty \d t\,\hat{a}_t(0,0).
\end{equation}
Note that $I_{\hat{a}}=\infty$ if and only if $\hat{a}(\cdot,\cdot)$ is recurrent (see e.g.\ \cite[Chapter 1]{Sp64}). Define
\begin{equation}
\label{thetadef}
\theta = \E_{\mu(0)}\left[ \frac{x_0 + Ky_0}{1+K}\right].
\end{equation}
If $\mu(0)$ is invariant and ergodic under translations, then $\theta$ is the initial density of $\heartsuit$ in the population. 

From the SSDE in \eqref{gh1}--\eqref{gh2} we see that
\begin{equation}\label{mart1}
\left(\frac{x_0(t)+Ky_0(t)}{1+K}\right)_{t\geq 0}
\end{equation}
is a martingale. In particular,
\begin{equation}
\label{e1203}
\theta = \E_{\mu(t)}\left[ \frac{x_0 + Ky_0}{1+K}\right] 
\qquad \forall\,t \geq 0.
\end{equation}
For $\theta\in[0,1]$, we define
\begin{equation}\label{deftergT}
\CT^{\mathrm{erg}}_\theta=\left\{\mu\in\CT^{\mathrm{erg}}\colon\, 
\E_{\mu(0)}\left[ \frac{x_0 + Ky_0}{1+K}\right]=\theta \right\}.
\end{equation}
Write $\mu(t)$ to denote the law of $(Z(t))_{t\geq 0}$, defined in \eqref{e409}. Recall that associated means that increasing functions of the configuration are positively correlated, i.e., if $f\colon\,E\to\R$ and $g\colon\,E\to\R$ depend on only finitely many coordinates and are coordinate-wise increasing, then 
\begin{equation}
\E_{\nu_\theta}[f(x)g(x)]\geq \E_{\nu_\theta}[f(x)]\,\E_{\nu_\theta}[g(x)].
\end{equation}

\begin{theorem}{{\bf [Long-time behaviour: model 1]}}
\label{T.ltb1}
Suppose that $\mu(0) \in \CT_\theta^{\mathrm{erg}}$.
\begin{itemize}
\item[{\rm (a)}] (Coexistence regime)
If $\hat{a}(\cdot,\cdot)$ is transient, i.e., $I_{\hat{a}}<\infty$, then
\begin{equation}
\label{gh13}
\lim_{t\to\infty} \mu(t) = \nu_\theta,
\end{equation}
where
\begin{eqnarray}
\label{e1172}
&&\nu_\theta \mbox{ is an equilibrium measure for the process on } E,\\
\label{gh14}
&&\nu_\theta \mbox{ is invariant, ergodic and mixing under translations},\\
\label{gh14*}
&&\nu_\theta \mbox{ is associated},\\
\label{gh15}
&&\E_{\nu_\theta}[x_0] = \E_{\nu_\theta}[y_0] = \theta,
\end{eqnarray}
with $\E_{\nu_\theta}$ denoting expectation over $\nu_\theta$.
\item[{\rm (b)}] (Clustering regime)
If $\hat{a}(\cdot,\cdot)$ is recurrent, i.e., $I_{\hat{a}}=\infty$, then
\begin{equation}
\label{gh12}
\lim_{t\to\infty} \mu(t)
= \theta\, [\delta_{(1,1)}]^{\otimes \G} + (1-\theta)\, [\delta_{(0,0)}]^{\otimes \G}.
\end{equation}
\end{itemize}
\end{theorem}

The results in \eqref{gh13}--\eqref{gh12} say that the system converges to an equilibrium whose density of type $\heartsuit$ equals $\theta$ in \eqref{thetadef}, a parameter that is controlled by the initial state $\mu(0)$ and the asymmetry parameter $K$. The equilibrium can be either locally mono-type or locally multi-type, depending on whether the symmetrised migration kernel is recurrent or transient. If the equilibrium is mono-type, then the system grows large mono-type clusters (= clustering). If the equilibrium is multi-type, then the system allows $\heartsuit$ and $\diamondsuit$ to mix (= coexistence).  In the case of coexistence, the equilibrium measure $\nu_\theta$ also depends on the migration kernel $a(\cdot,\cdot)$, the values of the parameters $e,K$, and the diffusion function $g \in \CG$  (recall \eqref{gh6}). The dichotomy itself, however, is \emph{controlled by $I_{\hat{a}}$ only}. In particular, $g\in\CG$ plays no role, a fact that will be shown with the help of a \emph{duality comparison} argument. In view of Theorem \ref{T.dichcrit.model1}, if $g=dg_{\mathrm{FW}}$, then $I_{\hat{a}}=\infty$ implies that with probability 1 two ancestral lineages in the dual coalesce. Therefore $I_{\hat{a}}=\infty$ is said to be \emph{the total hazard of coalescence}. Remarkably, this dichotomy is the same as the dichotomy observed for systems without seed-bank (see \cite{CG94}): clustering prevails for recurrent migration; coexistence prevails for transient migration; for $\G=\Z^d$ the critical dimension is $d=2$. From the proof in Section \ref{ss.clustcase} it will become clear that in the dual the ancestral lineages in the long run behave like the ancestral lineages without seed-bank, but are slowed down by a factor $\frac{1}{1+K}$. Consequently, the dormant periods of the ancestral lineages do not affect the dichotomy of the system. In particular, it does not affect the critical dimension separating clustering from coexistence. 

\begin{remark}{{\bf [Ergodic decomposition]}}
{\rm Because $\CT$ is a Choquet simplex, Theorem~\ref{T.ltb1} carries over from $\mu(0)\in\CT^{\mathrm{erg}}$ to $\mu(0)\in\CT$, after decomposition into ergodic components.} \hfill $\Box$
\end{remark}

%%%

\subsection{Long-time behaviour of  Model 2}
\label{ss.scal2}

For model 2 we need the extra condition that $a(\cdot,\cdot)$ is \emph{symmetric}, i.e.,
\begin{equation}
\label{sym}
a(i,j) = a(j,i) \qquad \forall\, i,j \in \G.
\end{equation} 
Note that $\hat{a}_t(0,0) = a_t(0,0)$ because of \eqref{sym}. Below we comment on what happens when we drop this assumption. Recall \eqref{emcond}--\eqref{rhodef}. It turns out that the long-time behaviour of model 2 is different for $\rho<\infty$ and $\rho=\infty$.

\paragraph{{\bf Case $\rho<\infty$.}}
For a finite seed-bank, we define the initial density as
\begin{equation}
\label{thetadefalt}
\theta = \E_{\mu(0)}\left[\frac{x_0 + \sum_{m\in\N_0} K_m\,y_{0,m}}{1+\rho}\right],
\end{equation}
which is the counter part of \eqref{thetadef} in model 1. Like in model 1, it follows from the SSDE in \eqref{gh1*}--\eqref{gh2*} that
\begin{equation}
\label{mart2}
\left(\frac{x_0(t) + \sum_{m\in\N_0} K_m\,y_{0,m}(t)}{1+\rho}\right)_{t\geq 0}
\end{equation}
is a martingale. Hence also here the density is a preserved quantity under the evolution of the system. The dichotomy is controlled by the same integral $I_{\hat{a}}$ as defined in \eqref{Ihatpdef} for model 1.

\paragraph{{\bf Case $\rho=\infty$.}} 
For an infinite seed-bank, we assume that (recall Remark \ref{fattails})
\begin{equation}
\label{assalt}
\begin{aligned}
&K_m \sim A\,m^{-\alpha}, \quad e_m \sim B\,m^{-\beta}, \quad m \to \infty,\\
&A,B \in (0,\infty), \quad \alpha,\beta \in \R\colon\,\alpha \leq 1 < \alpha +\beta,
\end{aligned}
\end{equation}
for which
\begin{equation}
\label{Ptail}
{P}(\tau >t) \sim C\,t^{-\gamma}, \quad t\to\infty,
\end{equation}
with $\gamma = \frac{\alpha+\beta-1}{\beta}\in(0,1)$ and $C = \frac{A}{\beta}\,B^{1-\gamma}\,\gamma\Gamma(\gamma)\in(0,\infty)$, where $\Gamma$ is the Gamma-function. In addition, we assume that the initial measure $\mu(0)$ is colour regular (recall Definition \ref{D.regular}), and define
\begin{equation}
\theta = \lim_{m\to\infty}\E[y_{0,m}].
\end{equation}
This ensures the existence of the initial density
\begin{equation}
\label{thetaatt}
\theta = \lim_{M\to\infty}  \E_{\mu(0)}\left[\frac{x_0 + \sum_{m=0}^M K_m\,y_{0,m}}
{1+\sum_{m=0}^M K_m}\right]. 
\end{equation}
It turns out that the dichotomy is controlled by the integral
\begin{equation}
\label{Idef}
I_{\hat{a},\gamma} = \int_1^\infty \d t\,\,t^{-(1-\gamma)/\gamma}\,\hat{a}_t(0,0)
\end{equation}
instead of the integral $I_{\hat{a}}$ for $\rho<\infty$.

\medskip
For $\theta\in(0,1)$, define (both for $\rho<\infty$ and $\rho=\infty$) 
\begin{equation}
\label{deftergTalt}
\CT^{\mathrm{erg}}_\theta=\left\{\mu\in\CT^{\mathrm{erg}}\colon\,\lim_{M\to\infty}
\E_{\mu(0)}\left[ \frac{x_0 + \sum_{m=0}^MK_my_{0,m}}{1+\sum_{m=0}^MK_m}\right]=\theta \right\}.
\end{equation}

\begin{theorem}{{\bf [Long-time behaviour: model 2]}}
\label{T.ltb2}
\begin{itemize} 
\item[{\rm (I)}] 
Let $\rho<\infty$. Assume \eqref{ass2} and \eqref{sym}. Suppose that $\mu(0) \in \CT^{\mathrm{erg}}_\theta$.
\begin{itemize}
\item[{\rm (a)}] (Coexistence regime)
If $I_{\hat{a}}<\infty$, then
\begin{equation}
\label{gh13alt}
\lim_{t\to\infty} \mu(t) = \nu_\theta,
\end{equation}
where
\begin{eqnarray}
\label{e1291}
&&\nu_\theta \mbox{ is an equilibrium measure for the process on } E,\\
\label{gh14alt}
&&\nu_\theta \mbox{ is invariant, ergodic and mixing under translations},\\
\label{gh14*alt}
&&\nu_\theta \mbox{ is associated},\\
\label{gh15alt}
&&\E_{\nu_\theta}[x_0] = \E_{\nu_\theta}[y_{0,m}] = \theta\,\,\forall\,m\in\N_0, 
\end{eqnarray}
with $\E_{\nu_\theta}$ denoting expectation over $\nu_\theta$.
Moreover,
\begin{equation}\label{vari}
\begin{aligned}
\liminf_{m\to\infty} e_m > 0\colon &\liminf_{m\to\infty} \mathrm{Var}_{\nu_\theta}(y_{0,m}) > 0,\\
\limsup_{m\to\infty} e_m = 0\colon &\limsup_{m\to\infty} \mathrm{Var}_{\nu_\theta}(y_{0,m}) = 0.
\end{aligned}
\end{equation}
\item[{\rm (b)}] (Clustering regime)
If $I_{\hat{a}}=\infty$, then
\begin{equation}
\label{gh12alt}
\lim_{t\to\infty} \mu(t) = \theta\, [\delta_{(1,1^{\N_0})}]^{\otimes \G} 
+ (1-\theta)\, [\delta_{(0,0^{\N_0})}]^{\otimes \G}.
\end{equation}
\end{itemize}
\item[{\rm (II)}] 
Let $\rho=\infty$. Assume \eqref{ass2}, \eqref{sym} and \eqref{assalt}. Suppose that $\mu(0)\in\CT^{\mathrm{erg}}$ and, in addition, is colour regular with initial density $\theta$ given by \eqref{thetaatt}. Then the same results as in {\rm (I)} hold after $I_{\hat{a}}$ in \eqref{Ihatpdef} is replaced by $I_{\hat{a},\gamma}$ in \eqref{Idef}. Moreover, 
\begin{equation}
\label{thetaattlim}
 \lim_{M\to\infty}  \E_{\nu_\theta}\left[\frac{x_0 + \sum_{m=0}^M K_m\,y_{0,m}}
{1+\sum_{m=0}^M K_m}\right]=\theta,  
\end{equation} 
and $\nu_\theta $ is colour regular.
\end{itemize}
\end{theorem}

The result in part (I) shows that for $\rho<\infty$ the long-time behaviour is similar to that of model 1. Like in model 1, the results in \eqref{gh13alt}--\eqref{gh12alt} say that the system converges to an equilibrium whose density of type $\heartsuit$ equals $\theta$ in \eqref{thetadef}, the density of $\heartsuit$ under the initial measure $\mu(0)$. Again, the equilibrium can be either mono-type or multi-type, depending on whether the symmetrised migration kernel is recurrent or transient. Like in model 1, in both cases the equilibrium measure depends on $\theta$. In the case of coexistence, the equilibrium measure $\nu_\theta$ also depends on the migration kernel $a(\cdot,\cdot)$, the sequences of parameters $(e_m)_{m\in\N_0}$ and $(K_m)_{m\in\N_0}$, and the diffusion function $g \in \CG$ (recall \eqref{gh6}). Again, the dichotomy itself is \emph{controlled by $I_{\hat{a}}$ only}, and the resampling rate given by $g\in\CG$ plays no role. Therefore if $g=dg_{FW}$, in view of Theorem \ref{T.dichcrit.model1}, whether or not two ancestral lineages in the dual coalesce with probability 1 is still only determined by the migration kernel $a(\cdot,\cdot)$.  The same dichotomy holds as for systems without seed-bank (see \cite{CG94}). Therefore part (I) of Theorem \ref{T.ltb2} indicates that, as long as the dormant periods of the ancestral lineages in the dual have a finite mean (here $\frac{\rho}{1+\rho}$; recall Remark \ref{fattails}), the seed-bank does not affect the dichotomy of the system. 

Even so, \eqref{vari} indicates that there is interesting behaviour in the deep seed-banks. Indeed, when the exchange rate $e_m$ between the $m$-dormant and the active population is bounded away from zero as $m\to\infty$ the deep seed-banks are asymptotically \emph{random}, while when $e_m$ tend to zero as $m\to\infty$ the deep seed-banks are asymptotically \emph{deterministic}. The latter means that the deep seed-banks serve as a reservoir, containing a fixed mixture of types. For $\rho<\infty$ this reservoir is too small to influence the dichotomy of the system, but not for $\rho=\infty$.  

For $\rho=\infty$ the system again converges to an equilibrium whose density of type $\heartsuit$ equals $\theta$ in \eqref{thetaatt}, the density of $\heartsuit$ under the initial measure $\mu(0)$. The equilibrium can be mono-type or multi-type, but \emph{the dichotomy criterion has changed.}  Instead of $I_{\hat{a}}$, the dichotomy is now controlled by the integral $I_{\hat{a},\gamma}$ (recall \eqref{Idef}), where $\gamma$ is the parameter determined by relative sizes $K_m$ of the colour $m$-dormant populations with respect to the active population and the exchanges rates $(e_m)_{m\in\N_0}$ with the seed-bank, recall \eqref{assalt}--\eqref{Ptail}.  If $g=dg_{FW}$, $\gamma$ is the parameter of the tail of the wake-up time of an ancestral lineages in the dual (recall \eqref{fattails}). Therefore if $g=dg_{Fw}$, in view of Theorem \ref{T.dichcrit.model2}, we see that the dormant periods of the ancestral lineages in the dual do affect whether or not two ancestral lineages in the dual coalesce with probability $1$. For general $g\in\CG$, the integral $I_{\hat{a},\gamma}$ in \eqref{Idef} shows a \emph{competition} between migration and exchange. The smaller $\gamma$ is, the longer the individuals remain dormant in the seed-bank, the smaller $I_{\hat{a},\gamma}$ is, and the more coexistence becomes likely. As a consequence clustering requires more stringent conditions than recurrent migration; for $\G=\Z^d$ the critical dimension is $1<d<2$ for $\gamma \in [\tfrac12,1]$ and $d=1$ for $\gamma \in (0,\tfrac12)$. The \emph{seed-bank enhances genetic diversity}. Note that $\gamma \uparrow 1$ links up with the case $\rho<\infty$, where coexistence occurs if and only if the migration is transient. Also note that for $\gamma \in (0,\tfrac12)$ there is always coexistence \emph{irrespective of the migration}.

In the case of clustering the equilibrium measure only depends on $\theta$, while in the case of coexistence, like for $\rho<\infty$, $\nu_\theta$ depends on the migration kernel $a(\cdot,\cdot)$, the sequences of parameters $(e_m)_{m\in\N_0},\ (K_m)_{m\in\N_0}$, and the diffusion function $g \in \CG$. Since we assumed \eqref{assalt}, we have $\limsup_{m\to\infty} e_m=0$ , and so we are automatically in the second case of \eqref{vari}. Hence the deep seed-banks are asymptotically deterministic, i.e., the $m$-dormant population converges in law to a deterministic state $\theta$ as $m\to\infty$.  Roughly speaking, in case $g=dg_{FW}$, in equilibrium the volatility of a colour is inversely proportional to its average wake-up time in the dual. Since $\rho=\infty$, for each $M\in\N_0$ we have $\sum_{m=M}^\infty K_m=\infty$, and in the coexistence regime the effect of the seed-bank can be interpreted as a migration towards an infinite reservoir with deterministic density $\theta$.  

Like for model 1, also here $\CT$ is a Choquet simplex, and Theorem~\ref{T.ltb2} carries over from $\CT^{\mathrm{erg}}$ to $\CT$, after decomposition into ergodic components.

\paragraph{Example of effect of infinite seed-bank.}

For a symmetric migration kernel with finite second moment the following holds:
\begin{itemize}
\item 
For $\G=\Z^2$, $\hat{a}_t(0,0) \asymp t^{-1}$, $t\to\infty$, and so coexistence occurs for all $\gamma \in (0,1)$. 
\item 
For $\G=\Z$, $\hat{a}_t(0,0) \asymp t^{-1/2}$, $t\to\infty$, and so coexistence occurs if and only if $\gamma \in (0,\tfrac23)$.
\end{itemize}
In both cases the migration is recurrent, so that clustering prevails in model 1. 

\begin{corollary}{\bf [Three regimes]}\label{threereg} 
Under the conditions of Theorem \ref{T.ltb2}, the system in \eqref{gh1*}--\eqref{gh2*} has three different parameter regimes:
\begin{itemize}
\item[{\rm (1)}]
$\gamma \in (1,\infty)$: migration determines the dichotomy.
\item[{\rm (2)}]
$\gamma \in [\tfrac12,1]$: interplay between migration and seed-bank determines the dichotomy. 
\item[{\rm (3)}]
$\gamma \in (0,\tfrac12)$: seed-bank determines the dichotomy. 
\end{itemize} 
\end{corollary}

\paragraph{Role of symmetry in migration.}

Unlike in model 1, it is \emph{not} possible to remove the symmetry assumption in \eqref{sym}, as the following counterexample shows. We consider model 2 with $\rho<\infty$ under assumption \eqref{ass2}, but we do not assume \eqref{sym}.

\begin{itemize}
\item
{\bf Counterexample:} Let $\G=\Z^2$, and for $\eta \in (0,1)$ pick
\begin{equation}
\label{achoice}
a(i,j) = \left\{\begin{array}{ll}
\tfrac14(1+\eta), &j=i+(1,0) \text{ or } i+(0,1),\\[0.2cm] 
\tfrac14(1-\eta), &j=i-(1,0) \text{ or } i-(0,1),\\ 
\end{array}
\right.
\end{equation}
i.e., two-dimensional nearest-neighbour random walk with drift upward and rightward. Suppose that $\tau$ in \eqref{Ptail} has a \emph{one-sided stable distribution} with parameter $\gamma \in (1,2)$ (obtained from \eqref{assalt} but with $\alpha,\beta \in \R$: $1<\alpha<1+\beta$). Then coexistence occurs while $I_{\hat{a}} = \infty$. 
\end{itemize}

\noindent 
Recall that for the two-dimensional nearest-neighbour random walk without drift we get clustering according to Theorem~\ref{T.ltb2}, independently of the distribution of $\tau$. The key feature of the counterexample is that it corresponds to $\mathbb{E}(\tau)<\infty$ and $\mathbb{E}(\tau^2)=\infty$. Hence the central limit theorem fails for $\tau$. We will see in Section~\ref{asymmig} that the failure of the central limit theorem for $\tau$ is responsible for turning clustering into coexistence.   

The above raises the question to what extent the equilibrium behaviour depends on the nature of the geographic space. To answer this question, we need a key concept for random walks on countable Abelian groups, which we describe next. 

\begin{remark}{{\bf [Dichotomy criterion and degrees of random walk]}}
\label{r.1430}
{\rm We can read the condition $I_{\hat{a},\gamma}<\infty$ for coexistence versus $I_{a,\gamma}=\infty$ for clustering in terms of the \emph{degree} of the random walk. Namely, let $\hat{a}(\cdot,\cdot)$ be the transition kernel of an irreducible random walk on a countable Abelian group. Then the degree $\delta$ of $\hat{a}(\cdot,\cdot)$ is defined as
\begin{equation}
\label{e1434}
\delta = \sup\left\{\zeta > -1\colon\,\int^\infty_1  \d t\, t^\zeta\, \hat{a}_t(0,0)< \infty\right\}.
\end{equation}
The degree is defined to be $\delta^+$ when the integral is finite \emph{at} the degree and $\delta^-$ when the integral is infinite \emph{at} the degree. Hence we can rephrase the dichotomy criterion in Theorem \ref{T.ltb2} as
\begin{equation}
\label{e1440}
\mbox{clustering} \quad \Longleftrightarrow \quad \text{ either } 
-\frac{1-\gamma}{\gamma} \geq \delta^- 
\mbox{ or } 
- \frac{1-\gamma}{\gamma} > \delta^+.
\end{equation}
For further details we refer to \cite{DGW04a}, \cite{DGW05}, which relate the degree of the random walk to the tail of its return time to the origin.} \hfill $\Box$
\end{remark}

\paragraph{Modulation of wake-up time with slowly varying function.} 

Under weak conditions it is possible to modulate \eqref{Ptail} by a slowly varying function. Assume that  
\begin{equation}
\label{mod}
\frac{P(\tau \in \d t)}{\d t}  \sim \varphi(t)\, t^{-(1+\gamma)}, \quad t \to\infty,  
\end{equation}
with $\varphi$ \emph{slowly varying at infinity}. Define
\begin{equation}
\hat\varphi(t) = \left\{\begin{array}{ll}
\varphi(t), &\gamma \in (0,1),\\[0.2cm]
\int_1^t \d s\,\varphi(s) s^{-1}, &\gamma=1. 
\end{array}
\right.
\end{equation} 
As shown in \cite[Section 1.3]{BGT87}, without loss of generality we may take $\hat\varphi$ to be infinitely differentiable and to be represented by the integral  
\begin{equation}
\label{hatphirepr}
\hat\varphi(t) = \exp\left[\int_{(\cdot)}^t \frac{\d u}{u}\,\psi(u)\right]
\end{equation}
for some $\psi\colon [0,\infty) \to \R$ such that $\lim_{u\to\infty} |\psi(u)| = 0$. If we assume that $\psi$ eventually has a sign and satisfies $|\psi(u)| \leq C/\log u$ for some $C<\infty$, then \eqref{Idef} needs to be replaced by
\begin{equation}
\label{cluscritseed-b}
I_{\hat{a},\gamma,\varphi} = \int_1^\infty  \d t\,\hat\varphi(t)^{-1/\gamma}\,
t^{-(1-\gamma)/\gamma}\,\hat{a}_t(0,0).
\end{equation}
A proof is given in Section \ref{ss.proofslowvar}. The modulation of the wake-up time by a slowly varying function appears naturally for the  model on the hierarchical group, analysed in \cite{GdHOpr3}. There the integral criterion for the dichotomy in \eqref{cluscritseed-b} is needed to apply Theorem \ref{T.ltb2}.

%%%

\subsection{Long-time behaviour of Model 3}
\label{ss.scal3}

It remains to see how the switch of colony during the exchange affects the dichotomy. We will focus on the special case where the displacement kernels do not depend on $m$, i.e.,   
\begin{equation}
\label{displa}
a_m(\cdot,\cdot) = a^\dagger(\cdot,\cdot) \qquad \forall \,m\in\N_0,
\end{equation}
with $a^\dagger(\cdot,\cdot)$ an irreducible \emph{symmetric} random walk kernel on $\G\times\G$. Let $\hat{a}^\dagger_t(\cdot,\cdot)$ denote the time-$t$ transition kernel of the random walk with symmetrised displacement kernel $\hat{a}^\dagger(\cdot,\cdot)$ ($= a^\dagger(\cdot,\cdot)$) and jump rate 1. Assume that (compare with \eqref{ass2})
\begin{equation}
\label{qsym}
\begin{aligned}
&t \mapsto (\hat{a}_t \ast \hat{a}^\dagger_t)(0,0) \text{ is regularly varying at infinity},\\
&(\hat{a}_{Ct} \ast \hat{a}^\dagger_t)(0,0) \asymp (\hat{a}_t \ast \hat{a}^\dagger_t)(0,0) \text{ as } t \to\infty 
\text{ for every } C \in (0,\infty),
\end{aligned}
\end{equation}
where $\ast$ stands for convolution. Let
\begin{equation}
\label{Idefaltplane}
I_{\hat{a} \ast \hat{a}^\dagger} 
= \int_1^\infty \d t\,\,(\hat{a}_t \ast \hat{a}^\dagger_t)(0,0)
\end{equation}
and
\begin{equation}
\label{Idefalt}
I_{\hat{a} \ast \hat{a}^\dagger,\gamma} 
= \int_1^\infty \d t\,\,t^{-(1-\gamma)/\gamma}\,(\hat{a}_t \ast \hat{a}^\dagger_t)(0,0).
\end{equation} 

\begin{theorem}{{\bf [Long-time behaviour: model 3]}}
\label{T.ltb3}
Suppose that, in addition to the assumptions of Theorem~\ref{T.ltb2}, both \eqref{displa} and \eqref{qsym} hold. Then the same results as for model $2$ hold: (I) for $\rho<\infty$ after $I_{\hat{a}}$ in \eqref{Ihatpdef} is replaced by $I_{\hat{a} \ast \hat{a}^\dagger}$ in \eqref{Idefaltplane};  (II) for $\rho=\infty$ after $I_{\hat{a},\gamma}$ in \eqref{Idef} is replaced by $I_{\hat{a} \ast \hat{a}^\dagger,\gamma}$ in \eqref{Idefalt}. 
\end{theorem}

\noindent
In the case of coexistence the equilibrium measure $\nu_\theta$ depends on $a(\cdot,\cdot)$, $a^\dagger(\cdot,\cdot)$, $(e_m)_{m\in\N_0}$, $(K_m)_{m\in\N_0}$ and $g \in \CG$. The dichotomy itself, however, is controlled by $I_{\hat{a} \ast \hat{a}^\dagger}$, respectively, $I_{\hat{a} \ast \hat{a}^\dagger,\gamma}$ alone.

An interesting observation is the following. Since $\hat{a}_t(\cdot,\cdot)$ and $\hat{a}^\dagger_t(\cdot,\cdot)$ are symmetric, we have (by a standard Fourier argument)
\begin{equation}
\label{e1102}
\hat{a}_t(i,j) \leq \hat{a}_t(0,0), \qquad \hat{a}^\dagger_t(i,j) \leq \hat{a}^\dagger_t(0,0) 
\qquad \forall\,i,j \in \G\,\,\forall\,t \geq 0.
\end{equation} 
Hence, $I_{\hat{a} \ast \hat{a}^\dagger,\gamma} \leq I_{\hat{a},\gamma} \wedge I_{\hat{a}^\dagger,\gamma}$. Consequently, the extra displacement in model 3 can only make coexistence more likely compared to model 2, which is intuitively plausible.  

If $a(\cdot,\cdot)=a^\dagger(\cdot,\cdot)$, then $(a_t \ast a^\dagger_t)(0,0) = a_{2t}(0,0)$ and therefore  the  dichotomy is the same as for model 2. Hence the extra displacement  has in this case no effect on the dichotomy. However, if the displacement is transient while the migration is recurrent, then there is a difference. For instance, if $\rho<\infty$, the migration is a simple random walk on $\Z$, and the displacement is a symmetric random walk on $\Z$ with infinite mean, e.g.\ $a^\dagger(0,x) = a^\dagger(0,-x) \sim D|x|^{-\delta}$, $D \in (0,\infty)$, $\delta \in (1,2)$, then $I_{\hat{a}}=\infty$, $I_{\hat{a}^\dagger}<\infty$ and $I_{\hat{a} \ast \hat{a}^\dagger} < \infty$ \cite[Section 8]{Sp64}. Therefore there is clustering in model 2, but coexistence in model 3.

%%%%%%%%%%%%%%% SECTION 4 %%%%%%%%%%%%%%%%%

\section{Proofs: Well-posedness and duality}
\label{s.wpdualpr}

In Section~\ref{ss.mp} we prove Theorem~\ref{T.wellp}, in Section~\ref{ss.dual} Theorems~\ref{T.dual1}, \ref{T.dual2} and \ref{T.dual3}, and in Section~\ref{ss.dichotomy} Theorems~\ref{T.dichcrit.model1} and \ref{T.dichcrit.model2}.

%%%

\subsection{Well-posedness}
\label{ss.mp}

In this section we prove Theorem~\ref{T.wellp}.

\begin{proof}
(a) We first prove Theorem~\ref{T.wellp}(a): existence and uniqueness of solutions to the SSDE. We do this for each of the three models separately.
	
\paragraph{Model 1.} 
Existence of the process defined in \eqref{gh1}--\eqref{gh2} for model 1 is a consequence 
of the assumptions in \eqref{gh10}, \eqref{pmdef} and \eqref{emcond}, in combination with 
\cite[Theorem 3.2]{SS80}, which reads as follows:
	
\begin{theorem}{{\bf [Unique strong solution]}}
\label{SS80theorem}
Let $\S$ be a countable set, and let $Z=\{z_u\}_{u\in \S} \in[0,1]^\S$. Consider the stochastic differential equation
\begin{equation}
\label{shigasde}
\d z_u(t)=\alpha_u(z_u(t))\,\d B_u(t) + f_u(Z(t))\,\d t,\qquad u\in \S,
\end{equation}
where $\alpha_u\colon\,[0,1] \to \R$ for all $u\in \S$, $f_u\colon\,[0,1]^\S \to [0,1]$ for all $u\in \S$, and $B=\{B_u\}_{u\in \S}$ is a collection of independent standard Brownian motions. Suppose that:
\begin{enumerate}
\item[{\rm (1)}] 
The functions $\alpha_u$, $u\in \S$, are real-valued, $\frac{1}{2}$-H\"older continuous (i.e., there are $C_u \in (0,\infty)$ such that $|\alpha_u(x)-\alpha_u(y)| \leq C_u |x-y|^{\frac{1}{2}}$ for all $x,y \in [0,1]$) and uniformly bounded, with $\alpha_u(0)=\alpha_u(1)=0$, $u\in \S$.
\item[{\rm (2)}]
The functions $f_u$, $u\in \S$, are continuous and satisfy:
\begin{itemize}
\item[$\bullet$] 
There exists a matrix $Q=\{Q_{u,v}\}_{u,v\in \S}$ such that  $Q_{u,v}\geq 0$ for all $u,v \in \S$,
$\sup_{u\in \S} \sum_{v\in \S}$ $Q_{u,v} < \infty$, and
\begin{equation}
\label{e1500}
|f_u(Z^1)-f_u(Z^2)| \leq \sum_{v\in \S} Q_{u,v}|z^1_{v}-z^2_{v}|, 
\quad \text{ for } \quad Z^1=\{z^1_{v}\}_{v\in \S}\in[0,1]^\S,\ Z^2=\{z^2_{v}\}_{v\in \S}\in[0,1]^\S.
\end{equation}
\item[$\bullet$] 
For $Z\in [0,1]^\S$ and $z_u=0$,
\begin{equation}
\label{e1505}
f_u(Z)\geq 0.
\end{equation}
\item[$\bullet$] 
For $Z\in [0,1]^\S$ and $z_u=1$,
\begin{equation}
\label{e1509}
f_u(Z)\leq 0.
\end{equation}
\end{itemize}
\end{enumerate}
Then \eqref{shigasde} has a unique $[0,1]^\S$-valued strong solution with a continuous path. 
\end{theorem}
	
To apply Theorem~\ref{SS80theorem} to model 1, recall that
\begin{equation}
\label{e1332}
\S=\G \times \{A,D\}, 
\end{equation}
where $A$ denotes the active part of a colony and $D$ the dormant part of a colony. Since $\G$ is countable and $\{A,D\}$ is finite, $\S$ is countable. As before, we denote the fraction of active individuals of type $\heartsuit$ at colony $i\in\G$ by $x_i$ and the fraction of dormant individuals of type $\heartsuit$ at colony $i\in\G$ by $y_i$. Note that for every $u \in \S$ we have either $u=(i,A)$ or $u=(i,D)$ for some $i \in \G$. Therefore $Z=\{z_u\}_{u\in \S}=\{x_i\colon\, i\in\G\}\cup\{y_i\colon\,i\in\G\}$, and $z_u=x_i$ when $u=(i,A)$ and $z_u=y_i$ when $u=(i,D)$. We can rewrite \eqref{gh1}--\eqref{gh2} in the form of \eqref{shigasde} by picking
\begin{equation}
\label{e1529}
\alpha_u(z_u)=\begin{cases}
\sqrt{g(x_i)}, &u=(i,A),\\
0, &u=(i,D),
\end{cases}
\end{equation}
and 
\begin{equation}
\label{e1530}
f_u(Z) = \begin{cases}
\sum_{j \in \G} a(i,j)\,(x_j-x_i)+Ke\,(y_i-x_i), 
&u=(i,A),\\
e\,(x_i-y_i),
&u=(i,D).
\end{cases}
\end{equation}
Since $g\in \CG$ (recall \eqref{gh6}), the conditions in (1) are satisfied. To check the conditions in (2), define the matrix $Q=\{Q_{u,v}\}_{u,v\in\S}$ by
\begin{equation}
\label{e1547}
Q_{u,v} =
\begin{cases}
\sum_{j\in\G} a(i,j)+Ke, &u=(i,A),\,v=(i,A),\\
a(i,j), &u=(i,A),\,v=(j,A),\\
Ke, &u=(i, A),\,v=(i,D),\\
e, &u=(i,D),\,v=(i,D) \text{ or } u=(i,D),\,v=(i,A),\\
0, &\text{otherwise.}
\end{cases}
\end{equation}
Then
\begin{equation}
\label{e1558}
\sum_{v\in S} Q_{u,v} = 
\begin{cases}
2 \sum_{j\in\G} a(i,j)+2Ke, &u=(i,A),\\
2e, &u=(i,D).
\end{cases}
\end{equation}
Since we have assumed that $\sum_{j \in \G} a(i,j)=\sum_{j \in \G} a(0,j-i)<\infty$, it follows that $\sup_{u\in \S} \sum_{v\in \S} Q_{u,v} <\infty$. Since $x_i\in[0,1]$ and $y_i\in[0,1]$, the requirements on $f_u$ are immediate. Hence we have a unique strong solution with a continuous path. 

By It\^o's formula, the law of the strong solution solves the martingale problem. Uniqueness of that solution follows from \cite[Theorem IX 1.7(i)]{RY99}. This in turn implies the Markov property.
	
\paragraph{Model 2.}
To apply Theorem~\ref{SS80theorem} to model 2, recall that
\begin{equation}
\label{e1332alt}
\S= \G \times \{A,(D_m)_{m\in\N_0}\}.
\end{equation}
Pick 
\begin{equation}
\label{e1580}
\alpha_u(z_u)=\begin{cases}
\sqrt{g(x_i)}, &u=(i,A),\\
0, &u=(i,D_m),\,m\in\N_0,
\end{cases}
\end{equation}
and
\begin{equation}
\label{e1587}
\begin{aligned}
f_u(Z) = \begin{cases}
\sum_{j \in \G} a(i,j)\,(x_j-x_i)+\sum_{m\in\N_0} K_me_m\,(y_{i,m}-x_i), 
&u=(i,A),\\
e_m\,(x_i-y_{i,m}),
&u=(i,D_m).
\end{cases}
\end{aligned}
\end{equation}
Set
\begin{equation}
\label{e1599}
Q_{u,v} =
\begin{cases}
\sum_{j\in\G} a(i,j)+\sum_{m\in\N_0} K_me_m, &u=(i,A),\, v=(i,A),\\
a(i,j), &u=(i,A),\, v=(j,A),\,j \neq i,\\
K_me_m, &u=(i, A),\, v=(i,D_m),\\
e_m, &u=(i,D_m),\, v=(i,D_m) \text{ or } u=(i,D_m),\, v=(i,A),\\
0, &\text{otherwise.}
\end{cases}
\end{equation}
Then, by assumptions \eqref{gh10} and \eqref{emcond}, $Q$, $f$ and $\alpha$ satisfy the conditions of Theorem~\ref{SS80theorem}.
	
\paragraph{Model 3.} 
The state space $\S$ and the function $\alpha$ are the same as in model $2$. When $u\in \S$ is of the form $(i, A)$, we must adapt the function $f_u$ such that it takes the displacement of seeds into account. The matrix $Q$ must be adapted accordingly and, by assumption \eqref{pmdef}, the conditions of Theorem~\ref{SS80theorem} are again satisfied. 
	
\medskip\noindent
(b) The proof of Theorem~\ref{T.wellp}(b) is the same for models 1--3. The Feller property can be proved by using duality if $g = d g_{\text{FW}}$, $d \in (0,\infty)$. For general $g$ we use \cite[Remark 3.2]{SS80} (see also \cite[Theorem 5.8]{Lig85}). The Feller property in turn implies the strong Markov property. 
\end{proof}

%%%

\subsection{Duality}
\label{ss.dual}

In this section we prove Theorems~\ref{T.dual1}, \ref{T.dual2} and \ref{T.dual3}.

\paragraph{Model 1: Proof of Theorem~\ref{T.dual1}.}

\begin{proof} 
We use the generator criterion (see \cite[p.190--193]{EK86} or \cite[Proposition 1.2]{JK14}) to prove the duality relation given in \eqref{e401}. Let $F$ be the generator  of the spatial block-counting process defined in \eqref{blockproc}, and let $H((m_j,n_j)_{j\in\G})$ be defined as in \eqref{hpoldef}, but read as a function of the second sequence only. Then
\begin{equation}
\begin{aligned}
(FH)\big((m_j,n_j)_{j\in\G}\big)
&=\sum_{i\in\G}\bigg[
\sum_{k\in\G} m_i 
a(i,k)\,\big[H\big(	({m}_j,{n}_j)_{j\in\G}-\delta_{(i,A)}+\delta_{(k,A)}\big)
-H\big((m_j,n_j)_{j\in\G}\big)\big]\\
& \qquad +\, d{m_i\choose 2}\, \big[H\big(({m}_j,{n}_j)_{j\in\G}-\delta_{(i,A)}\big)
-H\big((m_j,n_j)_{j\in\G}\big)\big]\\
& \qquad +\, m_i K e\,\big[H\big(({m}_j,{n}_j)_{j\in\G}-\delta_{(i,A)}+\delta_{(i,D)}\big)
-H\big((m_j,n_j)_{j\in\G}\big)\big]\\
& \qquad +\, n_i e\,\big[H\big(	({m}_j,{n}_j)_{j\in\G}+\delta_{(i,A)}-\delta_{(i,D)}\big)
-H\big((m_j,n_j)_{j\in\G}\big)\big]
\bigg].
\end{aligned}
\end{equation}
Recall that $G$ is the generator of the SSDE (recall \eqref{eq401}--\eqref{eq402}). Let $\mathcal{D}_G$ denote the domain of $G$ and $\mathcal{D}_F$ the domain of $F$. Let $(S_t)_{t \geq 0}$ denote the semigroup of the process $(Z(t))_{t\geq 0}$ in \eqref{e409} and $(R_t)_{t\geq 0}$ the semigroup of the process $(L(t))_{t\geq 0}$ in \eqref{blctpr}. Since 
\begin{equation}
\frac{d^2}{dt^2} (R_tH)((x_j,y_j,n_j,m_j)_{j\in\mathbb{G}}) = (F^2R_tH)((x_j,y_j,n_j,m_j)_{j\in\mathbb{G}}),
\end{equation} 
we see that $H((x_j,y_j,n_j,m_j)_{j\in\mathbb{G}})\in \mathcal{D}_G$ and $(R_tH)((x_j,y_j,n_j,m_j)_{j\in\mathbb{G}})\in\mathcal{D}_G$. It is also immediate that $H((x_j,y_j,n_j,m_j)_{j\in\mathbb{G}})\in\mathcal{D}_F$ and $(S_tH)((x_j,y_j,n_j,m_j)_{j\in\mathbb{G}})\in\mathcal{D}_F$. Applying the generator $G$ in \eqref{eq402} with $g =\frac{d}{2}g_{\text{FW}}$ to \eqref{hpoldef}, we find
\begin{equation}
\begin{aligned}
&(GH)\big((x_j,y_j)_{j\in\G}\big)\\
&= \sum_{i \in \G} \Bigg\{ \Bigg[ \sum_{k \in \G} 
a(i,k)\,(x_k - x_i)\Bigg] \frac{\partial}{\partial x_i} 
\Big(\prod_{j\in\G}x_j^{m_j}y_j^{n_j}\Big)\\
&\qquad + \tfrac{d}{2}\,x_i(1-x_i) \frac{\partial^2}{\partial x_i^2}
\Big(\prod_{j\in\G}x_j^{m_j}y_j^{n_j}\Big)
+ Ke\,(y_i-x_i) \frac{\partial}{\partial x_i}
\Big(\prod_{j\in\G}x_j^{m_j}y_j^{n_j}\Big)\\
&\qquad\qquad\qquad + e\,(x_i -y_i) \frac{\partial}{\partial y_i}
\Big(\prod_{j\in\G}x_j^{m_j}y_j^{n_j}\Big)\Bigg\}\\
&= \sum_{i \in \G} \Bigg\{ \Bigg[ \sum_{k \in \G} m_ia(i,k) 
\prod_{\substack{j\in\G\\ j\neq i\\j\neq k}} x_j^{m_j}y_j^{n_j}
\,\big(x_i^{m_i-1}y_i^{n_i}x_k^{m_k+1}y_k^{n_k}-x_i^{m_i}y_i^{n_i}x_k^{m_l}y_k^{n_k}\big)\,
\Bigg]\\
&\qquad+ \prod_{\substack{j\in\G\\ j\neq i}}x_j^{m_j}y_j^{n_j}\tfrac{d}{2}\,
m_i(m_i-1)\,\big(x_i^{m_i-1}y_i^{n_i}-x_i^{m_i}y_i^{n_i}\big)\,1_{\{m_i\geq 2\}}\\
&\qquad\qquad\qquad + m_i Ke\, \prod_{\substack{j\in\G\\ j\neq i}} 
x_j^{m_j}y_j^{n_j}\,\big(x_i^{m_i-1}y_i^{n_i+1}-x_i^{m_i}y_i^{n_i}\big)\,\\
&\qquad\qquad\qquad + n_i e\, \prod_{\substack{j\in\G\\ j\neq i}}
x_j^{m_j}y_j^{n_j}\,\big(x_i^{m_i+1}y_i^{n_i-1}-x_i^{m_i}y_i^{n_i}\big)\,\Bigg\}\\
&= (FH) \big((m_j,n_j)_{j\in\G}\big).
\end{aligned}
\end{equation}
Consequently, it follows from the generator criterion that 
\begin{equation}
\label{e1476}
\E\Big[H\Big((X_i(t),Y_i(t),m_i,n_i)_{i\in\G}\Big)\Big]=\E\Big[H\Big((x_i,y_i,M_i(t),N_i(t))_{i\in\G}\Big)\Big].
\end{equation}
This settles Theorem~\ref{T.dual1}.
\end{proof}

\paragraph{Model 2: Proof of Theorem~\ref{T.dual2}.}

\begin{proof}
Theorem~\ref{T.dual2} follows after replacing in the above proof the block-counting process in \eqref{blockproc} by the one in \eqref{blockproc2}, the duality function by the one in \eqref{dualmod2alt}, and checking the generator criterion. 
\end{proof}

\paragraph{Model 3: Proof of Theorem~\ref{T.dual3}.}

\begin{proof}
Theorem~\ref{T.dual3} follows after replacing the block-counting process in \eqref{blockproc} by the one in \eqref{blockproc3}, the duality function is by the one in \eqref{dualmod2alt}, and checking the generator criterion. 
\end{proof}

%%%

\subsection{Dichotomy criterion}
\label{ss.dichotomy}

In this section we prove Theorems~\ref{T.dichcrit.model1} and \ref{T.dichcrit.model2}.

\paragraph{Model 1: Proof of Theorem~\ref{T.dichcrit.model1}.} 

\begin{proof}
$\mbox{}$
	
\medskip\noindent	
``$\boldsymbol{\Longleftarrow}$'' 
The proof uses the duality relation in Theorem \ref{T.dual1}.  Define $\theta_x=\E_{\mu(0)}[x_0]$ and $\theta_y=\E_{\mu(0)}[y_0]$. Note that, since $\mu(0)$ is invariant under translations, we have $\E_{\mu(0)}[x_i]=\theta_x$ and $\E_{\mu(0)}[y_i]=\theta_y$  for all $i\in\G$. We proceed as in \cite[Proposition 2.9]{BCKW16}. Let $(m_i,n_i)_{i\in\G}\in E^\prime$ be such that $\sum_{i\in\G} [m_i(0)+n_i(0)]<\infty$, and put 
\begin{equation}
T=\inf\left\{t\geq 0\colon\, \sum_{i\in\G} [m_i(t)+n_i(t)]=1\right\}.
\end{equation}
By assumption, each pair of partition elements coalesces with probability $1$, and hence $\P(T<\infty)=1$. By duality
\begin{equation}
\label{dualarg}
\begin{aligned}
&\lim_{t\to\infty}\E\left[\prod_{i\in\G}x_i(t)^{m_i}y_i(t)^{n_i}\right]\\
&=\lim_{t\to\infty}\E\left[\prod_{i\in\G}x_i^{m_i(t)}y_i^{n_i(t)}\right]\\
&= \lim_{t\to\infty}\E\left[\prod_{i\in\G}x_i^{m_i(t)}y_i^{n_i(t)} ~\Big|~T<\infty\right]\P(T<\infty)
+\E\left[\prod_{i\in\G}x_i^{m_i(t)}y_i^{n_i(t)} ~\Big|~T=\infty\right]\,\P(T=\infty)\\
&= \lim_{t\to\infty}\E\left[\prod_{i\in\G}x_i^{m_i(t)}y_i^{n_i(t)} 
~\Big|~T<\infty,\ m(t)=1,\ n(t)=0\right]\,\P(m(t)=1,\ n(t)=0)\\
&\qquad +\lim_{t\to\infty}\E\left[\prod_{i\in\G}x_i^{m_i(t)}y_i^{n_i(t)} 
~\Big|~T<\infty,\ m(t)=1,\ n(t)=0\right]\,\P(m(t)=0,\ n(t)=1)\\
&= \theta_x\,\frac{1}{1+K}+\theta_y\,\frac{K}{1+K},
\end{aligned}
\end{equation}
where in the last step we use that a single lineage in the dual behaves like the Markov chain with transition kernel $b^{(1)}(\cdot,\cdot)$ defined in \eqref{mrw}. It follows from \eqref{dualarg} that, for all $i,j\in\G$,
\begin{equation}
\label{defclust}
\lim_{t\to\infty} \E\left[\frac{x_i(t)+Ky_i(t)}{1+K}\left(1-\frac{x_j(t)+Ky_j(t)}{1+K}\right)\right]=0.
\end{equation}
Hence, either $\lim_{t\to\infty}(x(t),y(t))=(0,0)^{\G}$ or $\lim_{t\to\infty}(x(t),y(t))=(1,1)^{\G}$. Computing $\lim_{t\to\infty}\E[x_i(t)]$ with the help of \eqref{dualarg}, we find
\begin{equation}
\label{defclus}
\lim_{t\to\infty}\mu(t)=(1-\theta)\,[\delta_{(0,0)}]^{\bigotimes \G}
+\theta\,[\delta_{(1,1)}]^{\bigotimes \G}
\end{equation} 
with $\theta = \E_{\mu(0)}\left[\frac{x_0+Ky_0}{1+K}\right]=\frac{\theta_x+K\theta_y}{1+K}$, which means that the system clusters. 

\medskip\noindent
``$\boldsymbol{\Longrightarrow}$''
Suppose that the systems clusters. Then \eqref{defclust} holds for all $i,j\in\G$, which means that
\begin{equation}
\label{defclust2}
\lim_{t\to\infty} \E\left[z_u(t)\left(1-z_v(t)\right)\right]=0 \qquad \forall\, u,v\in\S.
\end{equation}
Let 
\begin{equation}
\label{numlin}
\left|L(t)\right|=\sum_{u\in\S}L_u(t),
\end{equation}
be the total number of lineages left at time $t$. Applying the duality relation in \eqref{e401b} to \eqref{defclust2}, we find
\begin{equation}
\label{contra1}
\begin{aligned}
0 &= \lim_{t\to\infty}\E\left[z_u(t)(1-z_v(t))\right]\\
&=\lim_{t\to \infty} \E_{\mu(0)}\left[\E_{\delta_u}\left[\prod_{u\in\S}z_u^{L_u(t)}\right]\right]
-\E_{\mu(0)}\left[\E_{\delta_u+\delta_v}\left[\prod_{u\in\S}z_u^{L_u(t)}\right]\right]\\
&=\frac{\theta_x+K\theta_y}{1+K}\left[1-\lim_{t\to \infty}\P_{\delta_u+\delta_v}
\left(|L(t)|=1\right)\right]\\
&\qquad -\lim_{t\to\infty}\E_{\mu(0)}\left[\E_{\delta_u+\delta_v}\left[\prod_{u\in\S}
z_u^{L_u(t)} ~\Big|~ |L(t)|=2\right]\right]\P_{\delta_u+\delta_v}(|L(t)|=2).
\end{aligned}
\end{equation}
As to the last term in the right-hand side of \eqref{contra1}, we note that
\begin{equation}
\label{contra}
\begin{aligned}
&\lim_{t\to\infty}\E_{\mu(0)}\left[\E_{\delta_u+\delta_v}\left[\prod_{u\in\S}
z_u^{L_u(t)} ~\Big|~ |L(t)|=2\right]\right]\\
&\quad =\lim_{t\to\infty}\frac{1}{(1+K)^2}\E_{}\left[\prod_{u\in\S} z_u^{L_u(t)}
~\Big|~ L(t)=\delta_{(i,A)}+\delta_{(j,A)},\,  i,j\in\G\right]\\
&\qquad + \lim_{t\to\infty}\frac{2K}{(1+K)^2}\E_{}\left[\prod_{u\in\S} z_u^{L_u(t)}
~\Big|~ L(t)=\delta_{(i,A)}+\delta_{(j,D)},\,  i,j\in\G\right]\\
&\qquad + \lim_{t\to\infty}\frac{K^2}{(1+K)^2}\E_{}\left[\prod_{u\in\S} z_u^{L_u(t)}
~\Big|~L(t)=\delta_{(i,D)}+\delta_{(j,D)}, i,j\in\G\right]\\
&\quad <\frac{\theta_x}{(1+K)^2}+\frac{K\theta_x+K\theta_y}{(1+K)^2}+\frac{K^2\theta_y}{(1+K)^2}
=\frac{\theta_x+K\theta_y}{1+K} = \theta.
\end{aligned} 
\end{equation}
Here, the strict inequality follows from the non-trivial invariant initial distribution (ruling out $z \equiv 0$ and $z \equiv 1$), together with the fact that the swapping between active and dormant is driven by a positive recurrent Markov chain on $\{A,D\}$. Hence \eqref{defclust2} holds if and only if $\lim_{n\to\infty} \P_{\delta_u+\delta_v}(|L(t)|=2|\,|L(0)|=2)=0$ for every $u,v\in\S$. Therefore every pair of lineages coalesces with probability $1$. 

Thus, we have proved Theorem~\ref{T.dichcrit.model1}.  

\paragraph{Model 2: Proof of Theorem~\ref{T.dichcrit.model2}.}

$\mbox{}$

\paragraph{{\bf Case $\rho<\infty$.}} 
Like for model 1, we define
\begin{equation}
\theta_x = \E_{\mu(0)}[x_0], \qquad 
\theta_{y,m} = \E_{\mu(0)}[y_{0,m}], \qquad 
\theta = \frac{\theta_x+\sum_{m=0}^\infty K_m\theta_{y,m}}{1+\rho}.
\end{equation}
For $\rho<\infty$, a lineage in the dual moves as a positive recurrent Markov chain on $\{A,(D_m)_{m\in\N_0}\}$. Therefore the argument for ``$\boldsymbol{\Longleftarrow}$'' given for model 1 goes through via the duality relation, which gives
\begin{equation}
\label{limd}
\lim_{t\to \infty}\E\left[\prod_{u\in\S}z_u(t)^{l_u}\right]
=\lim_{t\to \infty}\E\left[\prod_{u\in\S}z_u^{L_u(t)}\right]
=\frac{\theta_x+\sum_{m\in\N_0}K_m\theta_{y,m}}{1+\sum_{m\in\N_0} K_m}.
\end{equation} 
With the duality relation in \eqref{e401b2}, the argument for ``$\boldsymbol{\Longrightarrow}$'' given for model 1 also goes through directly.

\paragraph{{\bf Case $\rho=\infty$.}}
For $\rho=\infty$, a lineage in the dual moves as a null-recurrent Markov chain, which has no stationary distribution, and so \eqref{limd} does not carry over. However, from \cite[Section 3]{Lin92} it follows that, for all $u_1,u_2\in \S$,
\begin{equation}\label{mo4714}
\lim_{t\to \infty} \big\|\P_{u_1}(L(t)=\delta_{(\cdot)}\mid L(t)=1)-\P_{u_2}(L(t)=\delta_{(\cdot)}\mid L(t)=1)\big\|_{tv}=0.
\end{equation}
Moreover, by null-recurrence,
\begin{equation}
\label{nulrec}
\begin{aligned}
\lim_{t\to\infty} \P(L(t) = \delta_{(\cdot,A)}) &= 0,\\
\lim_{t\to\infty} \P(L(t) = \delta_{(\cdot,D_m)}) &= 0 \qquad \forall m\in\N_0,\\ 
\lim_{t \to \infty} \sum_{m=M}^\infty \P(L(t) = \delta_{(\cdot,D_m)}) &=1 \qquad \forall M\in\N_0.
\end{aligned}
\end{equation}

\medskip\noindent
``$\boldsymbol{\Longleftarrow}$''
By duality, we have
\begin{equation}
\label{mo4714a}
\lim_{t\to \infty}\E\left[\prod_{u\in\S}z_u(t)^{l_u}\right]
= \lim_{t\to \infty}\E\left[\prod_{u\in\S}z_u^{L_u(t)}\right]
= \lim_{t\to\infty} \left[\theta_x \P(L(t) = \delta_{(\cdot,A)}) 
+ \sum_{m\in\N_0} \theta_{y,m} \P(L(t) = \delta_{(\cdot,D_m)})\right],
\end{equation} 
where we follow an argument similar as in \eqref{dualarg} and use that $\P(T<\infty)=1$. Because the initial measure is colour regular, we know that $\lim_{m\to\infty} \theta_{y,m} = \theta$ (recall Definition~\ref{D.regular}). But \eqref{nulrec}--\eqref{mo4714a} imply that all moments tend to $\theta$. In particular, 
\begin{equation}
\lim_{t\to \infty}\E[x_i(t)] = \theta = \lim_{t\to \infty}\E[y_{i,m}(t)], \qquad i \in \G,\, m\in\N_0.
\end{equation}

\medskip\noindent
``$\boldsymbol{\Longrightarrow}$'' 
By the duality relation in \eqref{e401b2} and the assumption of clustering, we find 
\begin{equation}
\label{cluscon2}
\lim_{t\to \infty}\E\left[z_u(t)(1-z_v(t))\right]=0\qquad \forall u,v\in\S.
\end{equation}   
Therefore
\begin{equation}
\begin{aligned}
&\lim_{t\to\infty}\E\left[z_u(t)(1-z_v(t))\right]\\
&=\lim_{t\to \infty} \E_{\mu(0)}\left[\E_{\delta_u}\left[\prod_{u\in\S}z_u^{L_u(t)}\right]\right]
- \E_{\mu(0)}\left[\E_{\delta_u+\delta_v}\left[\prod_{u\in\S}z_u^{L_u(t)}\right]\right]\\
&={\theta}\left[1-\lim_{t\to \infty}\P_{\delta_u+\delta_v}\left(|L(t)|=1\right)\right]\\
&\qquad -\lim_{t\to\infty} \E_{\mu(0)}\left[\E_{\delta_u+\delta_v}
\left[\prod_{u\in\S} z_u^{L_u(t)} ~\Big|~ |L(t)|=2\right]\right]\P_{\delta_u+\delta_v}(|L(t)|=2)
=0.
\end{aligned}
\end{equation}
Suppose that $\lim_{t\to\infty} \P_{\delta_u+\delta_v}(|L(t)|=2) \neq 0$. Then $\lim_{t\to\infty} \E_{\delta_u+\delta_v}\left[\prod_{u\in\S }z_u^{L_u(t)}\mid |L(t)|=2\right]=\theta$. However,
\begin{eqnarray}
\lim_{t\to\infty}\E_{\mu(0)}\left[\E_{\delta_u+\delta_v}\left[\prod_{u\in\S} z_u^{L_u(t)} ~\Big|~ |L(t)|=2\right]\right]
<\E_{\mu(0)}\left[\E_{\delta_u+\delta_v}\left[\prod_{u\in\S} z_u^{L_u(t)} ~\Big|~ |L(t)|=1\right]\right]=\theta,
\end{eqnarray}
because we start from a nontrivial stationary distribution. 

Thus, we have proved Theorem~\ref{T.dichcrit.model2}.

\paragraph{Model 3: Proof of Theorem~\ref{T.dichcrit.model2}.}

Since the duality relation for model 3 is exactly the same as for model 2, the same results hold by translation invariance and the extra displacement does not affect the dichotomy criterion. 

\end{proof}

%%%

\subsection{Outline remainder of paper}

In Sections~\ref{s.model1}--\ref{s.model3} we prove Theorems~\ref{T.ltb1}, \ref{T.ltb2} and \ref{T.ltb3}, respectively. For  each of the three models we split the proof into four parts:
\begin{enumerate}
\item Moment relations.
\item The clustering case.
\item The coexistence case. 
\item Proof of the dichotomy.
\end{enumerate}

%%%%%%%%%% SECTION 5 %%%%%%%%%%%%%%%%%%%%%%%%%%%

\section{Proofs: Long-time behaviour for Model 1}
\label{s.model1}

In Section~\ref{ss.momrels} we relate the first and second moments of the process $(Z(t))_{t\geq 0}$ in \eqref{gh1}--\eqref{gh2} to the random walk with internal states $\{A,D\}$ that evolves according to the transition kernel $b^{(1)}(\cdot,\cdot)$ given in \eqref{mrw} (Lemma~\ref{lem1cg94} below). These moment relations hold for all $g\in\CG$. In Section~\ref{ss.clustcase} we deal with the clustering case (Lemmas~\ref{lem:clusgisgfw}--\ref{lem:comparison} below), in Section \ref{ss.cos} with the coexistence case (Lemmas~\ref{muinrtheta}--\ref{lem10} below). In Section~\ref{ss.dichmodel1} we prove Theorem~\ref{T.ltb1}. In Sections \ref{ss.clustcase} and \ref{ss.cos} we will see that the moment relations are crucial when no duality is available. 

Below we write $\E_z$ for $\E_{\delta_z}$, the expectation when the process starts from the initial distribution $\delta_z$, $z \in E$.

%%%
\subsection{Moment relations}
\label{ss.momrels}

\begin{lemma}{\bf [First and second moment]}
\label{lem1cg94}
For $z\in E$, $t\geq 0$ and $(i,R_i),(j,R_j)\in \G\times\{A,D\}$, 
\begin{equation}
\label{eqexprw}
\E_z[z_{(i,R_i)}(t)]=\sum_{(k,R_k)\in \G\times\{A,D\}} b^{(1)}_t\big((i,R_i),(k,R_k)\big)\,z_{(k,R_k)}
\end{equation}
and 
\begin{equation}
\label{eqexprw2}
\begin{aligned}
&\E_z[z_{(i,R_i)}(t)z_{(j,R_j)}(t)]
=\sum_{(k,R_k),(l,R_l)\in \G\times\{A,D\}} b^{(1)}_t\big((i,R_i),(k,R_k)\big)\,b^{(1)}_t\big((j,R_j),(l,R_l)\big)\,
z_{(k,R_k)}z_{(l,R_l)}\\ 
& +\,2\int_{0}^{t} \d s \sum_{k\in \G} 
b^{(1)}_{(t-s)}((i,R_i),(k,A))\,b^{(1)}_{(t-s)}((j,R_j),(k,A))\,\E_z[g(x_k(s))].
\end{aligned}
\end{equation}
\end{lemma}

\begin{proof}
We derive systems of differential equations for the moments and solve these in terms of the random walk. Let $(RW_t)_{t \geq 0}$ denote the semigroup of the random walk with transition kernel $b^{(1)}(\cdot,\cdot)$, and recall that the corresponding generator is given by
\begin{equation}
(G_{RW} f)(i,R_i)=\sum_{(j,R_j)\in\G\times\{A,D\}} b^{(1)}\big((i,R_i),(j,R_j)\big)\left[f(j,R_j)-f(i,R_i)\right].
\end{equation}
Applying the generator \eqref{eq402} of the system in \eqref{gh1}--\eqref{gh2} to the function $f_{(i,R_i)}\colon\, E\to\R$,  $f_{(i,R_i)}(z)=z_{(i,R_i)}$, we obtain by standard stochastic calculus
\begin{equation}
\label{genonz}
\begin{aligned}
\frac{\d\E_z[z_{(i,R_i)}(t)]}{\d t}
&= \left[ \sum_{j\in\G} a(i,j)\,\big(\E_z[x_j(t)]-\E_z[x_i(t)])+Ke\,(\E_z[y_i(t)]-\E_z[x_i(t)]\big)\right]\,1_{(R_i=A)}\\
&\qquad + e\,\big(\E_z[x_i(t)]-\E_z[y_i(t)]\big)\,1_{(R_i=D)}.
\end{aligned}
\end{equation} 
Hence, denoting by $(S_t)_{t \geq 0} $ the semigroup of the system in \eqref{gh1}--\eqref{gh2}, we see from \eqref{genonz} and the definition of $b^{(1)}(\cdot,\cdot)$ in \eqref{mrw} that $(S_tf_{(i,R_i)})$ solves the differential equation
\begin{equation}
\label{Ligeq}
F^\prime(t)=(G_{RW} F)(t).
\end{equation}
On the other hand, for each $f\in \mathcal{C}_b(\G\times\{A,D\})$, $RW_tf$ also solves \eqref{Ligeq}. In particular, for $z\in\E$ define $f_z\colon\,\G\times\{A,D\}\to \R$  by $f_z(i,R_i)=z(i,R_i)$ for $z\in E$, then $RW_tf_z $ is a solution to \eqref{Ligeq}. Since 
\begin{equation}
(RW_0f_z)(i,R_i)=z(i,R_i)=(S_0f_{(i,R_i)})(z),
\end{equation}
we see that \eqref{eqexprw} holds. To prove \eqref{eqexprw2}, we derive a similar system of differential equations and again solve this in terms of the random walk moving according to the kernel $b(\cdot,\cdot)$. Let $f\colon\,E\to \R$ be given by $f(z)=z_{(i,R_i)}z_{(j,R_j)}$. Using the generator \eqref{eq402}, we obtain via It\^o-calculus that
\begin{equation}
\label{eqtworw}
\begin{aligned}
\frac{\d}{\d t}\E_z[z_{(i,R_i)}(t)z_{(j,R_j)}(t)]
&=\sum_{k\in\G} a(i,k)\,\big(\E_z[x_k(t)z_{(j,R_j)}(t)]-\E_z[x_i(t)z_{(j,R_j)}(t)]\big)\,1_{\{R_i= A\}}\\
&\quad + Ke\,\big(\E_z[y_i(t)z_{(j,R_j)}(t)]-\E_z[x_i(t)z_{(j,R_j)}(t)]\big)\,1_{\{R_i=A\}}\\
&\quad + e\, \big(\E_z[x_i(t)z_{(j, R_j)}(t)]-\E_z[y_i(t)z_{(j,R_j)}(t)]\big)\,1_{\{R_i= D\}}\\
&\quad + \sum_{l\in\G}a(j,l)\,\big(\E_z[x_l(t)z_{(i,R_i)}(t)]-\E_z[x_j(t)z_{(i,R_i)}(t)]\big)\,1_{\{R_j= A\}}\\
&\quad + Ke\,\big(\E_z[y_j(t)z_{(i,R_i)}(t)]-\E_z[x_j(t)z_{(i,R_i)}(t)]\big)\,1_{\{R_j=A\}}\\
&\quad + e\,\big(\E_z[x_j(t)z_{(i, R_i)}(t)]-\E_z[y_j(t)z_{(i,R_i)}(t)]\big)\,1_{\{R_j=D\}}\\	
&\quad + 2\, \E_z[g(x_i(t))]\, 1_{\{i=j\}}\,1_{\{R_i=R_j=A\}}.
\end{aligned}
\end{equation} 
Let  $U$ be the generator of two independent random walks each moving with transition kernel $b^{(1)}(\cdot,\cdot)$, i.e., for all $h\in\mathcal{C}_b((\G\times\{A,D\})^2)$,
\begin{equation}
\begin{aligned}
(U h)((i,R_i),(j,R_j))
&=\sum_{k\in\G} a(i,k)\,\big[h((k,A),(j,R_j))-h((i,R_i),(j,R_j))\big]\,1_{\{i,R_i= A\}}\\
&\quad + Ke\,\big[h((i,D),(j,R_j))- h((i,R_i),(j,R_j))\big]\,1_{\{i,R_i= A\}}\\
&\quad + e\, \big[h((i,A),(j,R_j))-h((i,R_i),(j,R_j))\big]\,1_{\{i,R_i= D\}}\\
&\quad + \sum_{l\in\G} a(j,l)\,\big[h((i,R_i),(l,A))-h((i,R_i),(j,R_j))\big]\,1_{\{R_j= A\}}\\
&\quad + Ke\,\big[h((i,R_i),(j,D))-h((i,R_i),(j,R_j))\big]\,1_{\{R_j= A\}}\\
&\quad + e\, \big[h((i,R_i),(j,A))-h((i,R_i),(j,D))\big]\,1_{\{R_j= D\}}.
\end{aligned}
\end{equation}
Let $F(t)=\E_z[z_{(i,R_i)}(t)z_{(j,R_j)}(t)]$ and  $H(t)= 2 \E_z[g(x_i(t))] 1_{\{i=j\}}1_{\{R_i=R_j=A\}}$. Then we can rewrite \eqref{eqtworw} as
\begin{equation}
\frac{\d}{\d t}F(t)=(U F)(t)+H(t). 
\end{equation}
Denote by $(RW_t^{(2)})_{t \geq 0}$ the semigroup corresponding to $U$.  Applying \cite[Theorem I.2.15]{Lig85}, we obtain
\begin{equation}
F(t)=RW_t^{(2)}F(0)+\int_{0}^{t}\d s\,RW_{t-s}^{(2)}H(s).
\end{equation}
Hence
\begin{equation}
\begin{aligned}
\E_z[z_{(i,R_i)}(t)z_{(j,R_j)}(t)]
&= \sum_{(k,R_k),(l,R_l)\in\G\times\{A,D\}}b^{(1)}_t\big((i,R_i),(k,R_k)\big)\,b^{(1)}_t\big((j,R_j),(l,R_l)\big)\,
\E_z[z_{(k,R_k)}z_{(l,R_l)}]\\
&\qquad + 2\int_{0}^{t} \d s\,\sum_{k\in\G}
b^{(1)}_{t-s}\big((i,R_i),(k,A)\big)\,b^{(1)}_{t-s}\big((j,R_j),(k,A)\big)\,\E_z[g(x_k(s))].
\end{aligned}
\end{equation}
\end{proof}

\begin{remark}{\bf [Density]}
\label{rem:prob}
{\rm From Lemma \ref{lem1cg94} we obtain that if $\mu$ is a translation invariant measure such that $\E_\mu[x_0(0)]=\theta_x$ and $\E_\mu[y_0(0)]=\theta_y$, then  
\begin{equation}
\label{e5.52}
\E_\mu[z_{(i,R_i)}(t)]=\theta_x\sum_{(k,R_k)\in \G\times\{A\}} b^{(1)}_t\big((i,R_i),(k,R_k)\big)
+\theta_y\sum_{(k,R_k)\in \G\times\{D\}} b^{(1)}_t\big((i,R_i),(k,R_k)\big),
\end{equation}	
in particular, $\lim_{t\to\infty}\E_\mu[z_{(i,R_i)}(t)]=\frac{\theta_x+K\theta_y}{1+K}=\theta$, recall \eqref{thetadef}, and 
\begin{equation}
\label{e5.53}
\begin{aligned}
&\E_\mu[z_{(i,R_i)}(t)z_{(j,R_j)}(t)]\\
&= \sum_{(k,R_k),(l,R_l)\in \G\times\{A,D\}} b^{(1)}_t\big((i,R_i),(k,R_k)\big)\,b^{(1)}_t\big((j,R_j),(l,R_l)\big)\,
\E_\mu[z_{(k,R_k)}z_{(l,R_l)}]\\ 
&\qquad +\,2\int_{0}^{t}\d s \sum_{k\in \G} 
b^{(1)}_{t-s}\big((i,R_i),(k,A)\big)\,b^{(1)}_{t-s}\big((j,R_j),(k,A)\big)\,\E_\mu[g(x_i(s))].
\end{aligned}
\end{equation}
} \hfill $\Box$
\end{remark}

\begin{remark}{\bf [First moment duality]}
{\rm Note that \eqref{eqexprw} shows that even for general $g\in\CG$ there is a \emph{first moment duality} between the process $Z(t)$ and the random walk $RW(t),$ that moves according to the kernel $b^{(1)}(\cdot,\cdot)$. The duality function is given by
\begin{equation}
H:E\times \G\times\{A,D\}\to\R,\qquad H(z,(i,R_i))=z_{(i,R_i)}.
\end{equation}	
Equation \eqref{eqexprw} in Lemma \ref{lem1cg94} tells us that $\E[H(Z(t),RW(0))]=\E[H(Z(0),RW(t))]$.
 }
\end{remark}

%%%%%%%%%%%%

\subsection{The clustering case}
\label{ss.clustcase}

The proof that the system in \eqref{gh1}--\eqref{gh2} converges to a unique trivial equilibrium when $\hat{a}(\cdot,\cdot)$ is recurrent goes as follows. We first consider the case where $g = dg_{\text{FW}}$, for which \emph{duality} is available (Lemma~\ref{lem:clusgisgfw}). Afterwards we use a \emph{duality comparison argument} to show that the dichotomy between coexistence and clustering does not depend on the choice of $g \in \CG$ (Lemma~\ref{lem:comparison}).

\paragraph{$\bullet$ Case $g = dg_{\text{FW}}$.}

\begin{lemma}{\bf [Clustering]}
\label{lem:clusgisgfw}
Suppose that $\mu(0)\in\CT_\theta^{\mathrm{erg}}$ and $g=dg_{\mathrm{FW}}$. Moreover, suppose that $\hat{a}(\cdot,\cdot)$ defined in \eqref{gh11} is recurrent, i.e., $I_{\hat{a}}=\infty$. Let $\mu(t)$ be the law at time $t$ of the system defined in \eqref{gh1}--\eqref{gh2}. Then  
\begin{equation}
\label{gh12p}
\lim_{t\to\infty} \mu(t)
= \theta\, [\delta_{(1,1)}]^{\otimes \G} + (1-\theta)\, [\delta_{(0,0)}]^{\otimes \G}.
\end{equation}
\end{lemma}

\begin{proof}
Since $g=dg_{\mathrm{FW}}$, we can use duality. By the dichotomy criterion in Theorem \ref{T.dichcrit.model1}, it is enough to show that in the dual two partition elements coalesce with probability 1. Recall from Section \ref{ss.duality} that each of the partition elements in the dual moves according to the transition kernel $b^{(1)}(\cdot,\cdot)$ on $\G \times \{A,D\}$ defined by \eqref{mrw} (see Fig.~\ref{fig:dualrw}). Recall from Section \eqref{ss.duality} that $b^{(1)}(\cdot,\cdot)$ describes a random walk on $\G$ with migration rate kernel $a(\cdot,\cdot)$ that becomes dormant (state $D$) at rate $Ke$ (after which it stops moving), and becomes active (state $A$) at rate $e$ (after which it can move again). When two partition elements in the dual are active and are at the same site, they coalesce at rate $d$, i.e., each time they are active and meet at the same site they coalesce with probability $d/[\sum_{j\in\Z^d} a(i,j)+Ke+d]>0$. Hence, in order to show that two partition elements coalesce with probability $1$, we have to show that \emph{with probability $1$ two partition elements meet infinitely often while being active}. The latter holds if and only if the expected total time the random walks spend together at the same colony while being active is infinite. We will show that this occurs if and only the random walk with \emph{symmetrised} transition rate kernel $\hat{a}(\cdot,\cdot)$ is recurrent. The proof comes in 4 Steps. 

\medskip\noindent 
{\bf 1.\ Active and dormant time lapses.}
Consider two copies of the random walk with kernel $b^{(1)}(\cdot,\cdot)$, both starting at 0 and in the active state.  Let 
\begin{equation}
\label{timelapses}
\begin{array}{lll}
&(\sigma_k)_{k\in\N}, &(\sigma'_k)_{k\in\N},
\end{array}
\end{equation}
denote the successive time lapses during which they are active and let
\begin{equation}
\label{timelapsesb}
\begin{array}{lll}
&(\tau_k)_{k\in\N}, &(\tau'_k)_{k\in\N},
\end{array}
\end{equation}
denote the successive time lapses during which they are dormant (see Fig.~\ref{fig:periods}). These are mutually independent sequences of i.i.d.\ random variables with marginal laws
\begin{equation}
\label{e1955}
\begin{array}{llllll}
\mathbb{P}(\sigma_1>t) &=& \mathbb{P}(\sigma'_1>t) &=& \e^{-Ke\,t}, &t \geq 0,\\
\mathbb{P}(\tau_1>t) &=& \mathbb{P}(\tau'_1>t) &=& \e^{-e\,t} &t \geq 0,
\end{array}
\end{equation}
where we use the symbol $\mathbb{P}$ to denote the joint law of the two sequences. 

%%%%%%%%%%%%%% FIGURE %%%%%%%%%%%%%%%%%
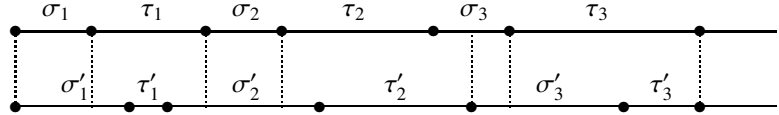
\begin{figure}[htbp]
\begin{center}
\setlength{\unitlength}{.5cm}
%%%
\begin{picture}(20,2)(0,0)
\put(0,0){\line(20,0){20}}
\put(0,-1.98){\line(20,0){20}}
%%%
\put(0,0){\circle*{.25}}
\put(2,0){\circle*{.25}}
\put(5,0){\circle*{.25}}
\put(7,0){\circle*{.25}}
\put(11,0){\circle*{.25}}
\put(13,0){\circle*{.25}}
\put(18,0){\circle*{.25}}
%%%
\put(0,-2){\circle*{.25}}
\put(3,-2){\circle*{.25}}
\put(4,-2){\circle*{.25}}
\put(8,-2){\circle*{.25}}
\put(12,-2){\circle*{.25}}
\put(16,-2){\circle*{.25}}
\put(18,-2){\circle*{.25}}
%%%
\put(.7,.4){$\sigma_1$}
\put(3.3,.4){$\tau_1$}
\put(5.7,.4){$\sigma_2$}
\put(8.6,.4){$\tau_2$}
\put(11.7,.4){$\sigma_3$}
\put(15,.4){$\tau_3$}
%%%
\put(1.2,-1.6){$\sigma'_1$}
\put(3.2,-1.6){$\tau'_1$}
\put(5.7,-1.6){$\sigma'_2$}
\put(9.7,-1.6){$\tau'_2$}
\put(13.7,-1.6){$\sigma'_3$}
\put(16.7,-1.6){$\tau'_3$}
%%%
\qbezier[15](0,-2)(0,-1)(0,0)
\qbezier[15](2,-2)(2,-1)(2,0)
\qbezier[15](5,-2)(5,-1)(5,0)
\qbezier[15](7,-2)(7,-1)(7,0)
\qbezier[15](12,-2)(12,-1)(12,0)
\qbezier[15](13,-2)(13,-1)(13,0)
\qbezier[15](18,-2)(18,-1)(18,0)
%%%
\end{picture}
%%%
\vspace{1.5cm}
\caption{\small Successive periods during which the two random walks are active and dormant. The time lapses between the dotted lines represent periods of joint activity.}
\label{fig:periods}
\end{center}
\end{figure}
%%%%%%%%%%%%%%%%%%%%%%%%%%%%%%%%%%%%%%%%

Let $a_t(\cdot,\cdot)$ denote the time-$t$ transition kernel of the random walk with migration kernel $a(\cdot,\cdot)$. Let
\begin{equation}
\label{Edefs}
\begin{aligned}
\CE(k,t) &= \left\{\sum_{\ell=1}^k (\sigma_\ell + \tau_\ell) \leq t 
< \sum_{\ell=1}^k (\sigma_\ell + \tau_\ell) + \sigma_{k+1}\right\},\\
\CE'(k',t) &= \left\{\sum_{\ell=1}^{k'} (\sigma'_\ell + \tau'_\ell) \leq t 
< \sum_{\ell=1}^{k'} (\sigma'_\ell + \tau'_\ell) + \sigma'_{k+1}\right\},
\end{aligned}
\end{equation}
be the events that the random walks are active at time $t$ after having become dormant and active exactly $k,k'$ times, and let
\begin{equation}
\label{Tdefs}
\begin{aligned}
T(k,t) 
&= \sum_{\ell=1}^k \sigma_\ell + \left(\left(t-\sum_{\ell=1}^k (\sigma_\ell + \tau_\ell)\right)\wedge \sigma_{k+1}\right),\\
T'(k',t) 
&= \sum_{\ell=1}^{k'} \sigma'_\ell + \left(\left(t-\sum_{\ell=1}^{k'} (\sigma'_\ell + \tau'_\ell)\right)\wedge \sigma_{k+1}\right),
\end{aligned}
\end{equation}
be the total accumulated activity times of the random walks on the events in \eqref{Edefs}. Note that the terms between brackets in \eqref{Tdefs} are at most $\sigma_{k+1}$, respectively, $\sigma'_{k'+1}$, and therefore are negligible as $k,k'\to\infty$.

Given the outcome of the sequences in \eqref{timelapses}--\eqref{timelapsesb}, the probability that at time $t$ both random walks are active and are at the same colony equals
\begin{equation}
\label{same1}
\sum_{k,k'\in\N} \left(\sum_{i\in\G} a_{T(k,t)}(0,i)\,a_{T'(k',t)}(0,i)\right)\,1_{\CE(k,t)}\,1_{\CE(k',t)}, 
\end{equation}    
Therefore the expected total time the random walks are active and are at the same colony equals
\begin{equation}
\label{same2}
I = \int_0^\infty \d t\,\sum_{k,k'\in\N} \mathbb{E}_{(0,A),(0,A)}\left[\left(\sum_{i \in \G} 
a_{T(k,t)}(0,i)\,a_{T'(k',t)}(0,i)\right)1_{\CE(k,t)}\,1_{\CE'(k',t)}\right], 
\end{equation}    
where $\mathbb{E}$ is the expectation over the sequences in \eqref{timelapses}. Let
\begin{equation}
\label{e2056}
N(t) = \max\left\{k\in\N\colon\,\sum_{\ell=1}^k (\sigma_\ell+\tau_\ell) \leq t\right\}, 
\qquad N'(t) = \max\left\{k'\in\N\colon\,\sum_{\ell=1}^{k'} (\sigma_\ell+\tau_\ell) \leq t\right\},  
\end{equation}
be the number of times the random walks have become dormant and active up to time $t$. Let
\begin{equation}
\label{e2061}
T(t) = T(N(t),t), \quad T'(t) = T'(N'(t),t), \quad \CE(t) = \CE(N(t),t), \quad \CE'(t) = \CE'(N'(t),t),
\end{equation}
be the total accumulated activity times of the random walks up to time $t$, respectively, the events that the random walks are active at time $t$. Then we may write
\begin{equation}
\label{same3}
I = \int_0^\infty \d t\,\,\mathbb{E}_{(0,A),(0,A)}\left[\left(\sum_{i\in\G} a_{T(t)}(0,i)\,a_{T'(t)}(0,i)\right)
1_{\CE(t)}\,1_{\CE'(t)}\right]. 
\end{equation}
We know that coalescence occurs with probability 1 if and only if $I=\infty$. 

\medskip\noindent
{\bf 2.\ Fourier analysis.} 
Define
\begin{equation}
\label{same4}
M(t) = T(t) \wedge T'(t), \qquad \Delta(t) = [T(t) \vee T'(t)]  - [T(t) \wedge T'(t)].
\end{equation}
Then
\begin{equation}
\label{same5} 
\sum_{i\in\G} a_{T(t)}(0,i)\,a_{T'(t)}(0,i) 
= \sum_{j\in\G} \hat{a}_{2M(t)}(0,j)\,a_{\Delta(t)}(j,0).
\end{equation}
Indeed, the difference of the two random walks at time $M(t)$ has distribution $\hat{a}_{2M(t)}(0,\cdot)$, and in order for the random walk with the largest activity time to meet the random walk with the smallest activity time at time $2M(t)+\Delta(t)$, it must bridge this difference in time $\Delta(t)$. To work out \eqref{same5}, we assume without loss of generality that $\sum_{j \in \G} a(0,j) =1$, and use \emph{Fourier analysis}. For ease of exposition we focus on the special case where $\G = \Z^d$, but the argument below extends to \emph{any} countable Abelian group endowed with the discrete topology, because these properties ensure that there is a version of Fourier analysis on $\G$ \cite[Section 1.2]{Ru62}. For $\phi \in [-\pi,\pi]^d$, define
\begin{equation}
\label{same6}
a(\phi) = \sum_{j \in \Z^d} \e^{\Im (\phi,j)} a(0,j), \qquad \hat{a}(\phi) = \mathrm{Re}\,a(\phi),
\qquad \tilde{a}(\phi) = \mathrm{Im}\,a(\phi).
\end{equation}
Then
\begin{equation}
\label{same7}
\begin{aligned}
\hat{a}_t(0,j) &= \frac{1}{(2\pi)^d} \int_{[-\pi,\pi]^d} \d\phi\,\e^{-\Im (\phi,j)}\,\e^{-t[1-\hat{a}(\phi)]},\\
a_t(j,0) &= \frac{1}{(2\pi)^d} \int_{[-\pi,\pi]^d} \d\phi'\,\e^{\Im (\phi',j)}\,
\e^{-t[1-\hat{a}(\phi')-\Im\tilde{a}(\phi')]},
\end{aligned}
\end{equation}
where we use that $a(\phi) = \hat{a}(\phi) + \Im \tilde{a}(\phi)$. Inserting these representations into \eqref{same5}, we get
\begin{equation}
\label{same8} 
\sum_{i\in\Z^d} a_{T(t)}(0,i)\,a_{T'(t)}(0,i) 
= \frac{1}{(2\pi)^d} \int_{[-\pi,\pi]^d} \d\phi\, 
\e^{-[2M(t)+\Delta(t)]\,[1-\hat{a}(\phi)]}\,\cos(\Delta(t) \tilde{a}(\phi)),
\end{equation}
where we use that $\sum_{j\in\Z^d} \e^{\Im(\phi'-\phi,j)} =  (2\pi)^d\delta(\phi'-\phi)$, with $\delta(\cdot)$ the Dirac distribution (Folland~\cite[Chapter 7]{F92}). 

\medskip\noindent
{\bf 3.\ Limit theorems.}
By the strong law of large numbers, we have
\begin{equation}
\label{SLLN}
\lim_{k\to\infty} \frac{1}{k} \sum_{\ell=1}^k \sigma_\ell = \frac{1}{Ke} \quad \P\text{-a.s.},
\qquad  
\lim_{k\to\infty} \frac{1}{k} \sum_{\ell=1}^k \tau_\ell = \frac{1}{e} \quad \P\text{-a.s.}
\end{equation}
Therefore, by the standard renewal theorem (Asmussen~\cite[Chapter I, Theorem 2.2]{A03}),
\begin{equation}
\label{ETasymp}
\begin{aligned}
&\lim_{t\to\infty} \frac{1}{t}\,N(t) = \lim_{t\to\infty} \frac{1}{t}\,N'(t) = A \quad \P\text{-a.s.},\\
&\lim_{t\to\infty} \frac{1}{t}\,T(t) =  \lim_{t\to\infty} \frac{1}{t}\,T'(t) = B \quad \P\text{-a.s.},\\[0.2cm]
&\lim_{t\to\infty} \mathbb{P}\big(\CE(t)\big) = \lim_{t\to\infty} \mathbb{P}\big(\CE'(t)\big) = B, 
\end{aligned}
\end{equation}
with
\begin{equation}
\label{e2138}
A = \frac{1}{\frac{1}{Ke}+\frac{1}{e}} = \frac{K}{1+K}\,e, \qquad 
B= \frac{\frac{1}{Ke}}{\frac{1}{Ke}+\frac{1}{e}} = \frac{1}{1+K}.
\end{equation}
Moreover,  by the central limit theorem, we have
\begin{equation}
\label{CLT}
\left(\frac{T(t)-Bt}{c\sqrt{t}},\frac{T'(t)-Bt}{c\sqrt{t}}\right) 
\quad \Longrightarrow \quad (Z,Z') \quad \text{in $\mathbb{P}$-distribution as $t\to\infty$}
\end{equation}
with $(Z,Z')$ independent standard normal random variables and 
\begin{equation}
\label{e2149}
c^2 = A\left[(1-B)^2\,\mathbb{V}\mathrm{ar}(\sigma_1)
+B^2\,\mathbb{V}\mathrm{ar}(\tau_1) \right]
\end{equation}
(see \cite{Smith55} or \cite[Theorem VI.3.2]{A03}). Since $T(t),\CE(t)$ and $T'(t),\CE'(t)$ are independent, and each pair is asymptotically independent as well, we find that 
\begin{equation}
\label{Fourier}
\mathbb{E}_{(0,A),(0,A)}\left[ \left(\sum_{i\in\Z^d} a_{T(t)}(0,i)\,a_{T'(t)}(0,i)\right)1_{\CE(t)}\,1_{\CE'(t)} \right]
\sim B^2 f(t), \qquad t \to \infty,
\end{equation}
with
\begin{equation}
\label{fdef}
\begin{aligned}
f(t) &= \frac{1}{(2\pi)^d} \int_{[-\pi,\pi]^d} \d\phi\, 
\e^{-[1+o(1)]\,2Bt\,[1-\hat{a}(\phi)]}\,
\mathbb{E}\left[\cos\Big([1+o(1)]\,c (Z-Z')\sqrt{t}\,\tilde{a}(\phi)\Big)\right]\\
&= \frac{1}{(2\pi)^d} \int_{[-\pi,\pi]^d} \d\phi\, 
\e^{-[1+o(1)]\,2Bt\,[1-\hat{a}(\phi)]}\,
\e^{-[1+o(1)]\,c^2t\,\tilde{a}(\phi)^2},
\end{aligned}
\end{equation}
where we use that $\cos$ is symmetric, $Z-Z' = \sqrt{2}\,Z''$ in $\mathbb{P}$-distribution with $Z''$ standard normal, and $\mathbb{E}(\e^{\Im\mu Z''}) = \e^{-\mu^2/2}$, $\mu \in \R$. From \eqref{same3} and \eqref{Fourier} we have that $I<\infty$ if and only if $t \mapsto f(t)$ is integrable. By Cram\'er's theorem, deviations of $T(t)/t$ and $T'(t)/t$ away from $B$ are exponentially costly in $t$. Hence the error terms in \eqref{fdef}, arising from \eqref{ETasymp} and \eqref{CLT}, do \emph{not} affect the integrability of $t \mapsto f(t)$. Note that, because $a(\cdot,\cdot)$ is assumed to be \emph{irreducible} (recall \eqref{gh10}), $\hat{a}(\phi) = 1$ if and only if $\phi=0$. Hence the integrability of $t \mapsto f(t)$ is determined by the behaviour of $\hat{a}(\phi)$ and $\tilde{a}(\phi)$ as $\phi \to 0$. 

\medskip\noindent
{\bf 4.\ Irrelevance of asymmetric part of migration.}
We next observe that $\tilde{a}(\phi)^2 \leq 1-\hat{a}(\phi)^2 \leq 2[1-\hat{a}(\phi)]$. Hence, $t\,\tilde{a}(\phi)^2 \leq 2t\,[1-\hat{a}(\phi)]$. Therefore we see from \eqref{fdef} that for sufficiently large $T\in\R$ we can bound $t \mapsto f(t)$ on $[T,\infty)$ from above and below by functions of the form $t \mapsto g_{C}(t)$ with 
\begin{equation}
\label{same9}
g_C(t) = \frac{1}{(2\pi)^d} \int_{[-\pi,\pi]^d} \d\phi\,\e^{-\,Ct\,[1-\hat{a}(\phi)]},\qquad C \in (0,\infty).
\end{equation}    
From \eqref{same7} we have
\begin{equation}
\label{same10}
g_C(t) = \hat{a}_{C t}(0,0) \asymp \hat{a}_t(0,0),
\end{equation} 
where the last asymptotics uses that $t \mapsto \hat{a}_t(0,0)$ is regularly varying at infinity (recall \eqref{ass2}).  Combining \eqref{same3}, \eqref{Fourier} and \eqref{same9}--\eqref{same10}, we get 
\begin{equation}
\label{ISLLN}
I  = \infty \quad \Longleftrightarrow \quad I_{\hat{a}} = \infty
\end{equation}
with $I_{\hat{a}} = \int_1^\infty \d t\,\hat{a}_t(0,0)$. Thus, if $\hat{a}(\cdot,\cdot)$ is recurrent, then $I=\infty$ and the system clusters. Moreover, we see from the bounds on $f(t)$ (recall \eqref{fdef}) that \emph{ the asymmetric part of the migration kernel has no effect on the integrability}. 

This settles the dichotomy between clustering and coexistence when $g = g_{\text{FW}}$.
\end{proof}

\paragraph{$\bullet$ Case $g\neq dg_{\text{FW}}$.}

For $g\neq dg_{\text{FW}}$ the proof of Lemma~\ref{lem:clusgisgfw} does not go through. However, the \emph{moments relations} in Lemma~\ref{lem1cg94} hold for general $g\in\CG$. Using these moment relations and a technique called duality comparison (see \cite{CG94}), we prove Lemma~\ref{lem:clusgisgfw} for general $g\in\CG$.  

\begin{lemma}{\bf [Duality comparison]}
\label{lem:comparison}
Suppose that $\mu(0)\in\CT_\theta^{\mathrm{erg}}$ and $g\in\CG$. Moreover, suppose that $\hat{a}(\cdot,\cdot)$ defined in \eqref{gh11} is recurrent, i.e., $I_{\hat{a}}=\infty$. Let $\mu(t)$ be the law at time $t$ of the system defined in \eqref{gh1}--\eqref{gh2}. Then  
\begin{equation}
\lim_{t\to\infty} \mu(t) = \theta\, [\delta_{(1,1)}]^{\otimes \G} + (1-\theta)\, [\delta_{(0,0)}]^{\otimes \G}.
\end{equation}
\end{lemma} 

\begin{proof}
We proceed as in the proof of \cite[Theorem]{CG94}. First assume that $\mu(0)=\delta_z$ for some $z\in E$, and satisfies
\begin{equation}
\lim_{t\to\infty}\sum_{(k,R_k)\in \G\times\{A,D\}}b^{(1)}_t((i,R_i),(k,R_k))\,z_{(k,R_k)}=\theta.
\end{equation}
By Lemma \ref{lem1cg94}, we have 
\begin{equation}
\label{m1}
\E_{z}\left[z_{(i,R_i)}(t)\right]=\sum_{(k,R_k)\in \G\times\{A,D\}}b_t^{(1)}((i,R_i),(k,R_k))\,z_{(k,R_k)}.
\end{equation}
Hence, by assumption, for all $(i,R_i)\in\G\times\{A,D\}$ we have
\begin{equation}\label{m1a}
\lim_{t\to\infty}\E_z\left[z_{(i,R_i)}(t)\right]=\theta.
\end{equation}
Since we have clustering if, for all $(i,R_i),(j,R_j)\in\G\times\{A,D\}$, 
\begin{equation}
\lim_{t\to\infty}\E_z\left[z_{(i,R_i)}(t)(1-z_{(j,R_j)}(t))\right]=0,
\end{equation}
we are left to prove that 
\begin{equation}
\label{m2}
\lim_{t\to\infty}\E_z\left[z_{(i,R_i)}z_{(j,R_j)}\right]=\theta.
\end{equation}
Since \eqref{m1a} implies that $\limsup_{t\to\infty}\E_z[z_{(i,R_i)}z_{(j,R_j)}]\leq \theta$, we are left to prove that
\begin{equation}
\label{enough1}
\liminf_{t\to\infty}\E_z[z_{(i,R_i)}z_{(j,R_j)}] \geq \theta.
\end{equation}
Like in \cite{CG94}, we will prove \eqref{enough1} by \emph{comparison duality}.
 	
Fix $\epsilon>0$. Since $g\in\CG$ we can choose a $c=c(\epsilon)>0$ such that $g(x) \geq \tilde{g}(x) = c(x-\epsilon)(1-(x+\epsilon))$, $x\in [0,1]$. Note that $\tilde{g}(x)<0$ for $x \in [0,\epsilon) \cup (1-\epsilon,1]$, so we cannot replace $g$ by $\tilde{g}$ in the SSDE. Instead we use $\tilde{g}$ as an auxiliary function. 

Consider the Markov chain $(B(t))_{t \geq 0}$, with state space $\{1,2\}\times \left(\G\times\{A,D\}\right)\times\left(\G\times\{A,D\}\right)$ and $B(t)=(B_0(t),B_1(t),B_2(t))$, evolving according to
\begin{equation}
\begin{aligned}
(1,(i,R_i),(i,R_i)) &\to\  (1,(k,R_k),(k,R_k)), \quad  \text{ at rate } b^{(1)}((i,R_i),(k,R_k)),\\
(2,(i,R_i),(j,R_j)) &\to
\begin{cases}
(2,(k,R_k),(j,R_j)), &\text{ at rate } b^{(1)}((i,R_i),(k,R_k)),\\
(2,(i,R_i),(l,R_l)), &\text{ at rate } b^{(1)}((j,R_j),(l,R_l)),\\
(1,(i,R_i),(i,R_i)), &\text{ at rate } c1_{\{i=j\}}1_{\{R_i=R_j=A\}}.
\end{cases}
\end{aligned}
\end{equation}
This describes two random walks, evolving independently according to the transition kernel $b^{(1)}(\cdot,\cdot)$, that coalesce at rate $c>0$ when they are at the same site and are active. We put $B_0(t)=1$ when the two random walks have already coalesced by time $t$, and $B_0(t)=2$ otherwise. Let $\P_{(2,(i,R_i),(j,R_j))}$ denote the law of the Markov Chain $B(t)$ that starts in $(2,(i,R_i),(j,R_j))$. Note that 
\begin{equation}
\P_{(2,(i,R_i),(j,R_j))}\left(B_1(t)=(k,R_k)\right)=b^{(1)}_t((i,R_i),(k,R_k)),
\end{equation}
and similarly 
\begin{equation}
\P_{(2,(i,R_i),(j,R_j))}\left(B_2(t)=(l,R_l)\right)=b^{(1)}_t((j,R_j),(l,R_l)).
\end{equation}
Since we have assumed that $\hat{ a}(\cdot,\cdot)$ is recurrent, i.e., $I_{\hat{a}}=\infty$, the two random walks meet infinitely often at the same site while being active and hence coalesce with probability 1. Therefore 
\begin{equation}
\label{P2lim}
\lim_{t \to \infty}\P_{(2,(i,R_i),(j,R_j))}\left(B_0(t)=2\right)=0.
\end{equation}
We can rewrite the SSDE in \eqref{gh1}--\eqref{gh2} in terms of $b^{(1)}(\cdot,\cdot)$, namely, for all $(i,R_i)\in\G\times\{A,D\}$,
\begin{equation}
\label{m3}
\d z_{(i,R_i)}(t)=\sum_{(k,R_k)\in\G\times\{A,D\}}b^{(1)}((i,R_i),(j,R_j))[z_{(j,R_j)}(t)-z_{(i,R_i)}(t)]\,\d t
+\sqrt{g(z_{i,R_i}(t))}\,1_{\{R_i=A\}}\,\d w_i(t).
\end{equation}
Using \eqref{m3} and It\^o-calculus, we obtain
\begin{equation}
\frac{\d \E_z[z_{(i,R_i)}(t)-\epsilon]}{\d t}
=\sum_{(k,R_k)\in\G\times\{A,D\}}b^{(1)}((i,R_i),(k,R_k))\,
\E\left[(z_{(k,R_k)}(t)-\epsilon)-(z_{(i,R_i)}(t)-\epsilon)\right]
\end{equation}
and
\begin{equation}
\begin{aligned}
&\frac{\d \E_z[(z_{(i,R_i)}(t)-\epsilon)(z_{(j,R_j)}(t)+\epsilon)]}{\d t}\\
&=\sum_{(k,R_k)\in\G\times\{A,D\}} b^{(1)}((i,R_i),(k,R_k))\,
\E_z\left[(z_{(j,R_j)}(t)+\epsilon)(z_{(k,R_k)}(t)-\epsilon)-(z_{(j,R_j)}(t)+\epsilon)(z_{(i,R_i)}(t)-\epsilon)\right]\\
&\qquad +\sum_{(l,R_l)\in\G\times\{A,D\}} b^{(1)}((j,R_j),(k,R_k))\,
\E_z\left[(z_{(i,R_i)}(t)-\epsilon)(z_{(l,R_l)}(t)+\epsilon)-(z_{(i,R_i)}(t)-\epsilon)(z_{(j,R_j)}(t)+\epsilon)\right]\\
&\qquad + \E_z\left[c (z_{(i,R_i)}(t)-\epsilon) (1-(z_{(j,R_j)}(t)+\epsilon)) 1_{\{i=j\}}1_{\{R_i=R_j=A\}}\right]\\
&\qquad + \E_z\left[\left(g(z_{(i,R_i)}(t))-\tilde{g}(z_{(i,R_i)}(t))\right)1_{\{i=j\}}1_{\{R_i=R_j=A\}}\right].
\end{aligned}
\end{equation}
For $t\geq 0$, define $F_t\colon\,\{0,1\}\times(\G\times\{A,D\})\times(\G\times\{A,D\})\to\R$ by
\begin{equation}
\begin{aligned}
F_t(1,(i,R_i),(i,R_i))
&=\E_z\left[z_{(i,R_i)}(t)-\epsilon\right]\\ F_t(2,(i,R_i),(j,R_j)),
&=\E_z\left[(z_{(i,R_i)}(t)-\epsilon)(z_{(j,R_j)}(t)+\epsilon)\right],
\end{aligned}
\end{equation}
and $H_t\colon\,\{0,1\}\times(\G\times\{A,D\})\times(\G\times\{A,D\})\to\R$ by
\begin{equation}
\begin{aligned}
H_t(1,(i,R_i),(i,R_i)) &=0,\\
H_t(2,(i,R_i),(j,R_j)) &=\E_z\left[\big(g(z_{(i,R_i)}(t))-\tilde{g}(z_{(i,R_i)}(t))\big)\,1_{\{i=j\}}\,1_{\{R_i=R_j=A\}}\right].
\end{aligned}
\end{equation}
Let $\mathfrak{B}$ denote the generator of $(B(t))_{t \geq 0}$, and let $(V_t)_{t \geq 0}$ the associated semigroup. Then 
\begin{equation}
\frac{\d F_t}{\d t}=\mathfrak{B}F_t+H_t.
\end{equation} 
Hence, by \cite[Theorem I.2.15]{Lig85}, it follows that 
\begin{equation}
F_t=V_tF_0+\int_0^t \d s\,V_{(t-s)}H_s.
\end{equation}
Since $H_t>0$ for all $t\geq 0$, we obtain
\begin{equation}
\begin{aligned}
&F_t(2,(i,R_i),(j,R_j))\geq V_t F_0(2,(i,R_i),(j,R_j))\\
&=\E_{(2,(i,R_i),(j,R_j))}\left[F_0(B(t))\right]\\
&=\E_{(2,(i,R_i),(j,R_j))}\left[F_0(B(t))1_{\{B_0(t)=1\}}+F_0(B(t))1_{\{B_0(t)=2\}}\right]\\
&=\sum_{(k,R_k),(l,R_l)\in \G\times\{A,D\}}\P_{(2,(i,R_i),(j,R_j))}\left[B_0(t)=1,B_1(t)=(k,R_k)\right](z_{(k,R_k)}-\epsilon)\\
&\qquad + \E_{(2,(i,R_i),(j,R_j))}\left[F_0(B(t))1_{\{B_0(t)=2\}}\right]\\
&=\sum_{(k,R_k),(l,R_l)\in \G\times\{A,D\}}\P_{(2,(i,R_i),(j,R_j))}\left[B_1(t)=(k,R_k)\right](z_{(k,R_k)}-\epsilon)\\
&\qquad -\sum_{(k,R_k),(l,R_l)\in \G\times\{A,D\}}
\P_{(2,(i,R_i),(j,R_j))}\left[B_0(t)=2,B_1(t)=(k,R_k)\right](z_{(k,R_k)}-\epsilon)\\
&\qquad + \E_{(2,(i,R_i),(j,R_j))}\left[F_0(B(t))1_{\{B_0(t)=2\}}\right]\\
&\geq \sum_{(k,R_k)\in \G\times\{A,D\}}b^{(1)}_t((i,R_i),(k,R_k))\,(z_{(k,R_k)}-\epsilon)
-(1+\epsilon^2)\,\P_{(2,(i,R_i),(j,R_j))}\left[B_1(t)=2\right].
\end{aligned}
\end{equation}
Hence, by \eqref{P2lim}, we obtain
\begin{equation}
\liminf_{t \to \infty}F_t(2,(i,R_i),(j,R_j)) \geq \liminf_{t \to \infty}\E_z\left[(z_{(i,R_i)}(t)-\epsilon)(z_{(j,R_j)}(t)+\epsilon)\right]\geq\theta-\epsilon^2.
\end{equation}
Letting $\epsilon\downarrow 0$, we get \eqref{m2}.
	
To get rid of the assumption $\mu(0)=\delta_z$, note that for $\mu(0)\in\CT^{\mathrm{erg}}_\theta$ we have (recall Remark \ref{rem:prob}) 
\begin{equation}
\lim_{t \to \infty}\sum_{(k,R_k)\in \G\times\{A,D\}}b_t((i,R_i),(k,R_k))\E_\mu[z_{(k,R_k)}]=\theta.
\end{equation}
Hence, by the above argument,
\begin{equation}
\begin{aligned}
\E_\mu\left[(z_{(i,R_i)}(t)-\epsilon)(z_{j,R_j}(t)+\epsilon)\right]&=\int\E_z\left[(z_{(i,R_i)}(t)-\epsilon)(z_{j,R_j}(t)+\epsilon)\right]\d \mu(z)\\
&\geq \int \sum_{(k,R_k)\in \G\times\{A,D\}}b^{(1)}_t((i,R_i),(k,R_k))(z_{(k,R_k)}-\epsilon)\\
&\qquad\qquad-(1+\epsilon^2)\P_{(2,(i,R_i),(j,R_j))}\left[B_1(t)=2\right]\d \mu(z)
\end{aligned}
\end{equation}
Letting first $t\to\infty$ and then $\epsilon\downarrow 0$, we find that
\begin{equation}
\lim_{t\to\infty}\E_\mu\left[(z_{(i,R_i)}(t)-\epsilon)(z_{j,R_j}(t)+\epsilon)\right]=\theta,
\end{equation}
and, for all $(i,R_i),(j,R_j)\in\G\times\{A,D\}$,
\begin{equation}
\lim_{t\to\infty} \E_\mu\left[z_{(i,R_i)}(t)(1-z_{j,R_j}(t))\right]=0.
\end{equation}	
\end{proof}

%%%%%%%%%%%%%%%%%%%%%%%%%%%%%%%%%%%%%%%%

\subsection{The coexistence case}
\label{ss.cos}

For the coexistence case we proceed as in \cite{CG94} with small adaptations. For the convenience of the reader we have written out the full proof. The proof relies on the moment relations in Lemma~\ref{lem1cg94} and no distinction between $g=dg_{\mathrm{FW}}$ and general $g\in\CG$ is needed. The proof consist of several lemmas (Lemmas~\ref{muinrtheta}--\ref{lem10} below), organised into 4 Steps.  In Step 1 we use the moment relations in Lemma~\ref{lem1cg94} to define a set of measures that are preserved under the evolution. In Step 2 we use coupling to prove that, for each given $\theta$, the system converges to a unique equilibrium. In Step 3 we show that, for each given $\theta$, each initial measure under the evolution converges to an invariant measure. In Step 4 we show that the limiting measure is invariant, ergodic and mixing under translations, and is associated. 

%%%

\paragraph{1. Properties of measures preserved under the evolution.}

Let $\theta$ be defined as in \eqref{thetadef} such that $\theta=\E_{\mu(0)}\left[\frac{x_0+Ky_0}{1+K}\right]=\frac{\theta_x+K\theta_y}{1+K}$. 

\begin{definition}{\bf [Preserved class of measure]}\label{Rtheta}
{\rm Let $\CR^{(1)}_\theta$ denote the set of measures $\mu \in \CT$ satisfying:
\begin{enumerate}
\item[{\rm (1)}] 
For all $(i,R_i) \in\G\times\{A,D\}$,
\begin{equation}
\label{fm}
\lim_{t\to \infty} \E_\mu[z_{(i,R_i)}(t)]=\theta. 
\end{equation}
\item[{\rm (2)}]
 for all $(i,R_i),\,(j,R_j)\in\G\times\{A,D\}$,
\begin{equation}
\label{sm}
\lim_{t\to \infty}\sum_{(k,R_k),(l,R_l)\in \G\times\{A,D\}} 
b^{(1)}_t\big((i,R_i),(k,R_k)\big)\,b^{(1)}_t\big((j,R_j),(l,R_l)\big)\,\E_\mu[z_{(k,R_k)}z_{(l,R_l)}]=\theta^2.
\end{equation}
\end{enumerate}}\hfill$\Box$
\end{definition}

\noindent
Clearly, if $\mu\in\CR^{(1)}_\theta$, then (1) and (2) together with Lemma~\ref{lem1cg94} imply
\begin{equation}
\label{rt}
\lim_{t\to \infty}\E_\mu\left[\left(\sum_{(k,R_k)\in\G\times\{A,D\}}
b^{(1)}_t((i,R_i),(k,R_k))\,z_{(k,R_k)}-\theta\right)^2\right]=0,
\end{equation}
and so $\lim_{t\to\infty}z_{i,R_i}(t)=\theta$ in $L^2(\mu)$. 

On the other hand, suppose that \eqref{rt} holds for \emph{some} $(i,R_i)\in\G\times\{A,D\}$. Then, by Lemma \eqref{lem1cg94}, we can rewrite \eqref{rt} as 
\begin{equation}
\label{rt3}
\lim_{t\to \infty}\E_\mu\left[\left(\E_z[z_{(i,R_i)}(t)]-\theta\right)^2\right]=0.
\end{equation}
This implies 
\begin{equation}
\label{mo90}
\lim_{t\to\infty}\E_\mu[z_{(i,R_i)}(t)]=\theta,
\end{equation}
and hence, by translation invariance,
\begin{equation}
\label{mo90alt}
\lim_{t\to\infty}\E_\mu[z_{(k,R_i)}(t)]=\theta \qquad \forall\, k\in\G.
\end{equation}
Using that switches between the active state at the dormant state occur at a positive rate, we can use the strong Markov property to obtain that \eqref{mo90alt} holds both for $R_i=A$ and for $R_i=D$. Hence \eqref{fm} holds. Combining \eqref{fm} and \eqref{rt}, we see that also \eqref{sm} holds.

\begin{lemma}
\label{muinrtheta}
$\mu\in\CR^{(1)}_\theta$ for all $\mu\in\CT_\theta^{\mathrm{erg}}$.
\end{lemma}

\begin{proof}
The proof relies on \emph{Fourier analysis} and the existence of spectral measures. As in Section~\ref{ss.clustcase}, for ease of exposition we focus on the special case where $\G = \Z^d$, but the argument below extends to \emph{any} countable Abelian group endowed with the discrete topology.

By translation invariance and the Herglotz theorem, there exist spectral measures $\lambda_A$ and $\lambda_D$ such that, for all $j,k\in\Z^d$,
\begin{equation}
\label{spem}
\begin{aligned}
\E_{\mu}\left[(x_j-\theta_x)(x_k-\theta_x)\right] &=\int_{(-\pi,\pi]^d}\e^{\Im (j-k,\phi)}\d \lambda_A(\phi),\\
\E_{\mu}\left[(y_j-\theta_y)(y_k-\theta_y)\right] &=\int_{(-\pi,\pi]^d}\e^{\Im (j-k,\phi)}\d \lambda_D(\phi).
\end{aligned}
\end{equation}
Let $a(\phi)=\sum_{k\in\Z^d}\e^{\Im(\phi,j)} a(0,k)$ be the characteristic function of the kernel $a(\cdot,\cdot)$ (recall \eqref{same6}), and $T(t)$ the activity time of the random walk up to time $t$ (recall \eqref{Tdefs}). Then
\begin{equation}
\sum_{k\in\Z^d} a_{T(t)}(0,k)\,\e^{\Im(k,\phi)}
=\sum_{n\in\N_0}\frac{\e^{-T(t)}[T(t)]^n}{n!}\sum_{k\in\Z^d} a^n(0,k)\,\e^{\Im(k,\phi)}
=\sum_{n\in\N_0}\frac{\e^{-T(t)}[T(t)\,a(\phi)]^n}{n!}
=\e^{-T(t)(1-a(\phi))}.
\end{equation}
Let $\CE(t)$ be defined as in \eqref{e2061}. Then, for fixed $t>0$, $\P_{(0,A)}(\CE(t))=\sum_{k\in\Z^d}b^{(1)}_t((0,A),(k,A))>0$ and hence
\begin{equation}
\label{spec1}
\begin{aligned}
\E_\mu
&\left[\left(\frac{1}{\P_{(0,A)}(\CE(t))}\sum_{k\in\Z^d}b^{(1)}_t((0,A),(k,A))x_k-\theta_x\right)^2\right]\\
&=\frac{1}{\P_{(0,A)}(\CE(t))^2}\sum_{k,l\in\Z^d} b^{(1)}_t((0,A),(k,A))\,b^{(1)}_t((0,A),(l,A))\,
\E_\mu\left[(x_k-\theta_x)(x_l-\theta_x)\right]\\
&=\frac{1}{\P_{(0,A)}(\CE(t))^2}\sum_{k,l\in\Z^d} b^{(1)}_t((0,A),(k,A))\,b^{(1)}_t((0,A),(l,A))
\int_{(-\pi,\pi]^d} \e^{\Im (k-l,\phi)}\d \lambda_A(\phi)\\
&=\frac{1}{\P_{(0,A)}(\CE(t))^2}\sum_{k,l\in\Z^d} \E_{(0,A),(0,A)}\left[a_{T(t)}(0,k)\,a_{T^\prime(t)}(0,l)\,1_{\CE(t)}\,
1_{\CE^\prime(t)}\right] \int_{(-\pi,\pi]^d} \e^{\Im (k-l,\phi)}\d \lambda_A(\phi)\\
&=\frac{1}{\P_{(0,A)}(\CE(t))^2}\int_{(-\pi,\pi]^d} \E_{(0,A),(0,A)}\left[\sum_{k\in\Z^d} a_{T(t)}\,\e^{\Im (k,\phi)}(0,k)
1_{\CE(t)}\sum_{l\in\Z^d} a_{T^\prime(t)}\,\e^{-\Im (l,\phi)}(0,l)1_{\CE^\prime(t)}\right] \d \lambda_A(\phi)\\
&=\frac{1}{\P_{(0,A)}(\CE(t))^2} \int_{(-\pi,\pi]^d} \E_{(0,A),(0,A)}\left[\e^{-T(t)(1-a(\phi))}1_{\CE(t)}\,
\e^{-T^\prime(t)(1-\bar{a}(\phi))}1_{\CE^\prime(t)}\right] \d \lambda_A(\phi).		
\end{aligned}
\end{equation}
Since $a(\cdot,\cdot)$ is irreducible, $a(\phi)\neq 1$ for all $\phi\in(-\pi,\pi]^d\backslash\{0\}$. Taking the limit $t\to\infty$, we find
\begin{equation}
\lim_{t\to\infty} \E_\mu\left[\left(\frac{1}{\P_{(0,A)}(\CE(t))}
\sum_{k\in\Z^d} b^{(1)}_t((0,A),(k,A))x_k-\theta_x\right)^2\right]=\lambda_A(\{0\}).
\end{equation}
Similarly, 	
\begin{equation}
\lim_{t\to\infty} \E_\mu\left[\left(\frac{1}{\P_{(0,A)}(\CE^c(t))}
\sum_{k\in\Z^d} b^{(1)}_t((0,A),(k,D))y_k-\theta_y\right)^2\right]=\lambda_D(\{0\}).
\end{equation}
Hence
\begin{equation}
\label{spec2}
\begin{aligned}
&\lim_{t\to \infty} \E_\mu\left[\left(\sum_{(k,R_k)\in\Z^d\times\{A,D\}}
b^{(1)}_t((0,A),(k,R_k)\,z_{(k,R_k)}-\theta\right)^2\right]\\
&=\lim_{t\to \infty}\E_\mu\Bigg[\Bigg(\P_{(0,A)}(\CE(t))\sum_{k\in\Z^d}\frac{b^{(1)}_t((0,A),(k,A))}{\P_{(0,A)}(\CE(t))}
\,x_{k}-\frac{\theta_x}{1+K}\\
&\qquad +\P_{(0,A)}(\CE^c(t))\sum_{k\in\Z^d}\frac{b^{(1)}_t((0,A),(k,D))}{\P_{(0,A)}(\CE^c(t))}\,y_{k}
-\frac{K\theta_y}{1+K}\Bigg)^2\Bigg]\\
&\leq \lim_{t\to \infty}\P_{(0,A)}(\CE(t))\,\E_\mu\left[\left(\sum_{k\in\Z^d}
\frac{b^{(1)}_t((0,A),(k,A))}{\P_{(0,A)}(\CE(t))}\,x_{k}-\frac{\theta_x}{(1+K)}\frac{1}{\P_{(0,A)}(\CE(t))}\right)^2\right]\\
&\qquad+\P_{(0,A)}(\CE^c(t))\,\E_\mu\left[\left(\sum_{k\in\Z^d}
\frac{b^{(1)}_t((0,A),(k,D))}{\P_{(0,A)}(\CE^c(t))}\,y_{k}
-\frac{K\theta_y}{1+K}\frac{1}{\P_{(0,A)}(\CE^c(t))}\right)^2\right]\\
&=\frac{1}{1+K}\lambda_A(\{0\})+\frac{K}{1+K}\lambda_D(\{0\}).
\end{aligned}
\end{equation}
Hence, if $\lambda_A(\{0\})=0$ and $\lambda_D(\{0\})=0$, then $\mu\in\CR^{(1)}_\theta$. We will show that 
$\lambda_A(\{0\})=0$ and $\lambda_D(\{0\})=0$ for $\mu\in\CT_\theta^{\mathrm{erg}}$. 

Let $\Lambda_N = [0,N)^d \cap \Z^d$. By the $L^1$-ergodic theorem, we have, for $\mu\in\CT_\theta^{\mathrm{erg}}$,
\begin{equation}
\lim_{N \to \infty}\E_\mu\left[\left(\frac{1}{\Lambda_N}\sum_{j\in\Lambda_N} x_j-\theta_x\right)^2\right]=0.
\end{equation}
(For general $\G$ not that countable groups endowed with the discrete topology are amenable. For amenable groups $\G$, $(\Lambda_N)_{N\in\N}$ must be replaced by a so-called F\o lner sequence, i.e., 
a sequence of finite subsets of $\G$ that exhaust $\G$ and satisfy $\lim_{N\to\infty} |\mathfrak{g} \Lambda_N \triangle \Lambda_N|/|\Lambda_N|=0$ for any $\mathfrak{g}\in\G$ \cite{Li99}. ) Using the spectral measure, we can write
\begin{equation}
\label{q4}
\begin{aligned}
\lim_{N \to \infty} \E_\mu\left[\left(\frac{1}{\Lambda_N} \sum_{j\in\Lambda_N} x_j-\theta_x\right)^2\right]
&=\lim_{N \to \infty} \frac{1}{\Lambda_N^2}\sum_{j,k\in\Lambda_N}
\int_{(-\pi,\pi]^d} \e^{\Im (j-k,\phi)} \d \lambda_A\\
&=\lim_{N \to \infty}\int_{(-\pi,\pi]^d}\left(\frac{1}{\Lambda_N}\sum_{j\in\Lambda_N}
\e^{\Im (j,\phi)}\right)\left(\frac{1}{\Lambda_N}\sum_{k\in\Lambda_N}
\e^{-\Im (k,\phi)}\right)\d \lambda_A
=\lambda_A\{0\}.
\end{aligned}
\end{equation}
In the last equality we use dominated convergence and
\begin{enumerate}
\item 
For all $\phi \in (-\pi,\pi]^d$,
\begin{equation}
\lim_{N \to \infty}\frac{1}{\Lambda_N}\sum_{j,k\in\Lambda_N}\e^{-\Im(k,\phi)}=1_{\{0\}}(\phi).
\end{equation}
\item  
For all $\delta>0$ there exist $\epsilon(N,\delta)>0$ such that if $J_\delta=(-\delta,\delta)$, then 
\begin{equation}
\left|\frac{1}{\Lambda_N}\sum_{j,k\in\Lambda_N}\e^{-\Im (k,\phi)}-1_{\{0\}}(\phi)\right|
\leq 1_{J_\delta}(\phi)+\epsilon(N,\delta),
\end{equation}
where $\epsilon(N,\delta) \downarrow 0$ as $N\to\infty$. 
\end{enumerate}  
We conclude that $\lambda_A(\{0\})=0$. Similarly we can show that $\lambda_D(\{0\})=0$, and hence $\mu\in\CR^{(1)}_\theta$.
\end{proof}

Recall that $(S_t)_{t \geq 0}$ is the semigroup associated with \eqref{gh1}--\eqref{gh2}. 

\begin{lemma}{\bf [Preservation]}
\label{lem2cg94}
If $b(\cdot,\cdot)$ is transient and $\mu\in\CR^{(1)}_\theta$, then the following hold:
\begin{enumerate}
\item[{\rm (a)}] 
$\mu S_t\in\CR^{(1)}_\theta$ for each $t\geq0$.
\item[{\rm (b)}] 
If $t_n\to\infty$ and $\mu S_{t_n}\to\mu(\infty)$, then $\mu(\infty)\in\CR^{(1)}_\theta$.
\end{enumerate}
\end{lemma}

\begin{proof} 
Our dynamics preserve translation invariance. To check property (1) of $\CR^{(1)}_\theta$ (see \eqref{fm}), set $f(z)=z_{(i,R_i)}$. Since $\mu\in \CR^{(1)}_\theta$, applying Lemma \ref{lem1cg94} multiple times, we obtain
\begin{equation}
\label{mo56}
\begin{aligned}
\lim_{s\to \infty} \E_{\mu S_t}[z_{(i,R_i)}(s)]
&= \lim_{s\to\infty} \sum_{(k,R_k)\in \G\times\{A,D\}} b^{(1)}_s\big((i,R_i),(k,R_k)\big)\,
\E_{\mu S_t}[z_{(k,R_k)}]\\
&= \lim_{s\to\infty} \sum_{(k,R_k)\in \G\times\{A,D\}} b^{(1)}_s\big((i,R_i),(k,R_k)\big)\,
\E_{\mu}[z_{(k,R_k)}(t)]\\
&=\lim_{s\to\infty} \sum_{(k^\prime,R^\prime_{k^\prime})\in \G\times\{A,D\}} 
b^{(1)}_{s+t}\big((i,R_i),(k^\prime,R_{k^\prime})\big)\,\E_{\mu}[z_{(k^\prime,R_{k^\prime})}]\\
&=\lim_{s\to \infty} \E_{\mu}[z_{(i,R_i)}(t+s)]=\theta.
\end{aligned}
\end{equation}
To check property (2) of $\CR^{(1)}_\theta$ (see \eqref{sm}), we set $f(z)=z_{(i,R_i)}z_{(j,R_j)}$. Then, again by applying Lemma \ref{lem1cg94}, we find
\begin{equation}
\begin{aligned}
&\lim_{s\to \infty} \sum_{(k,R_k),(l,R_l)\in \G\times\{A,D\}}
b^{(1)}_s\big((i,R_i),(k,R_k)\big)\,b^{(1)}_s\big((j,R_j),(l,R_l)\big)\,\E_{\mu S_t}[z_{(k,R_k)}z_{(l,R_l)}]\\
&=\lim_{s\to \infty}\sum_{(k,R_k),(l,R_l)\in \G\times\{A,D\}} 
b^{(1)}_s\big((i,R_i),(k,R_k)\big)\,b^{(1)}_s\big((j,R_j),(l,R_l)\big)\,\E_{\mu}[z_{(k,R_k)}(t)z_{(l,R_l)}(t)]\\ 
&=\lim_{s\to \infty}\Bigg[\sum_{(k^\prime,R_{k^\prime}),(l^\prime,R_{l^\prime})\in \G\times\{A,D\}} 
b^{(1)}_{t+s}\big((i,R_i),(k^\prime,R_{k^\prime})\big)\,b^{(1)}_{t+s}\big((j,R_j),(l^\prime,R_{l^\prime})\big)\,
\E_\mu[z_{(k^\prime,R_{k^\prime})}z_{(l^\prime,R_{l^\prime})}]\\
&\quad +\,2\int_{0}^{t} \d r\,\sum_{k^\prime\in \G} 
b^{(1)}_{t-r+s}\big((i,R_i),(k^\prime,A)\big)\,b^{(1)}_{t-r+s}\big((j,R_j),(k^\prime,A)\big)\,\E_\mu[g(x_{k^\prime}(r))]\Bigg].
\end{aligned}
\end{equation} 
Since $\mu\in\CR^{(1)}_\theta$, we are left to show that 
\begin{equation}
\lim_{s\to \infty}\int_{s}^{t+s} \d u\,\sum_{k^\prime \in \G} 
b^{(1)}_{u}\big((i,R_i),(k^\prime,A)\big)\,b^{(1)}_{u}\big((j,R_j),(k^\prime,A)\big)\,\E_\mu[g(x_{k^\prime}(t+s-u))]=0.
\end{equation}	
Using the notation of Section \ref{ss.clustcase}, we get
\begin{equation}
\begin{aligned}
&\lim_{s\to \infty} \int_{s}^{t+s} \d u\,\sum_{k^\prime \in \G} 
b^{(1)}_{u}\big((i,R_i),(k^\prime,A)\big)\,b^{(1)}_{u}\big((j,R_j),(k^\prime,A)\big)\,\E_\mu[g(x_{k^\prime}(t+s-u))]\\
&\leq \|g\| \lim_{s\to \infty}\int_{s}^{t+s} \d u\,\sum_{k^\prime \in \G} 
b^{(1)}_{u}\big((i,R_i),(k^\prime,A)\big)\,b^{(1)}_{u}\big((j,R_j),(k^\prime,A)\big)\\
&= \|g\| \lim_{s\to \infty} \int_{s}^{t+s}\d u\,\,\E_{(i,R_i),(j,R_j)}
\left[\sum_{k^\prime \in\G} a_{T(u)}(i,k^\prime)\,1_{\CE(u)}\,a_{T^\prime(u)}(j,k^\prime)\,1_{\CE^\prime(u)}\right]\\
&\leq \|g\| \lim_{s\to \infty} \int_{s}^{t+s}\d u\,\,\E_{(0,A),(0,A)}
\left[\sum_{k^\prime\in\G} a_{T(u)}(i,k^\prime)\,1_{\CE(u)}\,a_{T^\prime(u)}(j,k^\prime)\,1_{\CE^\prime(u)}\right]
= 0,
\end{aligned}
\end{equation}
where the last equality follows from the assumption $I_{\hat{a}}<\infty$ in Theorem \ref{T.ltb1}, \eqref{same2} and \eqref{ISLLN}. The last inequality follows from the Markov property and the observation that, in order to get a contribution to the integral, the two random walks first have to meet at the same site and both be active. We conclude that $\mu S_t\in\CR^{(1)}_\theta$ for all $t\geq0$. 

To show that $\mu(\infty)\in\CR^{(1)}_\theta$, we proceed like in \eqref{mo56}, to obtain
\begin{equation}
\lim_{s\to \infty}\E_{\mu(\infty)}[z_{(i,R_i)}(s)]
=\lim_{s\to \infty}\lim_{n\to\infty}\E_{\mu S_{t_n}}[z_{(i,R_i)}(s)]
= \lim_{s\to \infty}\lim_{n\to\infty}\E_\mu[z_{(i,R_i)}(t_n+s)]=\theta,
\end{equation}
and so \eqref{fm} is satisfied. To get \eqref{sm}, we note that, by Lemma \ref{lem1cg94},
\begin{equation}
\begin{aligned}
&\sum_{(k,R_k),(l,R_l)\in \G\times\{A,D\}} 
b^{(1)}_{t_n}\big((i,R_i),(k,R_k)\big)\,b^{(1)}_{t_n}\big((j,R_j),(l,R_l)\big)\,\E_\mu[z_{(k,R_k)}z_{(l,R_l)}]\\
& \leq E_\mu[z_{(i,R_i)}(t_n)z_{(j,R_j)}(t_n)]\\
&\leq \sum_{(k,R_k),(l,R_l)\in \G\times\{A,D\}} 
b^{(1)}_{t_n}\big((i,R_i),(k,R_k)\big)\,b^{(1)}_{t_n}\big((j,R_j),(l,R_l)\big)\,\E_\mu[z_{(k,R_k)}z_{(l,R_l)}]\\
&\qquad +\,2\|g\|\int_{0}^{t_n}\d s \sum_{k\in \G} 
b^{(1)}_{t_n-s}\big((i,R_i),(k,A)\big)\,b^{(1)}_{t_n-s}\big((j,R_j),(k,A)\big).
\end{aligned}
\end{equation}
Letting $n\to\infty$, we see that, since $\mu\in\CR^{(1)}_\theta$,
\begin{equation}
\label{mo44}
\begin{aligned}
\theta^2
&\leq \E_{\mu(\infty)}[z_{(i,R_i)} z_{(j,R_j)}]\\
&\leq \theta^2 + 2\|g\| \int_{0}^{\infty} \d s\,
\sum_{k\in \G} b^{(1)}_{r}\big((i,R_i),(k,A)\big)\,b^{(1)}_{r}\big((j,R_j),(k,A)\big).
\end{aligned}
\end{equation}
Inserting \eqref{mo44} into \eqref{sm}, we see that  it is enough to show that 
\begin{equation}
\begin{aligned}
&\lim_{s\to \infty} \sum_{(k,R_k),(l,R_l)\in \G\times\{A,D\}}b^{(1)}_s\big((i,R_i),(k,R_k)\big)\,
b^{(1)}_s\big((j,R_j),(l,R_l)\big)\\
&\qquad \times\,\, 2\|g\| \int_{0}^{\infty} \d r\,\sum_{k^\prime \in \G} 
b^{(1)}_{r}((k,R_k),(k^\prime,A))\,b^{(1)}_{r}((l,R_l),(k^\prime,A))\\
&= \lim_{s\to\infty} 2\|g\| \int_{0}^{\infty} \d r\,\sum_{k^\prime \in \G} 
b^{(1)}_{r+s}\big((i,R_i),(k^\prime,A)\big)\,b^{(1)}_{r+s}\big((j,R_j),(k^\prime,A)\big)
=  0.
\end{aligned}
\end{equation}
However, from the assumption $I_{\hat{a}}<\infty$ in Theorem \ref{T.ltb1}, \eqref{same2} and \eqref{ISLLN}, we have
\begin{equation}
\begin{aligned}
& \lim_{s\to\infty} \|g\| \int_{0}^{\infty} \d r\,\sum_{k^\prime \in \G} 
b^{(1)}_{r+s}\big((i,R_i),(k^\prime,A)\big)\,b^{(1)}_{r+s}\big((j,R_j),(k^\prime,A)\big)\\
&= \lim_{s\to\infty} \|g\| \int_{s}^{\infty} \d r\,\,
\E_{(i,R_i),(j,R_j)}\left[\sum_{k^\prime \in\G}a_{T(r)}(i,k^\prime)1_{\CE(r)}
a_{T^\prime(r)}(j,k^\prime)1_{\CE^\prime(r)}\right] = 0.
\end{aligned}
\end{equation}
\end{proof}

\paragraph{2. Uniqueness of the equilibrium.} 

In this section we show that, for given $\theta$, the equilibrium when it exists is unique. To prove this we extend the coupling argument in \cite{CG94}. Consider two copies of the system \eqref{gh1}--\eqref{gh2} \emph{coupled via their Brownian motions}:     
\begin{eqnarray}
\label{gh1k}
\d x^k_i(t) &=& \sum_{j \in \G} a(i,j)\, [x^k_j (t) - x^k_i(t)]\,\d t
+ \sqrt{g(x^k_i(t))}\,\d w_i(t) +\,Ke\,[y^k_i(t) - x^k_i(t)]\,\d t,\\[0.2cm]
\label{gh2k}
\d y^k_i (t) &=& e\,[x^k_i(t) - y^k_i(t)]\,\d t,\qquad k \in \{1,2\}.
\end{eqnarray}
Here, $k$ labels the copy, and the two copies are driven by the same set of Brownian motions $(w_i(t))_{t \geq 0}$, $i \in 
\G$. As initial probability distributions we choose $\mu^1(0)$ and $\mu^2(0)$ that are both invariant and ergodic under translations. 

Let 
\begin{equation}
\label{e1809}
\bar{z}_i(t) = (z^1_i(t),z^2_i(t)), \qquad z^k_i(t) = (x^k_i(t),y^k_i(t)), \quad k \in \{1,2\}.
\end{equation} 
The coupled system $(\bar{z}_i(t))_{i\in\G}$ has a unique strong solution \cite[Theorem 3.2]{SS80} whose marginals are the single-component systems. Write $\hat\P$ to denote the law of the coupled system, and let $\Delta_i(t) = x^1_i(t)-x^2_i(t)$ and $\delta_i(t) = y^1_i(t)-y^2_i(t)$.  

\begin{lemma}{\bf [Coupling dynamics]}
\label{L.coup}
For every $t \geq 0$,
\begin{equation}
\label{couprel}
\begin{aligned}
\frac{\d}{\d t}\,\hat\E\big[|\Delta_i(t)| + K |\delta_i(t)|\big] 
&= - 2 \sum_{j \in \G} a(i,j)\,\hat\E\left[|\Delta_j(t)|\,
1_{\{\sign\Delta_i(t)\,\neq\,\sign\Delta_j(t)\}}\right]\\
&\qquad - 2Ke\,\hat\E\left[\big(|\Delta_i(t)| + |\delta_i(t)|\big)\,
1_{\{\sign\Delta_i(t)\,\neq\,\sign\delta_i(t)\}}\right].
\end{aligned}
\end{equation} 
\end{lemma}

\begin{proof}
\label{pr.1831}
Let $f(x) = |x|$, $x \in \R$. Then $f'(x) = \sign x$ and $f''(x) = 0$ for $x \neq 0$, but $f$ is not differentiable at $x=0$, a point the path hits. Therefore, by a generalization of It\^o's formula, we have
\begin{equation}
\label{Ito1}
\begin{aligned}
\d |\Delta_i(t)| &= \sign\Delta_i(t)\,\d\Delta_i(t) +\d  L^0_t,\\
\d\Delta_i(t) &= \sum_{j\in\G} a(i,j) [\Delta_j(t)-\Delta_i(t)]\,\d t
+ \left[\sqrt{g(x^1_i(t))} - \sqrt{g(x^2_i(t))}\,\right]\, \d w_i(t)\\
&\qquad\qquad  + Ke\,\big[\delta_i(t) - \Delta_i(t)\big]\,\d t,
\end{aligned}
\end{equation}
where $L^0_t$ is the local time of $\Delta_i(t)$ at $0$ (see \cite[Section IV.43]{RoWi00}). Next, we use that $\Delta_i(t)$ has zero local time at $x=0$ because $g$ is Lipschitz (see \cite[Proposition V.39.3]{RoWi00}). Taking expectation, we get
\begin{equation}
\label{Ito2}
\frac{\d}{\d t}\,\hat\E\big[|\Delta_i(t)|\big] 
= \sum_{j\in\G} a(i,j)\,\hat\E\Big[\sign\Delta_i(t)\,[\Delta_j(t)-\Delta_i(t)]\Big]
+ Ke\,\hat\E\Big[\sign\Delta_i(t)\,[\delta_i(t)-\Delta_i(t)]\Big].
\end{equation}
Similarly, we have
\begin{equation}
\label{Ito3}
\begin{aligned}
\d |\delta_i(t)| &= \sign\delta_i(t)\,\d\delta_i(t),\\
\d\delta_i(t) &= e\,\big[\Delta_i(t) - \delta_i(t)\big]\,\d t.
\end{aligned}
\end{equation}
Taking expectation, we get
\begin{equation}
\label{Ito4}
\frac{\d}{\d t}\,\hat\E\big[|\delta_i(t)|\big] 
=  e\,\hat\E\Big[\sign\delta_i(t)\,[\Delta_i(t)-\delta_i(t)]\Big].
\end{equation}
Combining \eqref{Ito2} and \eqref{Ito4}, we get
\begin{equation}
\label{Ito5}
\begin{aligned}
\frac{\d}{\d t}\,\hat\E\big[|\Delta_i(t)| + K |\delta_i(t)|\big] 
&= \sum_{j\in\G} a(i,j)\,\hat\E\Big[\sign\Delta_i(t)\,[\Delta_j(t)-\Delta_i(t)]\Big]\\
&\qquad + K\,e\,\hat\E\Big[ [\sign\Delta_i(t)-\sign\delta_i(t)]\,[\delta_i(t)-\Delta_i(t)]\Big].
\end{aligned}
\end{equation}
Note that 
\begin{equation}
\label{Ito6*}
\sign\Delta_i(t)\,[\Delta_j(t)-\Delta_i(t)] = |\Delta_j(t)|-|\Delta_i(t)|
-2\,|\Delta_j(t)|\,1_{\{\sign\Delta_i(t) \neq \sign \Delta_j(t)\}}.
\end{equation}
By translation invariance, $\E[|\Delta_i(t)|]$ is independent of $i$. Hence the first sum in the right-hand side can be rewritten as 
\begin{equation}
\label{Ito6}
-2 \sum_{j\in\G} a(i,j)\,\hat\E\Big[|\Delta_j(t)|\,1_{\{\sign\Delta_i(t) \neq \sign \Delta_j(t)\}}\Big].
\end{equation}
Similarly, the second sum in the right-hand side can be rewritten as  
\begin{equation}
\label{Ito7}
- 2Ke\,\hat\E\left[\big(|\Delta_i(t)| + |\delta_i(t)|\big)\,
1_{\{\sign\Delta_i(t)\,\neq\,\sign\delta_i(t)\}}\right].
\end{equation}
Combining \eqref{Ito5} and \eqref{Ito6}--\eqref{Ito7}, we get the claim. 
\end{proof}

Lemma~\ref{L.coup} tells us that $t \mapsto \hat\E[|\Delta_i(t)| + K |\delta_i(t)|]$ is a non-increasing Lyapunov function. Therefore $\lim_{t\to\infty} \hat\E[|\Delta_i(t)| + K |\delta_i(t)|] = c_i \in [0,1+K]$ exists. To show that the coupling is successful we need the following lemma.

\begin{lemma}{\bf [Uniqueness of equilibrium]}
\label{L.suc}
If $a(\cdot,\cdot)$ is transient, then $c_i=0$ for all $i \in \G$, and so the coupling is successful, i.e., 
\begin{equation}
\lim_{t\to\infty}\hat\E\big[|\Delta_i(t)| + K |\delta_i(t)|\big]=0. 
\end{equation} 
\end{lemma}

\begin{proof}
Write $-h_i(t)$ to denote the right-hand side of \eqref{couprel}. We begin with the observation that 
$t \mapsto h_i(t)$ has the following properties:

\medskip
\begin{tabular}{ll}
&(a) $h_i \geq 0$.\\
&(b) $0 \leq \int_0^\infty \d t\,h_i(t) \leq 1+K$.\\
&(c) $h_i$ is differentiable with $h_i'$ bounded.
\end{tabular} 

\medskip\noindent
Property (a) is evident. Property (b) follows from integration of \eqref{couprel}:
\begin{equation}
\label{e3298}
\int_0^t \d s\,h_i(s) = \hat{\E}\big[|\Delta_i(0)|+K|\delta_i(0)|\big] - \hat{\E}\big[|\Delta_i(t)|+K|\delta_i(t)|\big]. 
\end{equation}  
The proof of Property (c) is given in Appendix \ref{appe}. It follows from (a)--(c) that $\lim_{t\to\infty} h(t) = 0$. Hence, for every $\epsilon>0$,
\begin{equation}
\label{eps}
\begin{aligned}
&\forall\,i,j\in\G \text{ with } a(i,j)>0\colon\\
&\qquad \lim_{t\to\infty} \hat{\P}\Big(\{\Delta_i(t)<-\epsilon,\, \Delta_j(t)>\epsilon\}
\cup \{\Delta_i(t)>\epsilon,\, \Delta_j(t)<-\epsilon\}\Big) = 0,\\ 
&\forall\,i\in\G\colon\\
&\qquad \lim_{t\to\infty} \hat{\P}\Big(\{\Delta_i(t)<-\epsilon,\, \delta_i(t)>\epsilon\}
\cup \{\Delta_i(t)>\epsilon,\, \delta_i(t)<-\epsilon\}\Big) = 0.
\end{aligned}
\end{equation}
In Appendix \ref{appC} we will prove the following lemma:

\begin{lemma}{\bf [Successful coupling ]}
\label{C2} 
For all $i,j\in\G$ and $\epsilon>0$,
\begin{equation}
\label{eqlem58}
\lim_{t\to\infty} \hat{\P}\Big(\{\Delta_i(t)<-\epsilon,\, \Delta_j(t)>\epsilon\}
\cup \{\Delta_i(t)>\epsilon,\, \Delta_j(t)<-\epsilon\}\Big) = 0.
\end{equation}
\end{lemma}

\noindent
The proof of this lemma relies on the fact that $\hat{a}(\cdot,\cdot)$ is irreducible. Let 
\begin{equation}
\begin{aligned}
E_0\times E_0=&\left\{\bar{z}\in E\times E\colon\, z_{(i,R_i)}^1(t)
\geq z_{(i,R_i)}^2(t)\ \forall (i,R_i)\in\G\times\{A,D\}\right\}
\\& \cup \left\{\bar{z}\in E\times E\colon\,z_{(i,R_i)}^2(t)
\geq z_{(i,R_i)}^1(t)\ \forall (i,R_i)\in\G\times\{A,D\}\right\}.
\end{aligned}
\end{equation}
Then Lemma \ref{C2} together with \eqref{eps} imply that $\lim_{t\to\infty}\hat{\P}\left(E_0\times E_0\right)=1$, which we express by saying that ``one diffusion lies on top of the other". 

Using Lemma \ref{C2} we can complete the proof of the successful coupling. Let $t_n\to\infty$ as $n\to\infty$ and suppose, by possibly going to further subsequences, that $\lim_{n\to\infty}\mu^1(t_n)=\nu_\theta^1$ and $\lim_{n\to\infty}\mu^2(t_n)=\nu_\theta^2$. Let $\bar{\nu}_\theta$ be the measure on $E\times E$ given by $\bar{\nu}_\theta = \nu^1_\theta\times \nu^2_\theta$. Using dominated convergence, invoking the preservation of translation invariance, and using the limiting distribution of $b^{(1)}_t(\cdot,\cdot)$ on $\{A,D\}$, we find
\begin{equation}
\begin{aligned}
\label{eqsuccoup} 
&\int_{E\times E} \d\bar{\nu}_\theta\,|\Delta_i|+K|\delta_i|\\
&\quad = (1+K) \int_{E_0\times E_0} \d\bar{\nu}_\theta \lim_{n\to\infty} \sum_{j\in\G}
\left[b^{(1)}_{t_n}\big((i,R_i),(j,A)\big)\,|x^1_i-x^2_i|+b^{(1)}_{t_n}\big((i,R_i),(j,D)\big)\,|y^1_i-y_i^2|\right]\\
&\quad = \lim_{n\to\infty} (1+K) \int_{E_0\times E_0} \d\bar{\nu}_\theta 
\left|\sum_{j\in\G\times\{A,D\}} b^{(1)}_{t_n}\big((i,R_i),(j,R_j)\big)\,(z^1_{(j,R_j)}-z^2_{(j,R_j)})\,\right| \\
&\quad \leq \lim_{n\to\infty}(1+K) \int_{E} \d\nu^1_\theta \left|\sum_{j\in\G\times\{A,D\}}
b^{(1)}_{t_n}\big((i,R_i),(j,R_j)\big)\,z^1_{(j,R_j)}-\theta\,\right|\\
&\qquad + \lim_{n\to\infty} (1+K) \int_{E} \d\nu^2_\theta \left|\sum_{i\in\G\times\{A,D\}}
b^{(1)}_{t_n}\big((i,R_i),(j,R_j)\big)\,z^2_{(j,R_j)}-\theta\,\right| = 0.
\end{aligned}
\end{equation}
Here, the last equality follows because both $\nu_\theta^1$ and $\nu_\theta^2$ are in $\CR^{(1)}_\theta$ by Lemma \ref{lem2cg94}. Thus, we see that $\bar{\nu}_\theta$ concentrates on the diagonal. Suppose now that there exists a sequence $(t_n)_{n\in\N}$ such that $\lim_{n\to\infty} \E[|\Delta_i(t_n)|+K|\delta_i(t_n)|]=\delta>0$. Since $\{\mathcal{L}(\bar{Z}(t_n))\}_{n\in\N}$ is tight (recall \eqref{e1809}), by Prokhorov's theorem there exists a converging subsequence $\{\mathcal{L}(\bar{Z}(t_{n_k}))\}_{k\in\N}$. Let $\bar\nu_\theta$ denote the limiting measure. Then, by Lemma \ref{lem2cg94} and \eqref{eqsuccoup},
\begin{equation}
\delta = \lim_{k\to\infty}\E[|\Delta_i(t_{n_k})|+K|\delta_i(t_{n_k})|]
= \int_{E\times E}\d\bar{\nu}_\theta\,[|\Delta_i|+K|\delta_i|]=0.
\end{equation}
Thus, $\lim_{t \to \infty}\E[|\Delta_i(t)|+K|\delta_i(t)|]=0$, and we conclude that the coupling is successful. Hence, given the initial average density $\theta$ in \eqref{thetadef}, the equilibrium measure is unique if it exists. 
\end{proof}

%%%

\paragraph{3. Stationarity of $\nu_\theta$ and convergence to $\nu_\theta$.} 

\begin{lemma}{\bf [Existence of equilibrium]}
\label{lemenu}
Let $\mu(0)\in\CT^{\mathrm{erg}}_\theta$. Then $\lim_{t\to\infty}\mu(t)=\nu_\theta$ for some invariant measure $\nu_\theta$.
\end{lemma}

\begin{proof}
To prove that the limit is an invariant measure, suppose that $\mu(0)=\mu=\delta_\theta$. Since the state space of $(Z(t))_{t\geq 0}$ is compact, each sequence $\{\mathcal{L}(Z(t_n))\}_{n\in\N}$ is tight. Hence, by Prokhorov's theorem, there exists a converging subsequence such that $\lim_{n\to\infty}\delta_\theta S_{t_n}=\nu_\theta$. Since $\delta_\theta\in\CR^{(1)}_\theta$, Lemma \ref{lem2cg94} tells us that $\lim_{n\to\infty} \delta_\theta S_{t_n}\in \CR^{(1)}_\theta$. To prove that $\nu_\theta$ is invariant, fix any $s_0\geq0$. Let $\mu=\delta_\theta S_{s_0}$. Then, by Lemma \ref{lem2cg94}, $\mu\in\CR^{(1)}_\theta$ and, by Lemma \ref{C2}, we can find a further subsequence such that $\lim_{k\to\infty} \mu(t_{n_k})=\nu_\theta$. By the Feller property of the SSDE in \eqref{gh1}--\eqref{gh2}, we obtain
\begin{equation}
\nu_\theta S_{s_0}= \lim_{n\to\infty} \delta_\theta S_{t_n} S_{s_0}
= \lim_{k\to\infty}\delta_\theta S_{s_0} S_{t_{n_k}} =\lim_{k\to\infty} \mu S_{t_{n_k}} = \nu_\theta.
\end{equation}
Hence, $\nu_\theta$ is indeed an invariant measure.
	
To prove the convergence of $\mu(t)$ to $\nu_\theta$, note that $\nu_\theta\in\CR^{(1)}_\theta$ by Lemma~\ref{lem2cg94}. Let $\nu=\nu_\theta$. Then, by the invariance of $\nu_\theta$, we have $\lim_{t\to\infty} \nu S_t=\nu_\theta$. By Lemma \ref{L.suc}, we have $\lim_{t\to\infty}\mu S_t=\lim_{t\to\infty}\nu S_t = \nu_\theta$ for all $\mu\in\CR^{(1)}_\theta$.   
\end{proof}

%%%

\paragraph{4. Ergodicity, mixing and associatedness.}

\begin{lemma}{\bf [Properties of equilibrium]}
\label{lem10}
Let $\mu(0)\in\CR^{(1)}_\theta$ be ergodic under translations. Then $\nu_\theta=\lim_{t\to\infty} \mu(t)$ is ergodic and mixing under translations, and is associated. 
\end{lemma}

\begin{proof}
After a standard approximation argument, \cite[Corollary1.5 and subsequent discussion]{HPitt} implies that associatedness is preserved over time. Note that $\delta_\theta$ is an associated measure and lies in $\CR^{(1)}_\theta$. Hence, by Lemma \ref{lemenu}, $\nu_\theta=\lim_{t\to\infty} \delta_\theta S_t$ and therefore $\nu_\theta$ is associated.  
	
We prove the ergodicity of $\nu_\theta$ by showing that the random field of components is mixing. To prove that $\nu_\theta$ is mixing, we use associatedness and decay of correlations. Let $B,B^\prime\subset\G$ be finite, and let $c_j$, $d_i$ be positive constants for $j\in B$, $i\in B^\prime$. For $k\in\G$, define 
the random variables
\begin{equation}
Y_0=\sum_{j\in B} c_j z_{(j,R_j)}, \qquad Y_k=\sum_{i\in B^\prime} d_i z_{(i+k,R_{i+k})}.
\end{equation}
Note that $Y_0$ and $Y_k$ are associated under $\nu_\theta$ because $(z_{(i,R_i)})_{(i,R_i)\in \G\times\{A,D\}}$ 
are associated. Therefore, by \cite[Eq.(2.2)]{NW81}, it follows that for $s,t\in\R$, 
\begin{equation}
\left|\E_{\nu_\theta}[e^{\Im(sY_0+tY_n)}]-\E_{\nu_\theta}[e^{\Im sY_0}]\,\E_{\nu_\theta}
[e^{\Im tY_n}]\right|\leq \left|st\right| \text{Cov}_{\nu_\theta}(Y_0,Y_n).
\end{equation}
Since $\mu\in\CR^{(1)}_\theta$ by Lemma \ref{lem1cg94},
\begin{eqnarray}
\text{Cov}_{\nu_\theta}(Y_0,Y_k)&=&\sum_{j\in B}\sum_{i\in B^\prime}c_jd_i
\lim_{t\to\infty}\text{Cov}_{\mu}(z_{(j,R_j)}(t),z_{(i+k,R_{i+k})}(t))\\ \nonumber
&\leq& 2\|g\|\sum_{j\in B}\sum_{i\in B^\prime}c_jd_i \int_{0}^{\infty} \d r\,
\sum_{(l,R_l)\in \G\times\{A\}} b^{(1)}_{r}\big((j,R_j),(l,A)\big)\,b^{(1)}_{r}\big((i+k,R_{i+k}),(l,A)\big).
\end{eqnarray}
The last integral gives the expected total time for two partition elements in the dual, starting in $(j,R_j)$ and $(i+k, R_{i+k})$, to be active at the same site. To show that this integral converges to $0$ as $\|k\|\to\infty$, we rewrite the sum as (recall \eqref{same4}--\eqref{same5})
\begin{equation}
\label{aso}
\begin{aligned}
&\E_{(i+k,R_{i+k}),(j,R_j)}\left[\left(\sum_{l\in\G} a_{T(r)}(j,l)\,a_{T^\prime(r)}(i+k,l)\right)
1_{\CE(r)}\,1_{\CE'(r)}\right]\\
&= \E_{(i+k,R_{i+k}),(j,R_j)}\left[\left(\sum_{l'\in\G} \hat{a}_{2M(r)}(i+k-j,l')\,a_{\Delta(r)}(l',0)\right)
1_{\CE(r)}\,1_{\CE'(r)}\right]\\
&\leq \E_{(i+k,R_{i+k}),(j,R_j)}\left[\left(\sum_{l'\in\G} \hat{a}_{2M(r)}(i+k-j,l')\,
\big[a_{\Delta(r)}(l',0)+a_{\Delta(r)}(0,l')\big]\right)
1_{\CE(r)}\,1_{\CE'(r)}\right]\\
&= \E_{(i+k,R_{i+k}),(j,R_j)}\left[ \hat{a}_{2M(r)+2\Delta(r)}(i+k-j,0)
1_{\CE(r)}\,1_{\CE'(r)}\right].
\end{aligned}
\end{equation}
Because $\hat{a}(\cdot,\cdot)$ is symmetric, we have $\hat{a}_{2M(r)+2\Delta(r)}(i+k-j,0) \leq \hat{a}_{2M(r)+2\Delta(r)}(0,0)$. Since 
\begin{equation}
T(t)+T^\prime(t)\leq 2M(r)+2\Delta(r)\leq 2\left(T(t)+T^\prime(t)\right),
\end{equation}
and the Fourier transform in \eqref{Fourier}--\eqref{fdef} implies that 
\begin{equation}
\int_{0}^{\infty} \d r\,\E_{(i+k,R_{i+k}),(j,R_j)}\left[ \hat{a}_{2M(r)+2\Delta(r)}(0,0)1_{\CE(r)}\,1_{\CE'(r)}\right]<\infty.
\end{equation}
if and only if $I_{\hat{a}}<\infty$. Since we are in the transient regime, i.e., $I_{\hat{a}}<\infty$, we can use dominated convergence, in combination with the fact that $\lim_{\|k\| \to \infty} \hat{a}_t (i+k-j,0)=0$ for all $i,j,t$, to conclude that $\lim_{\|k\|\to\infty} \text{Cov}_{\nu_\theta}(Y_0,Y_k)=0$.
\end{proof}

%%%

\subsection{Proof of the dichotomy}
\label{ss.dichmodel1}

Theorem \ref{T.ltb1}(a) follows from Lemmas \ref{muinrtheta}, \ref{lemenu} and \ref{lem10}. The equality $\E_{\nu_\theta}[x_0]=\E_{\nu_\theta}[y_0]=\theta$ follows from the evolution equations in \eqref{gh1}--\eqref{gh2}, the fact that $\nu_\theta$ is an equilibrium measure, and the preservation of $\theta$ (see \eqref{mart1}). Theorem \ref{T.ltb1}(b) follows from Lemma~\ref{lem:comparison}.

%%%%%%%%%% SECTION 6 %%%%%%%%%%%%%%%%%%%%%%%%%%%

\section{Proofs: Long-time behaviour for Model 2}
\label{s.model2}

In Sections \ref{ss.momrel2}--\ref{ss.recap2} we show that the results proved in Sections~\ref{ss.momrels}--\ref{ss.dichmodel1} carry over from model 1 to model 2. In Section~\ref{asymmig} we show that symmetry of $a(\cdot,\cdot)$ is needed. In Section~\ref{ss.proofslowvar} we show what happens when for infinite seed-bank the fat-tailed wake-up time is modulated by a slowly varying function.

%%%%%

\subsection{Moment relations}\label{ss.momrel2}

Like in model 1, we start by relating the first and second moments of the system in \eqref{gh1*}--\eqref{gh2*} to the
random walk that evolves according to the transition kernel $b^{(2)}(\cdot,\cdot)$ on $\G\times\{A,(D_m)_{m\in\N_0}\}$ given by \eqref{mrw2}. Also here these moment relations hold for all $g\in\CG$. Moreover these moment relations holds for $\rho<\infty$ as well as for $\rho=\infty$. Below we write $\E_z$ for $\E_{\delta_z}$, the expectation when the process starts from the initial measure $\delta_z$, $z \in E$.

\begin{lemma}{\bf [First and second moment]}
\label{lem1cg9422}
For $z\in E$, $t\geq 0$ and $(i,R_i),(j,R_j)\in \G\times\{A,(D_m)_{m\in\N_0}\}$, 
\begin{equation}
\label{eqexprw21}
\E_z[z_{(i,R_i)}(t)]=\sum_{(k,R_k)\in \G\times\{A,(D_m)_{m\in\N_0}\}} b^{(2)}_t\big((i,R_i),(k,R_k)\big)\,z_{(k,R_k)}
\end{equation}
and 
\begin{equation}
\label{eqexprw22}
\begin{aligned}
&&\E_z[z_{(i,R_i)}(t)z_{(j,R_j)}(t)]
=\sum_{(k,R_k),(l,R_l)\in \G\times\{A,(D_m)_{m\in\N_0}\}} b^{(2)}_t\big((i,R_i),(k,R_k)\big)\,
b^{(2)}_t\big((j,R_j),(l,R_l)\big)\,z_{(k,R_k)}z_{(l,R_l)}\\ 
& &+\,2\int_{0}^{t} \d s \sum_{k\in \G} 
b^{(2)}_{(t-s)}((i,R_i),(k,A))\,b^{(2)}_{(t-s)}((j,R_j),(k,A))\,\E_z[g(x_k(s))].
\end{aligned}
\end{equation}
\end{lemma}

\begin{proof}
The proof follows from that of Lemma~\ref{lem1cg94} after we replace $b^{(1)}(\cdot,\cdot)$ by $b^{(2)}(\cdot,\cdot)$ and use \eqref{gh1*}--\eqref{gh2*} instead of \eqref{gh1}--\eqref{gh2}.
\end{proof}

\begin{remark}{\bf [Density]}
\label{rem:prob2}
{\rm From Lemma \ref{lem1cg9422} we obtain that if $\mu$ is invariant under translations with $\E_\mu[x_0(0)]=\theta_x$ and $\E_\mu[y_{0,m}(0)]=\theta_{y,m}$ for all $m\in\N_0$, then  
\begin{equation}
\label{e5.522}
\E_\mu[z_{(i,R_i)}(t)]=\theta_x\sum_{(k,R_k) \in \G\times\{A\}} b^{(2)}_t\big((i,R_i),(k,R_k)\big)
+\sum_{m\in\N_0}\theta_{y,m}\sum_{(k,R_k)\in \G\times\{D_m\}} b^{(2)}_t\big((i,R_i),(k,R_k)\big)
\end{equation}	
and 
\begin{equation}
\label{e5.532}
\begin{aligned}
&\E_\mu[z_{(i,R_i)}(t)z_{(j,R_j)}(t)]\\
&= \sum_{(k,R_k),(l,R_l)\in \G\times\{A,(D_m)_{m\in\N_0}\}} b^{(2)}_t\big((i,R_i),(k,R_k)\big)\,
b^{(2)}_t\big((j,R_j),(l,R_l)\big)\,\E_\mu[z_{(k,R_k)}z_{(l,R_l)}]\\ 
&\qquad +\,2\int_{0}^{t}\d s \sum_{k\in \G} 
b^{(2)}_{t-s}\big((i,R_i),(k,A)\big)\,b^{(2)}_{t-s}\big((j,R_j),(k,A)\big)\,\E_\mu[g(x_i(s))].
\end{aligned}
\end{equation}
		
\begin{itemize}
\item For $\rho<\infty$, the kernel $b^{(2)}(\cdot,\cdot)$ projected on the second component (= the seed-bank) is a recurrent Markov chain. Therefore, by translation invariance in the first component, we have
\begin{equation}
\lim_{t\to\infty}\E_\mu[z_{(i,R_i)}(t)]=\frac{\theta_x +\sum_{m\in\N_0}K_m\theta_{y,m}}{1+\sum_{m\in\N_0}K_m}=\theta.
\end{equation}
\item For $\rho=\infty$ the kernel $b^{(2)}(\cdot,\cdot)$ viewed as a kernel on $\{A,(D_m)_{m\in\N_0}\}$ relates to a null-recurrent Markov chain. Hence, for all $(i,R_i)$ and all $ D_m,\ m\in\N_0$,
\begin{equation}
\lim_{t \to \infty}\sum_{k\in\G}b^{(2)}_t((i,R_i),(k,D_m))=0.
\end{equation}
Since for $\rho=\infty$ we assume not only that $\mu\in\CT^{\mathrm{erg}}_\theta$ but also that $\mu$ is colour regular, it follows that, for all $M\in\N_0$,
\begin{equation}
\begin{aligned}
\label{q1}
\lim_{t\to\infty}\E_\mu[z_{(i,R_i)}(t)]&=\lim_{t\to\infty}\theta_x\sum_{k\in \G} b^{(2)}_t\big((i,R_i),(k,R_k)\big)
+\sum_{m\in\N_0}\theta_{y,m}\sum_{(k,R_k)\in \G\times\{D_m\}} b^{(2)}_t\big((i,R_i),(k,R_k)\big)\\
&=\lim_{t\to\infty}\sum_{m=M}^\infty\theta_{y,m}\sum_{(k,R_k)\in \G\times\{D_m\}} b^{(2)}_t\big((i,R_i),(k,R_k)\big).\end{aligned}
\end{equation}
Therefore
\begin{equation}
\label{dens}
\lim_{t\to\infty}\E_\mu[z_{(i,R_i)}(t)]=\theta.
\end{equation}   
\end{itemize}
} \hfill $\Box$
\end{remark}
%%%

\subsection{The clustering case}
\label{ss.clustcasealt}

In this section we prove convergence to a trivial equilibrium when $\rho<\infty$ and $I_{\hat{a}}=\infty$ and when $\rho=\infty$ and $I_{\hat{a},\gamma}=\infty$. The proof follows along the same lines as in Section \ref{ss.clustcase}. Therefore we again first consider $g=dg_{\mathrm{FW}}$, and subsequently use a duality comparison argument to show that the results hold for $g \neq dg_{\mathrm{FW}}$ as well. 

\paragraph{{\bf Case $g=dg_{\mathrm{FW}}$.}}

We start by proving the equivalent of Lemma \ref{lem:clusgisgfw}, which is Lemma \ref{lem:clusgisgfw2} below.

\begin{lemma}{\bf [Clustering]}
\label{lem:clusgisgfw2}
Suppose that $\mu(0)\in\CT^{\mathrm{erg}}_\theta$ and $g=dg_{\mathrm{FW}}$. Let $\mu(t)$ be the law at time $t$ of the system defined in \eqref{gh1*}--\eqref{gh2*}. Then the following two statements hold:
\begin{itemize}
\item  
If $\rho<\infty$ and $I_{\hat{a}}=\infty$, i.e., $\hat{a}(\cdot,\cdot)$ is recurrent, then  
\begin{equation}
\label{gh12p2a}
\lim_{t\to\infty} \mu(t)
= \theta\, [\delta_{(1,1^{\N_0})}]^{\otimes \G} + (1-\theta)\, [\delta_{(0,0^{\N_0})}]^{\otimes \G}.
\end{equation}
\item  
If $\rho=\infty$ and  $I_{\hat{a},\gamma}=\infty$ then  
\begin{equation}
\label{gh12p2b}
\lim_{t\to\infty} \mu(t) = \theta\, [\delta_{(1,1^{\N_0})}]^{\otimes \G} + (1-\theta)\, [\delta_{(0,0^{\N_0})}]^{\otimes \G}.
\end{equation}
\end{itemize} 
\end{lemma}

\begin{proof}
We distinguish between $\rho < \infty$ and $\rho = \infty$, which exhibit different behaviour. 

\paragraph{Case $\rho < \infty$.}
\label{sss.caserho} 

The same dichotomy as for model 1 holds when the average wake-up time is finite (recall \eqref{emcond}--\eqref{rhodef}, \eqref{exptau}). Indeed, the argument in \eqref{SLLN}--\eqref{ISLLN} can be copied with  $Ke,e$ replaced by $\chi,\chi/\rho$ and $A,B$ by $\chi/(1+\rho),1/(1+\rho)$. Under the \emph{symmetry assumption} in \eqref{sym} we have $\tilde{a}(\phi)=0$. Hence only the law of large numbers in \eqref{ETasymp} is needed, not the central limit theorem in \eqref{CLT}, which may fail (see Section~\ref{asymmig}).

%%%

\paragraph{Case $\rho=\infty$.}
\label{ss.rhoinfty}

When the average wake-up time is infinite, we need the assumptions in \eqref{ass2} and \eqref{assalt}. By the standard law of large numbers for stable random variables (see e.g.\ \cite[Section XIII.6]{F71}), we have
\begin{equation}
\label{lawasympalt}
\lim_{k\to\infty} \frac{1}{k} \sum_{\ell=1}^k \sigma_\ell = \frac{1}{\chi} \quad \P\text{-a.s.},
\qquad 
\lim_{k\to\infty} \frac{1}{k^{1/\gamma}} \sum_{\ell=1}^k \tau_\ell  = W
\quad \text{ in $\mathbb{P}$-probability},
\end{equation}
with $W$ a stable law random variable on $(0,\infty)$ with exponent $\gamma$. Therefore
\begin{equation}
\label{ETasympalt}
\begin{aligned}
&\lim_{t\to\infty} \frac{1}{t^\gamma}\,N(t) = \lim_{t\to\infty} \frac{1}{t^\gamma}\,N'(t) 
= W^{-\gamma} \quad \text{ in $\mathbb{P}$-probability},\\[0.2cm]
&\lim_{t\to\infty} \frac{1}{t^\gamma}\,T(t) = \lim_{t\to\infty} \frac{1}{t^\gamma}\,T'(t) 
=  \frac{1}{\chi}\,W^{-\gamma} \quad\text{ in $\mathbb{P}$-probability},\\
&\lim_{t\to\infty} t^{1-\gamma}\,\mathbb{P}\big(\CE(t)\big) 
= \lim_{t\to\infty} t^{1-\gamma}\,\mathbb{P}\big(\CE'(t)\big) 
= \frac{1}{\chi}\,\mathbb{E}[W^{-\gamma}],
\qquad t\to\infty.
\end{aligned}
\end{equation}
For the last statement to make sense, we must check the following.

\begin{lemma}{\bf [Finite limits]}
\label{l.2357}
$\mathbb{E}[W^{-\gamma}]<\infty$.
\end{lemma}

\begin{proof}
\label{pr.2361}
Let $W_k = k^{-1/\gamma} \sum_{l=1}^k \tau_l$. Then  $W_k^{-\gamma} \leq k(\max_{1\leq i\leq k} \tau_i^\gamma)^{-1}$ and, since $\tau_i$ are i.i.d.\ random variables,
\begin{equation}
\hspace{-.3cm}
\E[W_k^{-\gamma}] \leq \int_{0}^{\infty} \d x\,\P\left(k\left(\max_{1\leq i\leq k}
\tau_i^\gamma\right)^{-1}>x\right)
= \int_{0}^{\infty} \d x\,\P\left(\tau_1^\gamma<\tfrac{k}{x}\right)^k.
\label{argWgamma}
\end{equation}
To estimate the integral in the right-hand side of \eqref{argWgamma}, we introduce three constants, $T$, $C_1$, $C_2$. Let $\epsilon\in(0,1)$ and choose $T\in\R_+$ such that, for all $t>T$, $|[\P(\tau>t)/(Ct^{-\gamma})]-1|<\epsilon$. Since  $\P(\tau\leq t) =1- \chi^{-1} \sum_{m\in\N_0} K_me_m\,\e^{-e_m t}$, we note that, under assumption \eqref{assalt}, $\tau$  admits a continuous bounded density. Hence there exists a $C_1\in\R_+$ such that $\P(\tau\leq t)<C_1t$. Finally, choose $C_2\in\R_+$ such that $C_2>\max(1,C_1^\gamma)$. Split
\begin{equation}
\label{e2381}
\int_{0}^{\infty} \d x\,\P\left(\tau_1^\gamma<\tfrac{k}{x}\right)^k=\int_{0}^{k/T} 
\d x\,\P\left(\tau_1^\gamma<\tfrac{k}{x}\right)^k
+ \int_{k/T}^{k C_2} \d x\, \P\left(\tau_1^\gamma<\tfrac{k}{x}\right)^k
+ \int_{k C_2}^{\infty} \d x\,\P\left(\tau_1^\gamma<\tfrac{k}{x}\right)^k.
\end{equation}
We estimate each of the three integrals separately. For the first integral, we use the estimate $(1-\P(\tau_1^\gamma\geq\tfrac{k}{x}))^k\leq \exp[-k\P(\tau_1^\gamma\geq\tfrac{k}{x})]$ to obtain 
\begin{equation}
\label{e2387}
\begin{aligned}
\int_{0}^{k/T} \d x\,\P\left(\tau_1^\gamma<\tfrac{k}{x}\right)^k 
= \int_{0}^{k/T} \d x\,\exp\left[-k\,\P\left(\tau_1 \geq \left(\tfrac{k}{x}\right)^{1/\gamma}\right)\right]
\leq \int_{0}^{k/T} \d x\,\e^{-(1-\epsilon)C x} 
\leq \frac{1}{(1-\epsilon)C}.
\end{aligned}
\end{equation}
For the second integral, we note that $t \mapsto t\,\P(\tau_1^\gamma>t)$ is a continuous function on $[\frac{1}{C_2},T]$, and hence attains a minimum value $C_3\in\mathbb{R}_+$ on $[\frac{1}{C_2},T]$. Therefore
\begin{equation}
\label{e2398}
\int_{k/T}^{k C_2} \d x\, \P\left(\tau_1^\gamma<\tfrac{k}{x}\right)^k
= \int_{k/T}^{k C_2} \d x\,\left[1- \P \left(\tau_1^\gamma\geq\tfrac{k}{x}\right)\right]^k
\leq \int_{k/T}^{k C_2} \d x\, \exp\left[-x\left(\frac{k}{x}\P(\tau_1^\gamma \geq \tfrac{k}{x})\right)\right]
\leq\frac{1}{C_3}.
\end{equation}
For the third integral, we compute
\begin{equation}
\label{e2417}
\int_{k C_2}^{\infty} \d x\,\P\left(\tau_1^\gamma<\tfrac{k}{x}\right)^k
\leq \int_{kC_2}^{\infty} \d x\,\left(C_1^\gamma\tfrac{k}{x}\right)^\frac{k}{\gamma}
= \int_{0}^{1/C_2} \d v\,\tfrac{k}{v^2}(C_1^\gamma v)^{\tfrac{k}{\gamma}}
= \frac{C_1^\gamma k}{\tfrac{k}{\gamma}-1}\left(\frac{C_1^\gamma}{C_2}\right)^{\frac{k}{\gamma}-1},
\end{equation}
where in the first equality we substitute $v=\frac{k}{x}$. Since  $C_2>C_1^\gamma$, we see that the right-hand side tends to zero as $k\to \infty$. Hence 
\begin{equation}
\label{e2428}
\E[W_k^{-\gamma}] \leq \frac{1}{(1-\epsilon)(C/\gamma)}
+\frac{1}{C_3}+\frac{C_1^\gamma k}{\tfrac{k}{\gamma}-1}\left(\frac{C_1^\gamma}{C_2}\right)^{k-1},
\end{equation}
and by dominated convergence it follows that $\E[W^{-\gamma}]=\lim_{k\to\infty}
\E[W_k^{-\gamma}]<\infty$.
\end{proof}	

By \eqref{sym}, we have $\hat{a}(\phi)=a(\phi)$ and $\tilde{a}(\phi)=0$ in \eqref{same6}, and so \eqref{Fourier} becomes, with the help of \eqref{ETasympalt},
\begin{equation}
\label{Fourieralt}
\mathbb{E}_{(0,A),(0,A)}\left[ \left(\sum_{i\in\G} a_{T(t)}(0,i)\,a_{T'(t)}(0,i)\right)1_{\CE(t)}\,1_{\CE'(t)} \right]
\asymp t^{-2(1-\gamma)} f(t), \qquad t \to \infty,
\end{equation}
with (recall \eqref{same8})
\begin{equation}
\label{same9alt}
f(t) = \hat{a}_{ct^\gamma}(0,0)
\end{equation}
for some $c \in (0,\infty)$. Here we use that deviations of $T(t)/t^\gamma$ and $T'(t)/t^\gamma$ away from order 1 are stretched exponentially costly in $t$ \cite{EJU19}, and therefore are negligible. Since $t \mapsto \hat{a}_t(0,0)$ is regularly varying at infinity (recall \eqref{ass2}), it follows that
\begin{equation}
\label{pasympalt}
\hat{a}_{ct^\gamma}(0,0) \asymp \hat{a}_{t^\gamma}(0,0), \qquad t \to \infty.
\end{equation}
Combining \eqref{same3} and \eqref{Fourieralt}--\eqref{pasympalt}, we get 
\begin{equation}
\label{Iappr1}
I = \infty \quad \Longleftrightarrow \quad I_{\hat{a},\gamma} = \infty
\end{equation} 
with $I_{\hat{a},\gamma} = \int_1^\infty \d t\,t^{-2(1-\gamma)}\,\hat{a}_{t^\gamma}(0,0)$. Putting 
$s=t^\gamma$, we have 
\begin{equation}
\label{Iappr2}
I_{\hat{a},\gamma} = \int_1^\infty \d s\,s^{-(1-\gamma)/\gamma} \hat{a}_s(0,0),
\end{equation} 
which is precisely the integral defined in \eqref{Idef}.
\end{proof}

\paragraph{{\bf Case $g \neq dg_{\mathrm{FW}}$.}}

To prove that the dichotomy criterion of Lemma \ref{lem:clusgisgfw2} holds for general $g\in\CG$ we need the equivalent of Lemma \ref{lem:comparison}.  Replacing \eqref{gh1}--\eqref{gh2} by \eqref{gh1*}--\eqref{gh2*}, replacing $b^{(1)}$ by $b^{(2)}$ in the proof of Lemma~\ref{lem:comparison}, and using the moment relations in Lemma~\ref{lem1cg9422} instead of the moment relations in Lemma~\ref{lem1cg94}, we see that Lemma \ref{lem:clusgisgfw2} also holds for $g\in\CG$.	

%%%

\subsection{The coexistence case}
\label{ss.cosalt}

In this section we prove the coexistence results stated in Theorem \ref{T.ltb2}. Like for model 1 the proofs hold for general $g\in\CG$ and we need not distinguish between $g=dg_{\mathrm{FW}}$ and $g\neq dg_{\mathrm{FW}}$. For $\rho<\infty$, the argument is given in Section~\ref{sss.cosprf} and proceeds as in Section \ref{ss.cos}. It is organised along the same 4 Steps as the argument for model 1, plus an extra Step 5 that settles the statement in \eqref{vari}. For $\rho=\infty$, the argument is given in Section~\ref{sss.cosprf2} and is also organised along 5 Steps, but structured differently.  In Step 1 we define a set of measures that is preserved under the evolution. In Step 2 we use a coupling argument to show the existence of invariant measures. In Step 3 we show that these invariant measures have vanishing covariances in the seed-bank direction. In Step 4 we use the vanishing covariances to show uniqueness of the invariant measure by coupling. Finally, in Step 5 we show that the unique equilibrium measure is invariant, ergodic and mixing under translations, and is associated. 
 
%%%

\subsubsection{Proof of coexistence for finite seed-bank}
\label{sss.cosprf} 

\paragraph{1. Properties of measures preserved under the evolution.}

For model 2 with $\rho<\infty$, the class of preserved measures is equivalent to $\CR^{(1)}_\theta$ for model 1 and is now defined as follows.

\begin{definition}{\bf [Preserved class of measure]}
\label{pres2}
{\rm Let $\mathcal{R}^{(2)}_\theta$ denote the set of measures $\mu \in \CT$ satisfying, for all $(i,R_i),(j,R_j)\in\G\times\{A,(D_m)_{m\in\N_0}\}$,
\begin{enumerate}
\item[{\rm (1)}] 
$\lim_{t\to \infty} \E_\mu[z_{(i,R_i)}(t)]=\theta$,
\item[{\rm (2)}] 
$\lim_{t\to \infty}\sum_{(k,R_k),(l,R_l)\in \G\times\{A,(D_m)_{m\in\N_0}\}} 
b^{(2)}_t\big((i,R_i),(k,R_k)\big)\,b^{(2)}_t\big((j,R_j),(l,R_l)\big)\,\E_\mu[z_{(k,R_k)}z_{(l,R_l)}]=\theta^{2}$.
\end{enumerate}
} \hfill $\Box$
\end{definition}

\noindent
Like for model 1, properties (1) and (2) of Definition \ref{pres2} hold if and only if
\begin{equation}
\label{q2}
\begin{aligned}
&\lim_{t\to\infty}\E_\mu\left[\left(\sum_{(k,R_k),(l,R_l)\in \G\times\{A,(D_m)_{m\in\N_0}\}}
b^{(2)}_t((0,A),(k,R_k))\,z_{(k,R_k)}-\theta\right)^2\right]=0\\ 
&\text{ for some } (i,R_i)\in\G\times\{A,(D_m)_{m\in\N_0}\}.
\end{aligned}
\end{equation}
Also for model 2 with $\rho<\infty$ we have $\CT_\theta^{\mathrm{erg}}\subset\CR^{(1)}_\theta$. To see why, note for all $t>0$ and $m\in\N_0$, $(x_i(t))_{i\in\G}$ and $ (y_{i,m}(t))_{i\in\G}$ still are stationary time series. Hence with the help of the Herglotz theorem we can define spectral measures $\lambda_A,\ \lambda_{D_m}$ for $m\in\N_0$ as in \eqref{spem}. Let $(RW_t)_{t \geq 0}$ be the random walk evolving according to $b^{(2)}(\cdot,\cdot)$. Introduce the sets
\begin{equation}
\label{dormantm}
\begin{aligned}
\CE(t) &= \big\{\text{at time $t$ the random walk is active}\big\},\\
\CE_m(t) &= \big\{\text{at time $t$ the random walk is dormant with colour } m\big\}.
\end{aligned}
\end{equation}
Note that
\begin{equation}
\label{q3}
\begin{aligned}
\lim_{t\to\infty}\E_\mu&\left[\left(\sum_{(k,R_k),(l,R_l)\in \G\times\{A,(D_m)_{m\in\N_0}\}}
b^{(2)}_t((0,A),(k,R_k))\,z_{(k,R_k)}-\theta\right)^2\right]\\
 &\leq \lim_{t\to \infty} \P_{(0,A)}(\CE(t))\, 
 \E_\mu\left[\left(\sum_{k\in Z^d}\frac{b^{(2)}_t((0,A),(k,A))}{\P_{(0,A)}(\CE(t))}
\,x_{k}-\frac{1}{\P_{(0,A)}(\CE(t))}\frac{\theta_x}{1+\rho}\right)^2\right]\\
&+\sum_{m\in\N_0}\P_{(0,A)}(\CE_m(t))\,
\E_\mu\left[\left(\sum_{k\in Z^d}\frac{b^{(2)}_t((0,A),(k,(D_m))}{\P_{(0,A)}(\CE_m(t))}\,y_{k,m}
-\frac{1}{\P_{(0,A)}(\CE_m(t))}\frac{K_m\theta_{y,m}}{1+\rho}\right)^2\right].
\end{aligned}
\end{equation}
Hence we can use the same argument as in the proof of Lemma~\ref{muinrtheta} to show that $\CT^{\mathrm{erg}}_\theta \subset \CR^{(2)}_\theta$.

Also Lemma~\ref{lem2cg94} carries over after we replace $b^{(1)}(\cdot,\cdot)$ by $b^{(2)}(\cdot,\cdot)$ and $\CR^{(1)}_\theta$ by $\mathcal{R}^{(2)}_\theta$, as defined in \eqref{pres2}.

\paragraph{2. Uniqueness of the equilibrium.}  

To prove uniqueness of the equilibrium for given $\theta$, we use a similar coupling as for model 1 in Section~\ref{ss.cos} in Step 3. Consider two copies of the system in \eqref{gh1*}--\eqref{gh2*} \emph{coupled via their Brownian motions}:     
\begin{eqnarray}
\label{gh1k2}
\d x^k_i(t) &=& \sum_{j \in \G} a(i,j)\, \big[x^k_j (t) - x^k_i(t)\big]\,\d t
+ \sqrt{g(x^k_i(t))}\,\d w_i(t) + \sum_{m\in\N_0}K_me_m\,\big[y^k_{i,m}(t) - x^k_i(t)\big]\,\d t,\\[0.2cm]
\label{gh2k2}
\d y^k_{i,m} (t) &=& e_m\,\big[x^k_i(t) - y^k_{i,m}(t)\big]\,\d t,\qquad m\in\N_0,\qquad k \in \{1,2\}.
\end{eqnarray}
Here, $k$ labels the copy, and the two copies are driven by the same Brownian motions $(w_i(t))_{t \geq 0}$, $i \in 
\G$. As initial measures we choose $\mu^1(0),\mu^2(0) \in \CT^{\mathrm{erg}}_\theta$. 

Let 
\begin{equation}
\label{e18092}
\bar{z}_i(t) = \big(z^1_i(t),z^2_i(t)\big), \qquad z^k_i(t) = \big(x^k_i(t),(y^k_{i,m}(t))_{m\in\N_0}\big), \quad k \in \{1,2\}.
\end{equation} 
By \cite[Theorem 3.2]{SS80}, the coupled system $(\bar{z}_i(t))_{i\in\G}$ has a unique strong solution whose marginals are the single-component systems. Write $\hat\P$ to denote the law of the coupled system, and let $\Delta_i(t) = x^1_i(t)-x^2_i(t)$ and $\delta_{i,m}(t) = y^1_{i,m}(t)-y^2_{i,m}(t)$, $m\in\N_0$. The analogue of Lemma~\ref{L.coup} reads:

\begin{lemma}{\bf [Coupling dynamics $\rho<\infty$]}
\label{L.coup2}
For every $t \geq 0$,
\begin{equation}
\label{couprel2}
\begin{aligned}
\frac{\d}{\d t}\,\hat\E\left[|\Delta_i(t)| + \sum_{m\in\mathbb{N}_0}K_m |\delta_i(t)|\right] 
&= - 2 \sum_{j \in \G} a(i,j)\,\hat\E\left[|\Delta_j(t)|\,
1_{\{\sign\Delta_i(t)\,\neq\,\sign\Delta_j(t)\}}\right]\\
&\qquad - 2\sum_{m\in\N_0}K_me_m\,\hat\E\left[\big(|\Delta_i(t)| + |\delta_{i,m}(t)|\big)\,
1_{\{\sign\Delta_i(t)\,\neq\,\sign\delta_{i,m}(t)\}}\right].
\end{aligned}
\end{equation} 
\end{lemma}

\begin{proof} 
Note that the left-hand side of \eqref{couprel2} is well defined because $\rho<\infty$. The proof of Lemma~\ref{L.coup2} carries over from that of Lemma~\ref{L.coup} after replacing \eqref{gh1}--\eqref{gh2} by \eqref{gh1*}--\eqref{gh2*}.	
\end{proof}

The analogue of Lemma~\ref{L.suc} reads as follows.

\begin{lemma}{\bf [Succesfull coupling $\rho<\infty$]}
\label{L.suc2}
If $a(\cdot,\cdot)$ is transient, then the coupling is successful, i.e., 
\begin{equation}
\lim_{t\to\infty}\hat\E\big[|\Delta_i(t)| + \sum_{m\in\N_0}K_m |\delta_{i,m}(t)|\big]=0,\qquad \forall i\in\G. 
\end{equation} 
\end{lemma}

\begin{proof}
This follows in the same way as in the proof of Lemma~\ref{L.suc}, by defining $-h_i(t)$ as in the right-hand side of \eqref{couprel2}. Using that the second line of \eqref{eps} now holds for $\delta_{i,m}(t)$ and all $m\in\N_0$, we can finish the proof after replacing $b^{(1)}_t(\cdot,\cdot)$ in \eqref{eqsuccoup} by $b^{(2)}_t(\cdot,\cdot)$ and summing over the seed-banks $D_m$,  $m\in\N_0$. 
\end{proof}

\paragraph{3. Stationarity of the equilibrium $\nu_\theta$ and convergence to $\nu_\theta$.} 

Lemma~\ref{lemenu} holds also for $\mu\in \mathcal{R}^{(2)}_\theta$. This follows after replacing $\mu\in \mathcal{R}^{(1)}_\theta$ by $\mu\in \mathcal{R}^{(2)}_\theta$ in the proof of Lemma~\ref{lemenu}, using the equivalent of Lemma~\ref{lem2cg94} and invoking Lemma~\ref{L.suc2} instead of Lemma~\ref{L.suc}.

\paragraph{4. Ergodicity, mixing and associatedness.}

Also Lemma~\ref{lem10} holds, after replacing $b^{(1)}(\cdot,\cdot)$ by $b^{(2)}(\cdot,\cdot)$. The proof even simplifies, since we can invoke the symmetry of $a(\cdot,\cdot)$ in \eqref{aso}.

\paragraph{5. Variances under the equilibrium measure $\nu_\theta$.}

If $\limsup_{m\to\infty} e_m=0$, then the claim in \eqref{vari} is a direct consequence of the proof of Lemma~\ref{lemvar} for $\rho=\infty$. If $\liminf_{m\to\infty} e_m>0$, then the claim follows from the fact that $\mu\in\R_\theta^{(2)}$ and
\begin{equation}
\text{Var}_{\nu_\theta}(y_{0,m})=2\int_{0}^{t}\d s \sum_{k\in \G} 
b^{(2)}_{t-s}\big((0,D_m),(k,A)\big)\,b^{(2)}_{t-s}\big((0,D_m),(k,A)\big)\,\E_\mu[g(x_i(s))].
\end{equation}
Since $e_m>0$ for all $m\in\N_0$ and $\liminf_{m\to\infty} e_m>0$, there is a positive probability that after the first steps the two random walks are both active at $0$, i.e., are both in state $(0,A)$. Hence, for all $m\in\N_0$ there exists a constant $c>0$ such that 
\begin{equation}
\text{Var}_{\nu_\theta}(y_{0,m})\geq c\text{Var}_{\nu_\theta}(x_0).
\end{equation}
Since $\nu_\theta$ is a non-trivial equilibrium, we have $\text{Var}_{\nu_\theta}(x_0)>0$.

%%%

\subsubsection{Proof of coexistence for infinite seed-bank}
\label{sss.cosprf2}

\paragraph{1. Properties of measures preserved under the evolution.}  

For $\rho=\infty$, the class of preserved measures is also given by $\CR_\theta^{(2)}$ (recall Definition \ref{pres2}). We show that if $\mu\in\CT^{\mathrm{erg}}_\theta$ is colour regular, then $\mu\in\CR_\theta^{(2)}$. Let the sets $\CE_m(t)$, $t>0$, $m\in\N_0$, be defined as in \eqref{dormantm}, and define $\lambda_A$ and $\lambda_{D_m}$ analogously to $\eqref{spem}$, like for $\rho<\infty$. The equivalent of \eqref{spec1} is
\begin{equation}
\begin{aligned}
\E_{\mu}&\left[\left(\frac{1}{\P_{(0,A)}(\CE(t))}\sum_{k \in \G} b^{(2)}_t((0,A),(k,A))x_k-\theta_x\right)^2\right]\\
&=\frac{1}{\P_{(0,A)}(\CE(t))^2}\int_{[-\pi,\pi]^d}\E_{(0,A),(0,A)}\left[\e^{-T(t)(1-a(\phi))}1_{\CE(t)}
\e^{-T^\prime(t)(1-\bar{a}(\phi))}1_{\CE^\prime(t)}\right]\d \lambda_A.
\end{aligned}
\end{equation}

Using that $T(t),T^\prime(t)\to\infty$ as $t\to\infty$ (see \eqref{ETasympalt}), that $T(t),\ T^\prime(t),\ \CE(t),\ \CE^\prime(t)$ are asymptotically independent and that $a(\cdot,\cdot)$ is irreducible, we still find
\begin{equation}
\begin{aligned}
\lim_{t\to\infty}\E_{\mu}\left[\left(\frac{1}{\P_{(0,A)}(\CE(t))}
\sum_{k \in \G} b^{(2)}_t((0,A),(k,A))x_k-\theta_x\right)^2\right]= \lambda_A(\{0\})
\end{aligned}
\end{equation}
and, similarly,
\begin{equation}
\begin{aligned}
\lim_{t\to\infty}\E_{\mu}\left[\left(\frac{1}{\P_{(0,A)}(\CE_m(t))}
\sum_{k \in \G} b^{(2)}_t((0,A),(k,A))y_{k,m}-\theta_{y,m}\right)^2\right]= \lambda_{D_m}(\{0\}).
\end{aligned}
\end{equation}
Since $\mu$ is ergodic, we have $\lambda_A(\{0\})=0$ and $\lambda_{D_m}(\{0\})=0$ for all $m\in\N_0$ (recall \eqref{q4}). By the colour regularity,
\begin{equation}
\lim_{t\to\infty}\theta_x\P_{(0,A)}(\CE(t))+\sum_{m\in\N_0}\theta_{y,m}\P_{(0,A)}(\CE_m(t))=\theta.
\end{equation}
Therefore we can rewrite \eqref{q3} as
\begin{equation}
\label{q3alt}
\begin{aligned}
\lim_{t\to\infty}\E_\mu&\left[\left(\sum_{(k,R_k),(l,R_l)\in \G\times\{A,(D_m)_{m\in\N_0}\}}
b^{(2)}_t((0,A),(k,R_k))\,z_{(k,R_k)}-\theta\right)^2\right]\\
&\leq \lim_{t\to \infty} \P_{(0,A)}(\CE(t))\, 
\E_\mu\left[\left(\sum_{k\in Z^d}\frac{b^{(2)}_t((0,A),(k,A))}{\P_{(0,A)}(\CE(t))}
\,(x_{k}-\theta_x)\right)^2\right]\\
&\qquad+\sum_{m\in\N_0}\P_{(0,A)}(\CE_m(t))\,
\E_\mu\left[\left(\sum_{k\in Z^d}\frac{b^{(2)}_t((0,A),(k,(D_m))}{\P_{(0,A)}(\CE_m(t))}\,(y_{k,m}
-\theta_{y,m})\right)^2\right]\\
&=\lim_{t \to \infty} \P_{(0,A)}(\CE(t))\,\lambda_A(\{0\})\ +\ \sum_{m\in\N_0} 
\P_{(0,A)}(\CE_m(t))\,\lambda_{D_m}(\{0\}) = 0.
\end{aligned}
\end{equation}
We conclude that indeed $\mu\in\CR^{(2)}_\theta$.

Like for $\rho<\infty$, Lemma~\ref{lem2cg94} carries over after we replace $b^{(1)}(\cdot,\cdot)$ by $b^{(2)}(\cdot,\cdot)$ and $\CR^{(1)}_\theta$ by $\mathcal{R}^{(2)}_\theta$.

\paragraph{2. Existence of invariant measures $\nu_\theta$ for $\rho=\infty$.}

Since the dynamics for $\rho=\infty$ and $\rho<\infty$ are the same, we can still use the coupling in \eqref{gh1k2}--\eqref{gh2k2}. Also Lemma~\ref{L.coup2} holds for $\rho=\infty$, but if $\rho=\infty$, then the left-hand side of \eqref{couprel2} can become infinite. Therefore we cannot use the line of argument used for model 1 to show that the coupling is successful for arbitrary colour regular initial measures $\mu_1,\ \mu_2\in\CT_\theta^{\mathrm{erg}}$. However, we can prove the following lemma.

\begin{lemma}{\bf [Successful coupling]}
If $\mu_1,\ \mu_2\in\CT_\theta^{\mathrm{erg}}$ are both colour regular and satisfy
\begin{equation}
\label{q5}
\hat\E\left[|\Delta_i(0)| + \sum_{m\in\mathbb{N}_0}K_m |\delta_i(0)|\right]<\infty,
\end{equation} 
then the coupling in \eqref{gh1k2}--\eqref{gh2k2} is successful.
\end{lemma}

\begin{proof}
We proceed similarly as in Step 3 for $\rho<\infty$. Note, in particular, that $h_i(t)$ (recall \eqref{couprel2}) is bounded from above by $\hat\E\left[|\Delta_i(0)| + \sum_{m\in\mathbb{N}_0}K_m |\delta_i(0)|\right]$ (compare with \eqref{e3298}). Also for $\rho=\infty$ we obtain Lemma \ref{C2}. Like for model 1, if we define
\begin{equation}
\label{q7}
\begin{aligned}
E_0\times E_0 = &\left\{\bar{z}\in E\times E\colon\, z_{(i,R_i)}^1(t)\geq z_{(i,R_i)}^2(t)\ 
\forall (i,R_i)\in\G\times\{A,(D_m)_{m\in\N_0}\}\right\}\\
& \cup \left\{\bar{z}\in E\times E\colon\,z_{(i,R_i)}^2(t)\geq z_{(i,R_i)}^1(t)\ \forall (i,R_i)
\in\G\times\{A,(D_m)_{m\in\N_0}\}\right\},
\end{aligned}
\end{equation}
then we find $\lim_{t \to \infty}\P(E_0\times E_0)=1$ and hence the coupled diffusions $(Z^1(t))_{t\geq 0}$ and $(Z^2(t))_{t\geq 0}$ lay on top of each other as $t\to\infty$. However, in \eqref{eqsuccoup} the limiting distribution of $b^{(1)}_{t_n}(\cdot,\cdot)$ was used ``to compensate" the factors $K_m$ in $|\Delta_i|+\sum_{m\in\N_0}K_m|\delta_{i,m}|$. Since, for $\rho=\infty$, $b^{(1)}_{t_n}(\cdot,\cdot)$ does not have a well-defined limiting distribution for the projection on the colour components, we need a different strategy. 
	
To obtain a successful coupling, as before, let $(t_n)_{n\in\N}$ be a subsequence such that $\nu^1_\theta=\lim_{n \to \infty}\CL(Z^1(t_n))$ with $\CL(Z^1(0))=\mu^1$ and $\nu^2_\theta=\lim_{n \to \infty}\CL(Z^2(t_n))$ with $\CL(Z^2(0))=\mu^2$. For $\G=\Z^d$, let $\Lambda_N = [0,N)^d \cap \Z^d$, $N\in\N$. (As noted before, for amenable groups $\G$, $(\Lambda_N)_{N\in\N}$ must be replaced by a so-called F$\phi$lner sequence.) Note that 
\begin{equation}
\E_{\nu^1_\theta}\left[\left(\frac{1}{|\Lambda_N|}\sum_{j\in\Lambda_N}x_j-\theta\right)^2\right]
=\frac{1}{|\Lambda_N|^2}\sum_{i,j\in\Lambda_N}\text{Cov}_{\nu_\theta^1}(x_i,x_j).
\end{equation}
Since $\mu^1$ is colour regular and $\mu^1\in\CT_\theta^{\mathrm{erg}}$, we have $\mu^1\in\CR^{(2)}_\theta$. Hence, by Lemma \ref{lem1cg9422}, 
\begin{equation}
\label{oef}
\begin{aligned}
\text{Cov}_{\nu_\theta^1}(x_i,x_j)&=\lim_{n \to \infty}\text{Cov}_{\mu^1}(x_i(t_n),x_j(t_n))\\
&\leq\,\lim_{n \to \infty}2\|g\|\int_{0}^{t_n} \d s \sum_{k\in \G} 
b^{(2)}_{(t_n-s)}((i,A),(k,A))\,b^{(2)}_{(t-s)}((j,A),(k,A))\\
&\leq \,2\|g\|\int_{0}^{\infty} \d s \sum_{k\in \G} \E_{(i,A),(j,A)}
\left[a_{T(s)}(i,k)\,1_{\CE(s)}\,a_{T^\prime(s)}(j,k)\,1_{\CE^\prime(s)}\right]\\
&\leq \,2\|g\|\int_{0}^{\infty} \d s\ \E_{(i,A),(j,A)}\left[\hat{a}_{T(s)+T^\prime(s)}(i-j,0)\,1_{\CE(s)}\,
1_{\CE^\prime(s)}\right]. 
\end{aligned}
\end{equation}
Since $I_{\alpha,\gamma}<\infty$, we see that the last integral is finite. Since $\lim_{||i-j||\to\infty}\hat{a}_{t}(i-j,0)=0$ for all $t>0$, it follows by transience and  dominated convergence that $\lim_{||i-j||\to\infty}\text{Cov}_{\nu_\theta^1}(x_i,x_j)=0$. Since $\text{Cov}_{\nu_\theta^1}(x_i,x_j)\leq1$ for all $i,j\in\G$, for all $\epsilon>0$ there exists an $L\in\N$ such that
\begin{equation}
\begin{aligned}
&\lim_{N \to \infty}\E_{\nu^1_\theta}\left[\left(\frac{1}{|\Lambda_N|}\sum_{j\in\Lambda_N}x_j-\theta\right)^2\right]
=\lim_{N \to \infty}\frac{1}{|\Lambda_N|^2}\sum_{i,j\in\Lambda_N}\text{Cov}_{\nu_\theta^1}(x_i,x_j)\\
&\qquad =\lim_{N \to \infty}\frac{1}{|\Lambda_N|^2}\sum_{\substack{i,j\in\Lambda_N\\ \|i-j\|\leq L}}
\text{Cov}_{\nu_\theta^1}(x_i,x_j)+\frac{1}{|\Lambda_N|^2}\sum_{\substack{i,j\in\Lambda_N\\ \|i-j\|>L}}
\text{Cov}_{\nu_\theta^1}(x_i,x_j)\\
&\qquad \leq \lim_{N\to\infty}\frac{|\{i,j\in\Lambda_N\colon\, \|i-j\|\leq L\}|}{|\Lambda_N|^2}
+ \epsilon \lim_{N\to\infty}\frac{|\{i,j\in\Lambda_N\colon\, \|i-j\|> L\}|}{|\Lambda_N|^2} < \epsilon.
\end{aligned}
\end{equation}
We conclude that 
\begin{equation}
\label{good1}
\lim_{N \to \infty}\E_{\nu^1_\theta}\left[\left(\frac{1}{\Lambda_N}\sum_{j\in\Lambda_N}x_j-\theta\right)^2\right]=0,
\end{equation}
and the same holds for $\nu_2^\theta$. Let $\lim_{n \to \infty}\CL(\bar{Z}(t_n))=\bar{\nu}_\theta$ such that 
$\lim_{n \to \infty}\CL(Z^1(t_n))=\nu_\theta^1$ and $\lim_{n \to \infty} \CL(Z^2(t_n))$ $= \nu_\theta^2$. Then by translation invariance of $\bar{\nu}_\theta$ and the fact that $\bar{\nu}_\theta(E_0\times E_0)=1$, we find
\begin{equation}
\label{sc2}
\int_{E\times E} \d \bar{\nu}_\theta |\Delta_i|
=\int_{E_0\times E_0} \d \bar{\nu}_\theta \frac{1}{|\Lambda_N|}\sum_{j\in\Lambda_N} |x^1_j-x^2_j|\\
\leq \int_{E_0} \d \nu^1_\theta \left|\frac{1}{|\Lambda_N|} \sum_{j\in\Lambda_N} x^1_j-\theta\,\right|
+\int_{E_0} \d \nu^2_\theta \left|\frac{1}{|\Lambda_N|} \sum_{j\in\Lambda_N}x^2_j-\theta\,\right|.
\end{equation}
Letting $N\to\infty$, we see by translation invariance of $\bar{\nu}_\theta$ that $\E_{\bar{\nu}_\theta}\left[|\Delta_i|\right]=0$ for all $i\in\G$. 
	
The result in \eqref{good1} holds also for $x_i$ replaced by $y_{i,m}$, $m\in\N_0$, since the integral in \eqref{oef} can only become smaller when we start from a dormant site. Replacing $|\Delta_i|$ in \eqref{sc2} by $|\delta_{i,m}|$, we obtain, for all $m\in\N_0$,  
\begin{equation}
\label{sc3}
\begin{aligned}
\E_{\bar{\nu}_\theta}\left[|\delta_{i,m}|\right]=0, \qquad\forall\, m\in\N_0.
\end{aligned}
\end{equation}
We conclude that the coupling is successful.
\end{proof}

Let $(S_t)_{t \geq 0}$ denote the semigroup associated with \eqref{gh1*}--\eqref{gh2*}. To prove the existence of an invariant measure, note that $E\times E$ is a compact space. Hence, if $t_n\to\infty$, then the sequence $\mu S_{t_n}$ has a convergent subsequence. In Lemma \ref{inv} below we show that each weak limit point of the sequence $\mu S_{t_n}$ is invariant under the evolution of \eqref{gh1*}--\eqref{gh2*}. 

\begin{lemma}{\bf [Invariant measure]}
\label{inv}
Suppose that $\mu\in\CR_\theta^{(2)}$ and that $\mu$ is colour regular. If $t_n\to\infty$ and $\mu S_{t_n}\to\nu_\theta$, then $\nu_\theta$ is an invariant measure under the evolution in \eqref{gh1*}--\eqref{gh2*}.
\end{lemma}

\begin{proof}
Fix $s>0$. Let $\mu_1=\mu$ and $\mu_2=\mu S_s$. We couple $\mu_1$ and $\mu_2$ via their Brownian motions (see \eqref{gh1k2}--\eqref{gh2k2}). Note that, by the SSDE in \eqref{gh1*}--\eqref{gh2*},
\begin{equation}
\label{q6}
\begin{aligned}
\hat\E\left[|\Delta_i(0)| + \sum_{m\in\mathbb{N}_0}K_m |\delta_{i,m}(0)|\right]
&=\E\left[|x_i(0)-x_i(s)| + \sum_{m\in\mathbb{N}_0}K_m |y_{i,m}(0)-y_{i,m}(s)|\right]\\
&= \E\Bigg[\Bigg|\int_0^s\sum_{j \in \G} a(i,j)\,[x_j (r) - x_i(r)]\,\d r
+\int_0^s \sqrt{g(x_i(r))}\,\d w_i(r)\\
&\qquad+\,\int_0^s\sum_{m\in\N_0} K_me_m\, [y_{i,m}(r) - x_i(r)]\,\d r\,\Bigg|\\
&\qquad+\sum_{m\in\N_0}K_m \int_0^s |e_m[y_{i,m}(r) - x_i(r)]|\,\d r\Bigg].
\end{aligned}
\end{equation}
Using that all rates are finite and that, by Knight's theorem (see \cite[Theorem V.1.9 p.183]{RY99}), we can write the Brownian integral as a time-transformed Brownian motion, we see that $\hat\E[|\Delta_i(0)| + \sum_{m\in\mathbb{N}_0}K_m |\delta_{i,m}(0)|]<\infty$. Hence, by Lemma \ref{q5}, we can successfully couple $\mu^1$ and $\mu^2$, and $\lim_{n \to \infty}\mu^2 S_{t_n}=\lim_{n\to\infty}\mu S_s S_{t_n}=\nu_\theta$. By the Feller property of the SSDE in \eqref{gh1*}--\eqref{gh2*}, it follows that
\begin{eqnarray}
\label{mo6}
\nu_\theta S_s=\lim_{n\to\infty}\mu(t_{n})S_s=\lim_{n\to\infty}\mu S_{t_{n}}S_s
=\lim_{n\to\infty}\mu S_s S_{t_{n}}=\nu_\theta.
\end{eqnarray}
We conclude that $\nu_\theta$ is indeed an invariant measure for the SSDE in \eqref{gh1*}--\eqref{gh2*}.
\end{proof}

\paragraph{3. Invariant measures have vanishing covariances in the seed-bank direction for $\rho=\infty$.}

In this step we prove that an invariant measure $\nu_\theta$ has vanishing variances in the seed-bank direction. In Step 5 we use this property to successfully couple any two invariant measures.

\begin{lemma}{\bf [Deterministic deep seed-banks]}
\label{lemvar}
If $\nu_\theta=\lim_{n\to\infty}\mu S_{t_n}$ for some colour regular $\mu\in\CR_\theta^{(2)}$ and $t_n\to\infty$, then
\begin{equation}
\label{varpre}
\lim_{m\to\infty} \text{\rm Var}_{\nu_\theta}[y_{i,m}]=0 \qquad \forall i\in\G.
\end{equation}
\end{lemma}

\begin{proof} 
Since $\nu_\theta$ is translation invariant, it is enough to show that $\lim_{m\to\infty} \text{Var}_{\nu_\theta}[y_{0,m}]=0$.  Since $\mu(0)\in\mathcal{R}_\theta^{(2)}$, it follows from Lemma~\ref{lem1cg9422} that 
\begin{equation}
\label{var}
\begin{aligned}
\lim_{m\to\infty} \text{Var}_{\nu_\theta}[y_{0,m}]
&=\lim_{m\to\infty}\lim_{n\to\infty} \E_{\mu}\left[\left(y_{0,m}(t_{n})-\E_{\mu}[y_{0,m}(t_{n})]\right)^2\right]\\
&=\lim_{m\to\infty} \lim_{n\to\infty} 2\int_{0}^{t_{n}} \d s \sum_{k\in \G} 
b^{(2)}_{(t_{n}-s)}((0,D_m),(k,A))\,b^{(2)}_{(t_{n_k}-s)}((0,D_m),(k,A))\,\E_z[g(x_k(s))].
\end{aligned}
\end{equation}	 
Since $g$ is positive and bounded, it is therefore enough to prove that 
\begin{equation}
\label{rw}
\lim_{m\to\infty}\lim_{n\to\infty}\int_{0}^{t_n} \d u \sum_{k\in \G} 
b^{(2)}_{u}((0,D_m),(k,A))\,b^{(2)}_{u}((0,D_m),(k,A))=0.
\end{equation}
Recall (see e.g.\ \eqref{oef}) that $b^{(2)}_{u}((0,D_m),(k,A))\,b^{(2)}_{u}((0,D_m),(k,A))$ is the probability that two random walks, denoted by $RW$ and $RW^\prime$ and moving according to $b^{(2)}(\cdot,\cdot)$, are at time $u$ at the same site and both active. Define
\begin{equation}
\label{tau}
\tau=\big\{t\geq 0: RW(t)=RW^\prime(t)=(i,A)\text{ for some }i\in\G\big\}.
\end{equation}
Then we can rewrite the left-hand side of \eqref{rw} as
\begin{equation}
\label{var1}
\begin{aligned}
\lim_{m\to\infty}\lim_{n\to\infty}&\int_{0}^{t_n} \d u\, \sum_{k\in \G} 
b^{(2)}_{u}((0,D_m),(k,A))\,b^{(2)}_{u}((0,D_m),(k,A))\\
&=\lim_{m\to\infty}\lim_{n\to\infty}
\int_{0}^{t_n}\d u\, \E_{(0,D_m),(0,D_m)}\left[\sum_{k\in\G} 
1_{\{RW(u)=k\}}\,1_{\{RW^\prime(u)=k\}}\,1_{\CE(u)}\,1_{\CE^\prime(u)}\right]\\
&=\lim_{m\to\infty}\lim_{n\to\infty}
\int_{0}^{t_n}\d u\, \E_{(0,D_m),(0,D_m)}\left[\sum_{k\in\G} 1_{\{RW(u)=k\}}\,
1_{\{RW^\prime(u)=k\}}\,1_{\CE(t)}\,1_{\CE^\prime(t)}\left(1_{\{\tau<\infty\}}
+1_{\{\tau=\infty\}}\right)\right]\\
&=\lim_{m\to\infty}\lim_{n\to\infty}
 \E_{(0,D_m),(0,D_m)}\left[1_{\{\tau<\infty\}}\E_{(0,D_m),(0,D_m)}\left[\int_{0}^{t_n}\d u\sum_{k\in\G} 
 1_{\{RW(u)=k\}}1_{\{RW^\prime(u)=k\}}1_{\CE(u)}1_{\CE^\prime(u)} \mid \CF_\tau\right]\right]\\
&=\lim_{m\to\infty}\lim_{n\to\infty}
 \E_{(0,D_m),(0,D_m)}\left[1_{\{\tau<\infty\}}\E_{(0,A),(0,A)}\left[\int_{0}^{t_n-\tau}\d u \sum_{k\in\G}  
 1_{\{RW(u)=k\}}1_{\{RW^\prime(u)=k\}}1_{\CE(u)}1_{\CE^\prime(u)}\right]\right]\\
&\leq\lim_{m\to\infty}\lim_{n\to\infty}
\E_{(0,D_m),(0,D_m)}\left[1_{\{\tau<\infty\}}\E_{(0,A),(0,A)}\left[\int_{0}^{\infty}\d u \sum_{k\in\G}  
1_{\{RW(u)=k\}}1_{\{RW^\prime(u)=k\}}1_{\CE(u)}1_{\CE^\prime(u)}\right]\right]\\
&=\lim_{m\to\infty} \P_{(0,D_m),(0,D_m)}\left(\tau<\infty\right)I_{\hat{a},\gamma},
\end{aligned}
\end{equation}
where we use that $I_{\hat{a},\gamma}<\infty$, the strong Markov property, and the fact that for $\tau=\infty$ the product of the indicators equals $0$ for all $u\in\R_{\geq 0}$. Therefore \eqref{varpre} holds if
\begin{equation}
\lim_{m\to\infty}\P_{(0,D_m),(0,D_m)}\left(\tau<\infty\right)=0.
\end{equation}

Define 
\begin{equation}
\label{tauster}
\tau^*=\inf\big\{t\geq 0:\text{ both }RW \text{ and }RW^\prime \text{ are active at time } t\big\}.
\end{equation}
Note that $\tau^*\leq \tau$. Theorefore we can write (recall that in model 2 the random walk kernel $a(\cdot,\cdot)$ is assumed to be symmetric),
\begin{equation}
\begin{aligned}
&\lim_{m\to\infty}\P_{(0,D_m),(0,D_m)}\left(\tau<\infty\right)\\
&=\lim_{m\to\infty}\E_{(0,D_m),(0,D_m)}[1_{\{\tau<\infty\}}]\\
&=\lim_{m\to\infty} \E_{(0,D_m),(0,D_m)}\left[1_{\{\tau^*<\infty\}}\,
\E_{(0,D_m)^2}\left[1_{\{\tau<\infty\}} \mid \CF_{\tau^*}\right]\right]\\
&=\lim_{m\to\infty} \E_{(0,D_m),(0,D_m)}\left[\E^{RW(\tau^*),RW^\prime(\tau^*)}
\left[1_{\{\tau<\infty\}}\right]\right]\\
&=\lim_{m\to\infty}\sum_{k,l\in\G}\P_{(0,D_m),(0,D_m)}\left(RW(\tau^*)=(k,A),RW^\prime(\tau^*)
=(l,A)\right)\,\E_{(k,A),(l,A)}\left[1_{\{\tau<\infty\}}\right]\\
&=\lim_{m\to\infty}\sum_{k,l\in\G} \E_{(0,D_m),(0,D_m)}[\hat{a}_{T(\tau^*)}(0,k)\,
\hat{a}_{T^\prime(\tau^*)}(0,l)]\,\E_{(0,A),(l-k,A)}\left[1_{\{\tau<\infty\}}\right]\\
&=\lim_{m\to\infty}\sum_{k,l\in\G} \E_{(0,D_m),(0,D_m)}[\hat{a}_{T(\tau^*)}(0,k)\,\hat{a}_{T^\prime(\tau^*)}(-k,l-k)]
\,\E_{(0,A),(l-k,A)}\left[1_{\{\tau<\infty\}}\right]\\
&=\lim_{m\to\infty}\sum_{k,j\in\G} \E_{(0,D_m),(0,D_m)}[\hat{a}_{T(\tau^*)}(0,-k)\,
\hat{a}_{T^\prime(\tau^*)}(-k,j)]\,\E_{(0,A),(j,A)}\left[1_{\{\tau<\infty\}}\right]\\
&=\lim_{m\to\infty}\sum_{j\in\G} \E_{(0,D_m),(0,D_m)}[\hat{a}_{T(\tau^*)+T^\prime(\tau^*)}(0,j)]
\,\E_{(0,A),(j,A)}\left[1_{\{\tau<\infty\}}\right]\\
&=\lim_{m\to\infty}\sum_{{j\in\G} \atop {\|j\|\leq L}} \E_{(0,D_m),(0,D_m)}[\hat{a}_{T(\tau^*)+T^\prime(\tau^*)}(0,j)]
\,\E_{(0,A),(j,A)}\left[1_{\{\tau<\infty\}}\right]\\
&\qquad+\lim_{m\to\infty}\sum_{ {j\in\G} \atop {\|j\|> L}} \E_{(0,D_m),(0,D_m)}[\hat{a}_{T(\tau^*)+T^\prime(\tau^*)}(0,j)]\,
\E_{(0,A),(j,A)}\left[1_{\{\tau<\infty\}}\right].
\end{aligned}
\end{equation}
To prove that the expression in the right-hand side tends to zero, we fix $\epsilon>0$ and prove that there exists an $L\in\N$ such that both sums are smaller that $\frac{\epsilon}{2}$.

\paragraph{Claim 1: There exists an $L$ such that  $\lim_{m\to\infty}\sum_{j\in\G, \|j\|> L}\E_{(0,D_m)^2}[ \hat{a}_{T(\tau^*)+T^\prime(\tau^*)}(0,j)]\,\E_{(0,A),(j,A)}[1_{\{\tau<\infty\}}]<\frac{\epsilon}{2}$.}

Using the symmetry of the kernel $a(\cdot,\cdot)$ in model 2, we find
\begin{equation}
\begin{aligned}
\E_{(0,A),(j,A)}\left[1_{\{\tau<\infty\}}\right]
&=\E_{(0,A),(j,A)}\left[\int_0^\infty \d s\, 1_{\{\tau\in\d s\}}\right]\\
&\leq\E_{(0,A),(j,A)}\left[\int_0^\infty \d s\sum_{k\in\G} 1_{\CE(s)}
1_{\CE^\prime(s)}1_{\{RW=k\}}1_{\{RW^\prime=k\}}\right]\\
&\leq \E_{(0,A),(j,A)}\left[\int_0^\infty \d s\, \sum_{k\in\G}
\hat{a}_{T(s)}(0,k)\,\hat{a}_{T^\prime(s)}(j,k)\,1_{\CE(s)}\,
1_{\CE^\prime(s)}\right]\\
&\leq \E_{(0,A),(j,A)}\left[\int_0^\infty \d s\, \hat{a}_{T(s)+T^\prime(s)}(j,0)1_{\CE(s)}
1_{\CE^\prime(s)}\right].
\end{aligned}
\end{equation}
The last integral in the right-hand side is dominated by $I_{\hat{a},\gamma}$ (recall \eqref{Idefalt}). Since, for all $t\in\R_{\geq 0}$,
\begin{equation}
\lim_{\|j\|\to\infty}\hat{a}_t(0,j)=0,
\end{equation}
it follows by dominated convergence that for each $\epsilon>0$ we can find an $L$ such that, for all $\|j\|>L$, 
\begin{equation}
\E_{(0,A),(j,A)}\left[1_{\{\tau<\infty\}}\right]<\tfrac{\epsilon}{2}.
\end{equation}
Hence, for $L$ sufficiently large, we find
\begin{equation}
\begin{aligned}
&\lim_{m\to\infty}\sum_{j\in\G, |\|j\||> L}\E_{(0,D_m),(0,D_m)}[ \hat{a}_{T(\tau^*)+T^\prime(\tau^*)}(0,j)]
\left[\E_{(0,A),(j,A)}\left[1_{\{\tau<\infty\}}\right]\right]\\
&\leq \lim_{m\to\infty}\tfrac{\epsilon}{2}\sum_{j\in\G, \|j\|> L}\E_{(0,D_m),(0,D_m)}
[\hat{a}_{T(\tau^*)+T^\prime(\tau^*)}(0,j)]\leq \tfrac{\epsilon}{2}.
\end{aligned}
\end{equation}

\paragraph{Claim 2: For $L$ given as in Claim 1, $\lim_{m\to\infty}\sum_{j\in\G, \|j\|\leq L}\E_{(0,D_m)^2}[ \hat{a}_{T(\tau^*)+T^\prime(\tau^*)}(0,j)]\,\E^{(0,A),(j,A)}[1_{\{\tau<\infty\}}]<\frac{\epsilon}{2}$.}

For the first sum, note that 
\begin{equation}
\begin{aligned}
\lim_{m\to\infty}&\sum_{{j\in\G} \atop {\|j\|\leq L}} 
\E_{(0,D_m),(0,D_m)}[ \hat{a}_{(T(\tau^*)+T^\prime(\tau^*))}(0,j)]
\,\E_{(0,A),(j,A)}\left[1_{\{\tau<\infty\}}\right]\\
&\leq \lim_{m\to\infty}\sum_{ {j\in\G} \atop {\|j\|\leq L}}
\E_{(0,D_m),(0,D_m)}[\hat{a}_{T(\tau^*)+T^\prime(\tau^*)}(0,j)]\\
&= \lim_{m\to\infty}\sum_{ {j\in\G} \atop {\|j\|\leq L}}
\E_{(0,A),(0,D_m)}[\hat{a}_{T(\tau^*)+T^\prime(\tau^*)}(0,j)],
\end{aligned}
\end{equation}
where in the last equality we condition on the first time one of the two random walks wakes up, and use the strong Markov property. We will show that the right-hand side tends to zero as $m\to\infty$. Recall that we assumed \eqref{assalt}: $e_m\sim Bm^{-\beta}$ for $\beta>0$. Note that, in order for the random walks to be both active at the same time, the random walk starting in $(0,D_m)$ has to become active at least once. Hence, for all $t \geq 0$, we have 
\begin{equation}
\label{sleep}
\lim_{m\to\infty}\P_{(0,D_m),(0,A)}(\tau^*\leq t)\leq \lim_{m\to\infty} 1-e^{-e_m t}=0.
\end{equation}
By \eqref{ETasympalt} and \cite{EJU19}, we also have for the random walk starting in $(0,A)$  that
\begin{equation}
\label{act}
\lim_{t\to\infty}T(t)\sim ct^\gamma.
\end{equation}

Fix $\epsilon>0$. Since $\lim_{t \to \infty}\hat{a}_t(0,j)=0$ for all $j\in\G$, we can find a $T^\star$ such that, for all $t>T^\star$,
\begin{equation}
\sum_{ {j\in\G} \atop {\|j\|\leq L}} \hat{a}_{t}(0,j)<\tfrac{\epsilon}{6}.
\end{equation}
By \eqref{act}, we can find a $\tilde{t}\in\R_{\geq 0}$ such that $\P_{(0,A)}(T(\tilde{t})>T^\star)\geq 1-\frac{\epsilon}{6}$. By \eqref{sleep}, we can find an $M\in\N_0$ such that for all $m>M$,
\begin{equation}
\lim_{m\to\infty}\P_{(0,D_m),(0,A)}(\tau^*\leq \tilde{t})< \tfrac{\epsilon}{6},
\end{equation}
and hence
\begin{equation}
\begin{aligned}
\lim_{m\to\infty}\sum_{ {j\in\G} \atop {\|j\|\leq L}} \E_{(0,A),(0,D_m)}[ \hat{a}_{T(\tau^*)+T^\prime(\tau^*)}(0,j)]
< \tfrac{\epsilon}{6}+\tfrac{\epsilon}{6}+\tfrac{\epsilon}{6} = \tfrac{\epsilon}{2}.
\end{aligned}
\end{equation}
\end{proof}

\paragraph{4. Uniqueness of the invariant measure $\nu_\theta$ when $\rho=\infty$.}

\begin{lemma}{\bf[Uniqueness of and convergence to $\nu_\theta$.]}
\label{lemuniq}
For all $\theta\in(0,1)$ there exists a unique invariant measure $\nu_\theta$ such that $\lim_{t\to\infty}\mu(t)=\nu_\theta$ for all colour regular $\mu(0)\in\CT^{\mathrm{erg}}_\theta$. 
\end{lemma}

\begin{proof} 
Suppose that $\nu_\theta^1$ and $\nu_\theta^2$ and are two different weak limit points of $\mu(t_n)$ as $t_n\to\infty$, and that $\mu\in\CR_\theta^{(2)}$ is colour regular. Let $(\bar{Z}(t))_{t\geq 0}=(Z^1(t),Z^2(t))_{t\geq 0}$ be the coupled process from \eqref{gh1k}--\eqref{gh2k} with $\CL(\bar{Z}(0))=\bar{\nu}_\theta$, $\CL(Z^1(0))=\nu^1_\theta$ and $\CL(Z^2(0))=\nu^2_\theta$. Define the process $Y^1$ by
\begin{equation}
\begin{aligned}
&Y^{1} =\left(Y^1(m)\right)_{m\in\{-1\}\cup\N_0},\\
&Y^1(-1) = (x^1_i(0))_{i\in\G},\qquad Y^1(m) = (y^1_{i,m}(0))_{i\in\G}\text{ for }{m\in\N_0}.
\end{aligned}
\end{equation}
Thus, $Y^1$ has state space $[0,1]^\G$ and $\CL(Y^1)=\CL(Z^1(0))=\nu^1_\theta$. We can interpret $Y^1$ as a process that describes the states of the population \emph{in the seed-bank direction}. Similarly, define the process $Y^2$ by
\begin{equation}
\begin{aligned}
&Y^{2} =\left(Y^2(m)\right)_{m\in\{-1\}\cup\N_0},\\
&Y^2(-1) = (x^2_i(0))_{i\in\G}, \qquad Y^2(m) = (y^2_{i,m}(0))_{i\in\G}\text{ for }{m\in\N_0}.
\end{aligned}
\end{equation}
Thus, $Y^2$ has state space $[0,1]^\G$ and $\CL(Y^2)=\CL(Z^2(0))=\nu^2_\theta$. 

Define the $\sigma$-algebra's $\CB^1_M$ and $\CB^1$, respectively, $\CB^2_M$ and $\CB^2$ by 
\begin{equation}
\label{tailcolorpr}
\CB^k = \cap_{M\in\N_0} \CB^k_M, \qquad \CB^k_M 
= \sigma\big(\ y^k_{i,m}\colon\,i \in \G,\,m \geq M\big),\quad k\in\{1,2\}.
\end{equation}
Here, $\CB^1$ and $\CB^2$ are \emph{the tail-$\sigma$-algebras in the seed-bank direction.} By Lemma \ref{lemvar}, we have
\begin{equation}
\label{taildis}
\lim_{m\to\infty}\CL_{\nu^1_\theta}(y_{i,m}) = \lim_{m\to\infty}\CL_{\nu^2_\theta}(y_{i,m})=\delta_\theta.
\end{equation} 
Hence, $\CB^1=\CB^2$, both are trivial, and $\nu^1_\theta$ and $\nu^2_\theta$ agree on $\CB$. Therefore Goldstein's Theorem \cite{G78} implies that there exists a successful coupling of $Y^1$ and $Y^2$. Consequently, there exists a random variable $T^{\mathrm{coup}}\in\{-1\}\cup\N_0$ such that, for all $m\geq T^{\mathrm{coup}}$, $Y^1(m)=Y^2(m)$, i.e., $|\delta_{i,m}(0)|=0$ for all $i\in\G$ and $\P(T^{\mathrm{coup}}<\infty)=1$. Hence 
\begin{equation}
\label{63}
\hat\E\left[|\Delta_i(0)| + \sum_{m\in\mathbb{N}_0}K_m |\delta_i(0)|\right] 
=\hat\E\left[|\Delta_i(0)| + \sum_{m=0}^{T^{\mathrm{coup}}}K_m |\delta_i(0)|\right].
\end{equation}
However, we cannot conclude that the left-hand side of \eqref{63} is finite. Therefore, let $\bar{\nu}_\theta|_{\{T^{\mathrm{coup}}<T\}}$ denote the restriction of the measure $\bar\nu_\theta$ to the set ${\{T^{\mathrm{coup}}<T\}}$. Since ${\{T^{\mathrm{coup}}<T\}}$ is a translation-invariant event in the spatial direction, the measure $\bar\nu_\theta|_{\{T^{\mathrm{coup}}<T\}}$ is translation invariant. Moreover,
\begin{equation}
\label{63a}
\hat\E_{\bar{\nu}_\theta|_{\{T^{\mathrm{coup}}<T\}}}\left[|\Delta_i(0)| 
+ \sum_{m\in\mathbb{N}_0}K_m |\delta_i(0)|\right] 
= \hat\E_{\bar{\nu}_\theta|_{\{T^{\mathrm{coup}}<T\}}}\left[|\Delta_i(0)| 
+ \sum_{m=0}^{T}K_m |\delta_i(0)|\right]<\infty.
\end{equation}    
Therefore we can use the dynamics in \eqref{couprel2} and conclude that, for all $T\in\N$, $\hat\P_{\bar{\nu}_\theta|_{\{T^{\mathrm{coup}}<T\}}}(E_0\times E_0)=1$ (recall \eqref{q7}). Since $\lim_{T\to\infty}\bar{\nu}_\theta|_{\{T^{\mathrm{coup}}<T\}}=\bar{\nu}_\theta$, it follows that 
\begin{equation}
\hat\P_{\bar{\nu}_\theta}(E_0\times E_0)=1.
\end{equation}

By \eqref{sc2} and \eqref{sc3}, we conclude that $\nu_\theta^1=\nu_\theta^2$ and hence that all weak limit points of $(\mu(t))_{t\geq 0}$ are the same. Suppose now that $\mu^1(0)\in\CT^{\mathrm{erg}}_\theta$ and $\mu^2(0)\in\CT^{\mathrm{erg}}_\theta$ are two different colour regular initial measures. By the above argument, we know that $\lim_{t\to\infty} \mu^1(t)=\nu^1_\theta$ and $\lim_{t \to \infty} \mu^2(t)=\nu_\theta^2$. By Lemma~\ref{lemvar}, we know that $\nu_\theta^1$ and $\nu_\theta^2$ have the same trivial tail-$\sigma$-algebras in the seed-bank direction. Hence, repeating the above argument, we find that $\nu_\theta^1=\nu_\theta^2$. We conclude that for each colour regular initial measure $\mu\in\CT^{\mathrm{erg}}_\theta$ the SSDE in \eqref{gh1*}--\eqref{gh2*} converges to a unique non-trivial equilibrium measure $\nu_\theta$.  
\end{proof}

\paragraph{5. Ergodicity, mixing and associatedness.}

The equivalent of Lemma \ref{lem10} for $\rho=\infty$ follows in the same way as for $\rho<\infty$.

%%%

\subsection{Proof of the dichotomy}
\label{ss.recap2}

Theorem~\ref{T.ltb2}(I)(a) follows from Lemma \eqref{L.suc2} and Steps 3-5 in Section \ref{sss.cosprf}. The equality $\E_{\nu_\theta}[x_0] = \E_{\nu_\theta}[y_{0,m}] = \theta$, $m \in \N_0$, follows from \eqref{gh1*}--\eqref{gh2*},
the fact that $\nu_\theta$ is an equilibrium measure, and the preservation of $\theta$ (see Section~\ref{ss.scal2}). 
Theorem~\ref{T.ltb2}(I)(b) follows by combining Lemma \ref{lem:clusgisgfw2} with the analogue of Lemma~\ref{lem:comparison}. Theorem~\ref{T.ltb2}(II) follows from Lemmas~\ref{lem:clusgisgfw2}, \ref{lemvar}, \ref{lemuniq}, the analogue of Lemma~\ref{lem:comparison}, and Step 6 in Section \ref{sss.cosprf2}. The equality $\E_{\nu_\theta}[x_0] = \E_{\nu_\theta}[y_{0,m}] = \theta$, $m \in \N_0$, follows from \eqref{dens} in Step 1 of Section \ref{sss.cosprf2}.

Corollary \ref{threereg}(1) corresponds to $\gamma\in(1,\infty)$ and $\rho<\infty$, and migration dominates. Corollary \ref{threereg}(2) corresponds to $\gamma\in[\frac{1}{2},1]$ and $\rho=\infty$, and $I_{\hat{ a},\gamma}$ shows in interplay between migration and seed-bank. Corollary \ref{threereg}(3) corresponds to $\gamma\in(0,\frac{1}{2},1)$ and $\rho=\infty$, and the seed-bank dominates: $I_{\hat{a},\gamma}<\infty$ because $\hat{a}_t(0,0)\leq 1$.

%%%

\subsection{Different dichotomy for asymmetric migration}
\label{asymmig}

It remains to explain how the counterexample below Theorem~\ref{T.ltb2} arises. We focus on the case when $\rho<\infty$, which implies $\mathbb{E}(\tau)<\infty$, but we assume $\mathbb{E}(\tau^2)=\infty$. Therefore the central limit theorem does not hold for $T(t),\ T^\prime(t)$, and $\Delta(t) \gg \sqrt{M(t)}$. Hence \eqref{fdef} must be replaced by
\begin{equation}
\label{newf}
f(t) = \frac{1}{(2\pi)^d} \int_{[-\pi,\pi]^d} \d\phi\, \e^{-[1+o(1)]\,2Bt\,[1-\hat{a}(\phi)]}\,
\mathbb{E}\left[\cos\Big(\Delta(t)\tilde{a}(\phi)\Big)\right].
\end{equation}
The key observation is that if $\tilde{a}(\phi) \neq 0$ (due to the asymmetry of $a(\cdot,\cdot)$; recall \eqref{same6}), 
then the expectation in \eqref{newf} can change the integrability properties of $f(t)$.

Under the assumption that $\tau$ has a \emph{one-sided stable distribution} with parameter $\gamma \in (1,2)$, we have \eqref{ETasymp} with $A=\chi/(1+\rho)$ and $B=1/(1+\rho)$, while there exists a constant $C \in (0,\infty)$ such that (see \cite[Chapter XVII]{F71})
\begin{equation}
\label{e2487}
\mathbb{E}[\cos(\Delta(t)\tilde{a}(\phi))] = \e^{-[1+o(t)]\,At |C\tilde{a}(\phi)|^\gamma}.
\end{equation}
Substituting \eqref{e2487} into \eqref{newf}, we see that for large $t$ the contribution  to $f(t)$ comes from $\phi$ such that $\hat{a}(\phi)\to 1$ and $\tilde{a}(\phi)\to 0$. By our choice of the migration kernel in \eqref{achoice}, this holds as $\phi = (\phi_1,\phi_2) \to (0,0)$.  Using that  $1-\hat{a}(\phi) \sim \tfrac12 (\phi_1^2+\phi_2^2)$ and $\tilde{a}(\phi) \sim \tfrac12\eta(\phi_1+\phi_2)$ for $(\phi_1,\phi_2) \to (0,0)$, we find that \eqref{newf} equals
\begin{equation}
\label{e2494}
f(t) = \frac{1}{(2\pi)^2} \int_{[-\pi,\pi]^2} \d\phi\,\e^{-[1+o(1)]\,
\{Bt(\phi_1^2+\phi_2^2)+ At[|\tfrac{1}{2} C\eta(\phi_1+\phi_2)|]^\gamma\}},
\qquad t \to \infty.
\end{equation}
Hence the integral in \eqref{e2494} is determined by $\phi$ such that
\begin{equation}
B(\phi_1^2+\phi_2^2)+ A\big[|\tfrac12 C\eta(\phi_1+\phi_2)|\big]^\gamma \leq \frac{c}{t}.
\end{equation}
for $c$ a positive constant, and we find that $f(t) \asymp t^{-\left(\frac{1}{\gamma}+\frac{1}{2}\right)}$. Since $\gamma \in (1,2)$, $f(t)$ is much smaller than $\hat{a}_t(0,0) \asymp 1/t$, valid for two-dimensional simple random walk. Thus we see that $t \mapsto f(t)$ is integrable, while $t \mapsto \hat{a}_t(0,0)$ is not.

%%%

\subsection{Modulation of the law of the wake-up times by a slowly varying function}
\label{ss.proofslowvar}

The integral in \eqref{cluscritseed-b} is the \emph{total hazard of coalescence} of two dual lineages: 
\begin{itemize}
\item
If $\gamma \in (0,1)$, then the probability for each of the lineages to be active at time $s$ decays like $\asymp \varphi(s)^{-1} s^{-(1-\gamma)}$ \cite{AB16}. Hence the expected total time they are active up to time $s$ is $\asymp \varphi(s)^{-1} s^\gamma$. Because the lineages only move when they are active, the probability that the two lineages meet at time $s$ is $\asymp a^{(N)}_{\varphi(s)^{-1} s^\gamma}(0,0)$. Hence the total hazard is $\asymp \int_1^\infty \d s\, \varphi(s)^{-2}s^{-2(1-\gamma)}\,a^{(N)}_{\varphi(s)^{-1} s^\gamma}(0,0)$. After the transformation $t=t(s)=\varphi(s)^{-1} s^\gamma$, we get the integral in \eqref{cluscritseed-b}, modulo a constant. (When carrying out this transformation, we need that $\lim_{s\to\infty} s\varphi'(s)/\varphi(s)=0$, which is immediate from \eqref{hatphirepr}, and $\varphi(t(s))/\varphi(s) \asymp 1$ as $s\to\infty$, which is immediate from the bound we imposed on $\psi$ together with the fact that $\lim_{s\to\infty} \log \varphi(s)/\log s = 0$.)   
\item
If $\gamma = 1$, then the probability for each of the lineages to be active at time $s$ decays like $\hat\varphi(s)^{-1}$ \cite{AB16}. Hence the expected total time they are active up to time $s$ is $\asymp s \hat\varphi(s)^{-1}$. Hence the total hazard is $\asymp \int_1^\infty \d s\, \hat\varphi(s)^{-2} \,a^{(N)}_{\hat\varphi(s)^{-1}s}(0,0)$. After the transformation $t=t(s)=\hat\varphi(s)^{-1}s$, we get the integral in \eqref{cluscritseed-b}, modulo a constant.   
\end{itemize}

%%%%%%%%%% SECTION 7 %%%%%%%%%%%%%%%%%%%%%%%%%%%

\section{Proofs: Long-time behaviour for Model 3}
\label{s.model3}

The arguments for model 2 in Section~\ref{s.model2} all carry over with minor adaptations. The only difference is that for $\rho=\infty$ the clustering criterion changes. In this section we prove the new clustering criterion and comment on the modifications needed in the corresponding proofs for model 2 in Section~\ref{s.model2}.

%%%
\subsection{Moment relations}
Like in model 1 and 2, we can relate the first and second moments of the system in \eqref{gh1**}--\eqref{gh2**} to the
random walk that evolves according to the transition kernel $b^{(3)}(\cdot,\cdot)$ on $\G\times\{A,(D_m)_{m\in\N_0}\}$ given by \eqref{mrw3}. Replacing in Lemma~\ref{lem1cg9422} the kernel $b^{(2)}(\cdot,\cdot)$ by $b^{(3)}(\cdot,\cdot)$, we find the moment relation for model 3.  Also here these moment relations hold for all $g\in\CG$. Moreover these moment relations holds for $\rho<\infty$ as well as for $\rho=\infty$. 

\subsection{The clustering case }

To obtain the equivalent of Lemma \ref{lem:clusgisgfw2}, we need to replace the kernel $\hat{a}(\cdot,\cdot)$ by the convoluted kernel $(\hat{a}\ast \hat{a}^{\dagger})(\cdot,\cdot)$. Each time one of the two copies of the random walk with migration kernel $a(\cdot,\cdot)$ moves from the active state to the dormant state, it makes a transition according to the displacement kernel $a^\dagger(\cdot,\cdot)$ (recall \eqref{displa}). Therefore the expression in \eqref{same2} needs to be replaced by
\begin{equation}
\label{same2alt}
I = \int_0^\infty \d t\,\sum_{k,k'\in\N} \sum_{i,i' \in \G} \sum_{j \in \G}
\mathbb{E}_{(0,A)}\Big[\hat{a}_{T(k,t)}(0,i)\,\hat{a}_{T'(k',t)}(0,i')\,\hat{a}^\dagger_k(i,j)\,\hat{a}^\dagger_{k'}(i',j)\,
1_{\CE(k,t)}\,1_{\CE'(k',t)}\Big], 
\end{equation}   
where $\hat{a}^\dagger_k(\cdot,\cdot)$ is the step-$k$ transition kernel of the random walk with displacement kernel $\hat{a}^\dagger(\cdot,\cdot)$. Using the \emph{symmetry} of both kernels, we can carry out the sum over $j,i'$ and write 
 \begin{equation}
\label{e2550}
\begin{aligned}
I &= \int_0^\infty \d t\,\sum_{k,k'\in\N} \sum_{j \in \G}
\mathbb{E}_{(0,A)}\Big[\hat{a}_{T(k,t)+T'(k',t)}(0,j)\,\,\hat{a}^\dagger_{k+k'}(0,j)\,1_{\CE(k,t)}\,1_{\CE'(k',t)}\Big]\\
&= \int_0^\infty \d t\,\sum_{j \in \G}
\mathbb{E}_{(0,A)}\Big[\hat{a}_{T(t)+T'(t)}(0,j)\,\,\hat{a}^\dagger_{N(t)+N'(t)}(0,j)\,1_{\CE(t)}\,1_{\CE'(t)}\Big]\\
&= \int_0^\infty \d t\,
\mathbb{E}_{(0,A)}\Big[\big(\hat{a}_{T(t)+T'(t)} \ast \hat{a}^\dagger_{N(t)+N'(t)}\big)\,(0,0)\,1_{\CE(t)}\,1_{\CE'(t)}\Big]. 
\end{aligned}
\end{equation}   
The last expression is the analogue of \eqref{same3}. 

For $\rho<\infty$, following the same line of argument as for model 2, we find with the help of \eqref{qsym} that
\begin{equation}
\label{Iappr1altb}
I \asymp \int_1^\infty \d t\, (\hat{a}_{t} \ast \hat{a}^\dagger_{t})(0,0).
\end{equation}
For $\rho=\infty$, with the help of the Fourier transform we compute
\begin{equation}
\label{vraag}
\begin{aligned}
\mathbb{E}_{(0,A)}\big[\big(a_{T(t)+T'(t)}\ast a^\dagger_{N(t)+N'(t)}\big)\,(0,0)\,\big]
&=\E_{(0,A)}\left[\frac{1}{(2\pi)^d} \int_{(-\pi,\pi]^d}\d\phi\,\e^{-(T(t)+T^\prime(t))[1-\hat{a}(\phi)]}\,
\hat{a}^\dagger(\phi)^{N(t)+N^\prime(t)}\right]\\
&=\frac{1}{(2\pi)^d}\int_{(-\pi,\pi]^d} \d\phi\,\e^{-[1+o(1)]\,2ct^{-\gamma}\,[1-\hat{a}(\phi)]}\,
\e^{-[1+o(1)]\,2t^{-\gamma}[1-\hat{a}^\dagger(\phi)]}\\
&\asymp (\hat{a}_{ct^{-\gamma}}\ast\hat{a}^\dagger_{t^{-\gamma}})(0,0)
\asymp (\hat{a}_{t^{-\gamma}}\ast\hat{a}^\dagger_{t^{-\gamma}})(0,0),
\end{aligned}
\end{equation}
where we use \eqref{qsym}, \eqref{ETasympalt} and the fact that deviations of $T(t)/t^\gamma$ and $T'(t)/t^\gamma$ away from order 1 are stretched exponentially costly in $t$ \cite{EJU19}. Hence
\begin{equation}
\label{Iappr1alt}
I \asymp \int_1^\infty \d t\,t^{-2(1-\gamma)} (\hat{a}_{t^\gamma} \ast \hat{a}^\dagger_{t^\gamma})(0,0).
\end{equation}
Putting $s=t^\gamma$ we obtain, instead of \eqref{Iappr1},
\begin{equation}
\label{Iappr12alt}
I = \infty \quad \Longleftrightarrow \quad I_{\hat{a} \ast \hat{a}^\dagger,\gamma} = \infty
\end{equation}
with
\begin{equation}
\label{e2581}
I_{\hat{a} \ast \hat{a}^\dagger,\gamma} = \int_1^\infty \d s\,s^{-(1-\gamma)/\gamma}\,(\hat{a}_s \ast \hat{a}^\dagger_s)(0,0),
\end{equation} 
which is precisely the integral in \eqref{Idefalt}. 

\subsection{The coexistence case}

The coexistence results in Theorem~\ref{T.ltb3} follow for both $\rho<\infty$ and $\rho=\infty$ by the same type of argument as the one we used for model 2 in Section \ref{ss.cosalt}. We replace \eqref{gh1*}--\eqref{gh2*} by \eqref{gh1**}--\eqref{gh2**}, replace $b^{(2)}(\cdot,\cdot)$ (see \ref{mrw2}) by $b^{(3)}(\cdot,\cdot)$ (see \ref{mrw3}), and use the Fourier transform of $\hat{a}\ast \hat{a}^{\dagger}(\cdot,\cdot)$ instead of $\hat{a}(\cdot,\cdot)$. The key of the argument is that, in the coexistence case, for $\rho<\infty$ we have  $I_{\hat{a}\ast\hat{a}^\dagger}<\infty$, while for $\rho=\infty$ we have $I_{\hat{a}\ast\hat{a}^\dagger,\gamma}<\infty$.

\subsection{Proof of the dichotomy}

This follows in exactly the same way as for model 2.

%%%%%%%%%% APPENDICES %%%%%%%%%%%%%

\appendix

%%%%%%%%% APPENDIX A %%%%%%%%%%%%%%%

\section{Derivation of continuum frequency equations}
\label{appA}

\paragraph{Model 1.}
We give the derivation of \eqref{gh1}--\eqref{gh2} as the continuum limit of an individual-based model when the size of the colonies tends to infinity. We start with the continuum limit of the Fisher-Wright model with (strong) seed-bank for a \emph{single-colony} model as defined in \cite{BCKW16}. Subsequently we show how the limit extends to a \emph{multi-colony} model with seed-bank.
 
\paragraph{Single-colony model.}
The Fisher-Wright model with (strong) seed-bank defined in \cite{BCKW16} consists of a \emph{single colony} with $N\in\N$ active individuals and $M\in\N$ dormant individuals. Each individual can carry one of two types: $\heartsuit$ or $\diamondsuit$. Let $\epsilon \in [0,1]$ be such that $\epsilon N$ is integer and $\epsilon N \leq M$. Put $\delta=\frac{\epsilon N}{M}$. The evolution of the population is described by a discrete-time Markov chain that undergoes four transitions per step:
\begin{itemize}
\item[(1)] 
From the $N$ active individuals, $(1-\epsilon)N$ are selected uniformly at random without replacement. Each of these individuals resamples, i.e. it adopts the type of an active individual selected uniformly at random with replacement, and remains active. 
\item[(2)] 
Each of the $\epsilon N$ active individuals not selected first resamples, it adopts the type of an active individual selected uniformly at random with replacement, and subsequently becomes dormant.
\item[(3)] 
From the $M$ dormant individuals, $\delta M= \epsilon N$ are selected uniformly at random without replacement, and each of these becomes active. Since these individuals come from the dormant population they do not resample.  
\item[(4)] 
Each of $(1-\delta)M$ dormant individuals not selected remains dormant and retains its type.  
\end{itemize}
Note that the total sizes of the active and the dormant population remain fixed. During the evolution the dormant and active population \emph{exchange } individuals. We are interested in the fractions of individuals of type $\heartsuit$ in the active and the dormant population. 

%%%%%%%%%%%%%%%%%%%%%%%%%%%%%%%%%%%%%%%%%%%%%%
\begin{figure}[htbp]
\begin{center}
\begin{tikzpicture}[xscale=0.6, yscale=0.8]
\draw [fill=black!20!black!20!] (6.5,-0.5) rectangle (9.5,4.5);
\foreach \x in {1,3,4}
{
\node [] at (\x,0) {\textcolor{black}{\Huge $\heartsuit$}} ;
}
\foreach \x in {2,5}
{
\node [] at (\x,0) {\Huge $\diamondsuit$} ;
}
\foreach \x in {3}
{
\node [] at (\x+6,0) {\textcolor{black}{\Huge $\heartsuit$}} ;
}
\foreach \x in {1,2}
{
\node [] at (\x+6,0) {\Huge $\diamondsuit$} ;
}
		
%%%%%%%%%% Generation 2
		
\foreach \x in {4}
{
\node [] at (\x,1) {\textcolor{black}{\Huge $\heartsuit$}} ;
}
\foreach \x in {1,2,3,5}
{
\node [] at (\x,1) {\Huge $\diamondsuit$} ;
}
\foreach \x in {1,3}
{
\node [] at (\x+6,1) {\textcolor{black}{\Huge $\heartsuit$}} ;
}
\foreach \x in {2}
{
\node [] at (\x+6,1) {\Huge $\diamondsuit$} ;
}
%%%%%%%%%%%%% Generation 3
\foreach \x in {}
{
\node [] at (\x,2) {\textcolor{black}{\Huge $\heartsuit$}} ;
}
\foreach \x in {1,2,3,4,5}
{
\node [] at (\x,2) {\Huge $\diamondsuit$} ;
}
\foreach \x in {1,2,3}
{
\node [] at (\x+6,2) {\textcolor{black}{\Huge $\heartsuit$}} ;
}
\foreach \x in {}
{
\node [] at (\x+6,2) {\Huge $\diamondsuit$} ;
}
%%%%%%% Generation 4
\foreach \x in {4,5}
{
\node [] at (\x,3) {\textcolor{black}{\Huge $\heartsuit$}} ;
}
\foreach \x in {1,2,3}
{
\node [] at (\x,3) {\Huge $\diamondsuit$} ;
}
\foreach \x in {1,3}
{
\node [] at (\x+6,3) {\textcolor{black}{\Huge $\heartsuit$}} ;
}
\foreach \x in {2}
{
\node [] at (\x+6,3) {\Huge $\diamondsuit$} ;
}
%%%%%%%%% Generatie 5
\foreach \x in {2,4}
{
\node [] at (\x,4) {\textcolor{black}{\Huge $\heartsuit$}} ;
}
\foreach \x in {1,3,5}
{
\node [] at (\x,4) {\Huge $\diamondsuit$} ;
}
\foreach \x in {1}
{
\node [] at (\x+6,4) {\textcolor{black}{\Huge $\heartsuit$}} ;
}
\foreach \x in {2,3}
{
\node [] at (\x+6,4) {\Huge $\diamondsuit$} ;
}
\draw [ultra thick, black] (1,4-0.25)to [out=90,in=270](1, 3+0.25) ;
\draw [ultra thick, black] (1,4-0.25)to [out=270,in=90](3, 3+0.25) ;
\draw [ultra thick, black] (9,4-0.25) to [out=200,in=30] (2, 3+0.25) ;
\draw [ultra thick, black] (2,4+0.25)to [out=20, in=160](9, 4.5+0) ;
\draw [ultra thick, black] (1.9,4+0.3)to [out=170, in=180](2, 5) ;
\draw [ultra thick, black] (2,5)to [out=0, in=0](2.1, 4.3) ;
\draw [ultra thick, black] (9,4.5)to [out=340, in=0](9.25, 3+0) ;
\draw [ultra thick, black] (4,4-0.25)to [out=270,in=90](4, 3+0.25) ;
\draw [ultra thick, black] (4,4-0.25)to [out=270,in=90](5, 3+0.25) ;
\foreach \x in {0,...,4}
{
\node [] at (-0.5,4-\x) { $\x$} ;
}
\node[]at (-0.5,5) {t};
\node[align=left, below] at (8,-.5)
{Dormant\\population};
\node[align=left, below] at (2,-.5)
{Active\\population};
\end{tikzpicture}	
\caption{\small Example of the evolution for a population with $N=5$ active individuals and $M=3$ dormant individuals. The solid lines within the active population represent resampling, those between the active and the dormant population represent exchange with the seed-bank.  Only 1 active individual and 1 dormant individual exchange places per unit of time, which corresponds to $\epsilon=\tfrac15$ and $\delta=\tfrac13$. The relative size of the dormant and the active population is $K=\tfrac35$. Note that the genetic diversity in the active population is lost in generation $t=2$, but returns in generation $t=3$ via the seed-bank.}
\label{fig:seedbankcartoon}		
\end{center}
\end{figure}
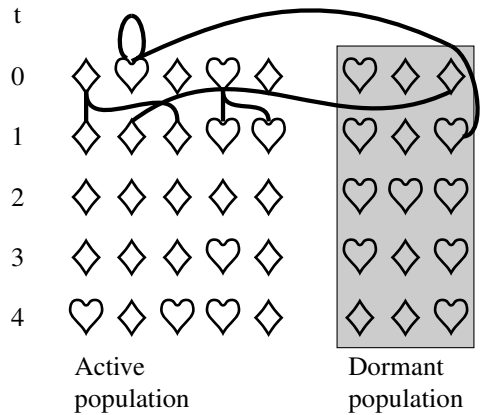
%%%%%%%%%%%%%%%%%%%%%%%%%%%%%%%%%%%%%%%%%%%%%%%%%%%%

Let $c=\epsilon N=\delta M$, i.e., $c$ is the number of pairs of individuals that change state. Label the $N$ active individuals from $1$ to $N$ and the $M$ dormant individuals from $1$ up to $M$. We denote by $[N] = \{1,\ldots,N\}$ and by $[M] = \{1,\ldots, M\}.$  Let $\xi(k)=(\xi_j(k))_{j\in[N]}\in\{0,1\}^{[N]}$ be the random vector where $\xi_j(k)=1$ if the $j$'th individual is of type $\heartsuit$ at time $k$ and $\xi_j(k)=0$ if the $j$'th individual is of type $\diamondsuit$ at time $k$. Similarly, we let $\eta(k)=(\eta_j(k))_{j\in[M]}\in\{0,1\}^{[M]}$ be the random vector where $\eta_j(k)=1$ if the $j$'th individual is of type $\heartsuit$ at time $k$ and $\eta_j(k)=0$ if the $j$'th individual is of type $\diamondsuit$ at time $k$. Let $I^N=\{0,\frac{1}{N},\frac{2}{N},\frac{3}{N}\ldots,1\}$ and $I^M=\{0,\frac{1}{M},\frac{2}{M},\frac{3}{M}\ldots,1\}$. Define the variables 
\begin{equation}
\label{e2975}
\begin{aligned}
X^N(k)=\frac{1}{N}\sum_{j\in [N]}\mathbf{1}_{\{\xi_j(k)=\heartsuit\}} \quad \text{ on } I^N,\\
Y^N(k)=\frac{1}{N}\sum_{j\in [N]}\mathbf{1}_{\{\eta_j(k)=\heartsuit\}} \quad \text {on } I^M.
\end{aligned}
\end{equation} 
Let $\mathbb{P}_{x,y}$ denote the law of 
\begin{equation}\label{e2982}
(X^N,Y^N) = (X^N(k),Y^N(k))_{k \in \N_0}
\end{equation}
given that $(X^N(0),Y^N(0)) = (x,y) \in I^N\times I^M$. Then, as shown in \cite{BCKW16},
\begin{equation}
\label{e2986}
\begin{aligned}
p_{x,y}(\bar{x},\bar{y}) 
&=\mathbb{P}_{x,y}(X_1^N=\bar{x},Y_1^N=\bar{y})\\
&=\sum_{c'=0}^{c} \mathbb{P}_{x,y}(Z=c')\,\mathbb{P}_{x,y}(U=\bar{x}N-c')\,
\mathbb{P}_{x,y}(V=\bar{y}M-yM+c'). 
\end{aligned}
\end{equation}
Here, $Z$ denotes the number of dormant $\heartsuit$-individuals in generation $0$ that become active in generation $1$ ($\mathcal{L}_{x,y}(Z)=\text{Hyp}_{M,c,yM}$), $U$ denotes the number of active individuals in generation $1$ that are offspring of active $\heartsuit$-individuals in generation $0$ ($\mathcal{L}_{x,y}(U)=\text{Bin}_{N-c,x}$), and $V$ denotes the number of active individuals in generation $0$ that become dormant $\heartsuit$-individuals in generation $1$ ($\mathcal{L}_{x,y}(V)=\text{Bin}_{c,x}$). 

Speed up time by a factor $N$. The generator $G^N$ for the process $((X^{N}(\lfloor Nk \rfloor),Y^{N}(\lfloor Nk \rfloor))_{k\in\N_0}$ equals
\begin{equation}
\label{e3001}
\begin{aligned}
&(G^Nf)(x,y) = N\, \mathbb{E}_{x,y}\big[f(X^N(1),Y^N(1))-f(x,y)\big],\\ 
&(x,y) \in I^N \times I^M,
\end{aligned}
\end{equation} 
where the prefactor $N$ appears because one step of the Markov chain takes time $\frac{1}{N}$. Inserting the Taylor expansion for $f$ (which we assume to be smooth), using that $X^N(1)=\frac{U+Z}{N}$ and $Y^{N}(1)=\frac{yM+V-U}{M}$ and letting $N\to\infty$, we end up with the limiting generator $G$ given by
\begin{equation}
\label{e3010}
\begin{aligned}
&(Gf)(x,y) = c(y-x)\frac{\partial f}{\partial x}(x,y)+\frac{c}{K}(x-y) \frac{\partial f}{\partial y}(x,y)
+\tfrac{1}{2}x(1-x)\frac{\partial ^2 f}{\partial x^2}(x,y),\\ 
&(x,y) \in [0,1] \times [0,1],
\end{aligned}
\end{equation} 
where $K=\frac{M}{N}$ is the relative size of the dormant population compared to the active population. This is the generator of the Markov process in the continuum limit \cite[Section 7.8]{EK86}. It follows from the form of $G$ that this limit is described by the system of coupled stochastic differential equations 
\begin{equation}
\label{seedbankwithoutmigration}
\begin{aligned}
\d x(t)  &= c\,[y(t)-x(t)]\,\d t+\sqrt{x(t)(1-x(t))}\,\d w(t),\\
\d y(t) &= \frac{c}{K}\,[x(t)-y(t)]\,\d t.
\end{aligned}
\end{equation}
This is the version of \eqref{gh1}--\eqref{gh2} for a single colony (no migration) and exchange rate
\begin{equation}
\label{e3029}
e = \frac{c}{K}.
\end{equation} 

\paragraph{Multi-colony model.}
First fix a number $L\in\N$ and consider $|\G|=L$ colonies. The \emph{multi-colony} version with migration is obtained by allowing the $(1-\epsilon)N$ selected active individuals to undergo a migration in step (1):
\begin{itemize}
\item[(1)]
Each active individual at colony $i\in\G$ chooses colony $j\in\G$ with probability $\frac{1}{N} a(i,j)$ and adopts the type of a parent chosen from colony $j$. If an active individual does not migrate, it adopts the type of a parent chosen from its own population.    
\end{itemize}
Using the same strategy as in the single-colony model, this  results in \eqref{gh1}--\eqref{gh2}, for $|\G|=L$. Subsequently we can let $L\to\infty$ and use convergence of generators to obtain \eqref{gh1}--\eqref{gh2} for countable $\G$.

\paragraph{Model 2.}

The same argument works for \eqref{gh1*}--\eqref{gh2*}. Steps (1)-(4) are extended by considering a seed-bank with colours labelled by $\N_0$. First we consider the truncation where only finitely many colours are allowed, for which the argument carries through with minor adaptations. Afterwards, we pass to the limit of infinitely many colours, which is straightforward for a finite time horizon because large colours are only seen after large times. See also \cite{M18}.

\paragraph{Model 3.}

To get \eqref{gh1**}--\eqref{gh2**}, also extend Step (3) by adding a displacement via the kernel $a^\dagger(\cdot,\cdot)$ for each transition into the seed-bank.

%%%%%%%%% APPENDIX B %%%%%%%%%%%%%%%

\section{Alternative models}
\label{appB}

In this appendix we consider the Moran versions of models 1 and 2. What is written below is based on \cite{M18}. In the Moran version each active individual resamples at rate 1 and becomes dormant at a certain rate, while each dormant individual does not resample and becomes active at a certain rate. Since switches between active and dormant are done independently, the sizes of the active and the dormant population are \emph{no longer} fixed and individuals \emph{change} state without the necessity to \emph{exchange} state. In model 1 there are two Poisson clocks, in model 2 there are two sequences of Poisson clocks, namely, two for each colour. In Appendices~\ref{app1c}--\ref{app2c} we compute the scaling limit for the case where the number of colours is $\mathfrak{m}=1$ and $\mathfrak{m}=2$, respectively. The extension to $\mathfrak{m} \geq 3$ is given in Appendix~\ref{app3c}. Migration can be added in the same way as is done in Appendix \ref{appA}.  

%%%

\subsection{Alternative for Model 1}
\label{app1c}

To describe the Moran version of Model 1 we need the following variables.
\begin{itemize}
\item 
Total number of individuals: $N\in\mathbb{N}$.
\item 
Two types: $\heartsuit$ and $\diamondsuit$.
\item 
$X(t)$ is the number of $\heartsuit$-individuals in the active population at time $t$.
\item 
$Y(t)$ is the number of $\heartsuit$-individuals in the dormant population at time $t$.
\item 
$Z(t)$ is the number of individuals in the active population at time $t$ (either $\heartsuit$ 
or $\diamondsuit$).
\end{itemize}
In the Moran model with seed-bank each active individual resamples at rate $1$, each active individual becomes dormant at rate $\epsilon$ and each dormant individual becomes active at rate $\delta$. Hence the transition rates for $(X(t),Y(t),Z(t))$ are:
\begin{itemize}
\item 
$(i,j,k)\rightarrow(i+1,j,k)$ at rate $(k-i)\frac{i}{k}$.
\item 
$(i,j,k)\rightarrow(i-1,j,k)$ at rate $i\frac{(k-i)}{k}$.
\item 
$(i,j,k)\rightarrow(i-1,j+1,k-1)$ at rate $\epsilon i $.
\item 
$(i,j,k)\rightarrow(i+1,j-1,k+1)$ at rate $\delta j $.
\item $(i,j,k)\rightarrow(i,j,k-1)$ at rate $  \epsilon \frac{k-i}{N}$.
\item $(i,j,k)\rightarrow(i,j,k+1)$ at rate $\delta \frac{N-k-j}{N}$.
\end{itemize}
For the scaling limit we consider the variables
\begin{equation}
\label{B1}
\bar X(t)=\frac{1}{N}X(N t),\quad \bar Y(t)=\frac{1}{N}Y(N t), \quad \bar Z(t)=\frac{1}{N}Z(N t).
\end{equation}
Hence
\begin{equation}
\label{e3100}
(\bar X(t), \bar Y(t), \bar Z(t)) \in I^N \times I^N \times I^N, \qquad
I^N = \left\{0,\tfrac{1}{N},\tfrac{2}{N},\ldots,\tfrac{N-1}{N},1\right\}.
\end{equation}
Since in \eqref{B1} we speed up time by a factor $N$, we must also speed up the transition rates by a factor $N$. To get a meaningful scaling limit, we assume that there exist $c^A, c^D \in (0,\infty)$ such that (see \cite[p.~8]{BCKW16})
\begin{equation}
\label{e3107}
N\epsilon=c^A, \qquad N\delta=c^D, \qquad N\in\mathbb{N}.
\end{equation} 
We can then write down the generator $G^N$:
\begin{equation}
\begin{aligned}
\label{e3111}
(G^Nf)\left(\frac{i}{N},\frac{j}{N},\frac{k}{N}\right)&=N(k-i)\frac{i}{k}\left[f\left(\frac{i+1}{N},\frac{j}{N},\frac{k}{N}\right)-f\left(\frac{i}{N},\frac{j}{N},\frac{k}{N}\right)\right]\\
&\quad+N i\frac{k-i}{k}\left[f\left(\frac{i-1}{N},\frac{j}{N},\frac{k}{N}\right)-f\left(\frac{i}{N},\frac{j}{N},\frac{k}{N}\right)\right]\\
&\quad+c^A i \left[f\left(\frac{i-1}{N},\frac{j+1}{N},\frac{k-1}{N}\right)-f\left(\frac{i}{N},\frac{j}{N},\frac{k}{N}\right)\right]\\
&\quad+c^D j \left[f\left(\frac{i+1}{N},\frac{j-1}{N},\frac{k+1}{N}\right)-f\left(\frac{i}{N},\frac{j}{N},\frac{k}{N}\right)\right]\\
&\quad+c^A (k-i) \left[f\left(\frac{i}{N},\frac{j}{N},\frac{k-1}{N}\right)-f\left(\frac{i}{N},\frac{j}{N},\frac{k}{N}\right)\right]\\
&\quad+c^D (N-k-j) \left[f\left(\frac{i}{N},\frac{j}{N},\frac{k+1}{N}\right)-f\left(\frac{i}{N},\frac{j}{N},\frac{k}{N}\right)\right]\\
\end{aligned}
\end{equation}
Assuming that $f$ is smooth and Taylor expanding $f$ around $\left(\frac{i}{N},\frac{j}{N},\frac{k}{N}\right)$, we get
\begin{equation}
\begin{aligned}
\label{MoranAN}
(G^Nf)\left(\frac{i}{N},\frac{j}{N},\frac{k}{N}\right)&=\frac{i(k-i)}{k}
\left[\left(\frac{1}{N}\right)\frac{\partial^2 f}{\partial x ^2}+\mathcal{O}\left(\left(\frac{1}{N}\right)^2\right)\right]\\
&\quad+c^A i \left[\left(\frac{-1}{N}\right)\frac{\partial f}{\partial x }
+\left(\frac{1}{N}\right)\frac{\partial f}{\partial y }+\left(\frac{-1}{N}\right)\frac{\partial f}{\partial z }
+\mathcal{O}\left(\left(\frac{1}{N}\right)^2\right)\right]\\
&\quad+c^D j \left[\left(\frac{1}{N}\right)\frac{\partial f}{\partial x }
+\left(\frac{-1}{N}\right)\frac{\partial f}{\partial y }+\left(\frac{1}{N}\right)\frac{\partial f}{\partial z }
+\mathcal{O}\left(\left(\frac{1}{N}\right)^2\right)\right]\\
&\quad+c^A (k-i) \left[\left(\frac{-1}{N}\right)\frac{\partial f}{\partial z }
+\mathcal{O}\left(\left(\frac{1}{N}\right)^2\right)\right]\\
&\quad+c^D (N-k-j) \left[\left(\frac{1}{N}\right)\frac{\partial f}{\partial z }
+\mathcal{O}\left(\left(\frac{1}{N}\right)^2\right)\right].
\end{aligned}
\end{equation}

Next, suppose that 
\begin{equation}
\label{e3157}
\lim_{N\rightarrow\infty}\frac{i}{N}=x, \qquad \lim_{N\rightarrow\infty}\frac{j}{N}=y,
\qquad\lim_{N\rightarrow\infty}\frac{k}{N}=z.
\end{equation}
Letting $N\rightarrow\infty$ in \eqref{MoranAN}, we obtain the limiting generator $G$:
\begin{equation}
\label{MoranA}
\begin{aligned}
(Gf)(x,y,z)=\  z \frac{x}{z}\left(1-\frac{x}{z}\right)\left(\frac{\partial^2 f}{\partial x^2}\right)
+[c^D\,y-c^A\,x]\frac{\partial f}{\partial x}
+[c^A\,x-c^D\,y]\frac{\partial f}{\partial y}
+\big[c^D\,(1-z)-c^A\,z\big]\frac{\partial f}{\partial z}.
\end{aligned}
\end{equation}
Therefore the continuum limit equals
\begin{equation}
\label{Moranone}
\begin{aligned}
\d x(t) &=\sqrt{z(t)\,\frac{x(t)}{z(t)}\left(1-\frac{x(t)}{z(t)}\right)}\,\,\d w(t)+\big[c^D\,y(t)-c^A\,x(t)\big]\,\d t,\\
\d y(t) &= \big[c^A\,x(t)-c^D\,y(t)\big]\,\d t,\\
\d z(t) &= \big[c^D\,(1-z(t))-c^A\,z(t)\big]\,\d t.
\end{aligned}
\end{equation}
Since $z(t)$ is the fraction of active individuals in the population, $1-z(t)$ is the fraction of dormant individuals in the population. Therefore the equivalent of the parameter $K$ in Appendix \ref{appA} is $K(t)=(1-z(t))/z(t)$. Moreover, $x(t)/z(t)$ is the fraction of $\heartsuit$-individuals in the active population at time $t$ and $y(t)/(1-z(t))$ is the fraction of $\heartsuit$-individuals in the dormant population at time $t$. The last line of \eqref{Moranone} is an autonomous differential equation whose solution converges to
\begin{equation}
z^*=\frac{1}{1+\frac{c^A}{c^D}}
\end{equation}
exponentially fast. After this transition period we can replace $z(t)$ by $z^*$, and we see that $K^*=c^A/c^D$. 

Time is to be scaled by the total number of active \emph{and} dormant individuals, instead of the total number of active individuals only:
\begin{equation}
x(t)=\frac{ \text{number of  active individuals of type }\heartsuit}{\text{ total number of individuals}},
\quad
y(t)=\frac{ \text{number of dormant individuals of type }\heartsuit}{\text{ total number of individuals}}.
\end{equation}
To compare the Moran model with a 1-colour seed-bank with the Fisher-Wright model with a 1-colour seed-bank, we look at the variables
\begin{equation}
\bar{x}(t)=\left(1+\frac{c^A}{c^D}\right)x\left(\frac{t}{1+\frac{c^A}{c^D}}\right),
\qquad 
\bar{y}(t)=\left(1+\frac{c^A}{c^D}\right)\left(\frac{c^D}{c^A}\right)y\left(\frac{t}{1+\frac{c^A}{c^D}}\right). 
\end{equation}
After a short transition period in which $z(t)$ tends to $z^*$, we see that by setting 
\begin{equation}
K=K^*=\frac{c^A}{c^D}, \qquad e=\frac{c^D}{c^A}\frac{c^Ac^D}{c^A+c^D},
\end{equation} 
we obtain 
\begin{equation}
\label{Moranonealt}
\begin{aligned}
\d \bar{x}(t)  &= \sqrt{\bar{x}(t)(1-\bar{x}(t))}\,\d w(t)+Ke\,[\bar{y}(t)-\bar{x}(t)]\,\d t,\\
\d \bar{y}(t) &= e\,[\bar{x}(t)-\bar{y}(t)]\,\d t,
\end{aligned}
\end{equation}
which is the single-colony version of \eqref{gh1}--\eqref{gh2} but without migration. Migration can be added in the same way as was done in Appendix~\ref{appA}.

%%%

\subsection{Alternative for Model 2: Two colours}
\label{app2c}

We consider the following system:
\begin{itemize}
\item 
Total number of individuals: $N\in\mathbb{N}$.
\item 
Two types: $\heartsuit$ and $\diamondsuit$.
\item 
$X(t)$ is the number of $\heartsuit$-individuals in the active population at time $t$.
\item 
$Y_1(t)$ is the number of $\heartsuit$-individuals of colour 1 in the dormant population 
at time $t$.
\item 
$Y_2(t)$ is the number of $\heartsuit$-individuals of colour 2 in the dormant population 
at time $t$.
\item 
$Z_{D_1}(t)$ is the number of dormant individuals of colour 1 at time $t$ (either $\heartsuit$ or $\diamondsuit$).
\item 
$Z_{D_2}(t)$ is the number of dormant individuals of colour 2 at time $t$.
(either $\heartsuit$ or $\diamondsuit$).
\end{itemize}
Note that the number of active individuals at time $t$ (either $\heartsuit$ or $\diamondsuit$) is given by $Z_A(t)=N-Z_{D_1}(t)-Z_{D_2}(t)$. Since the number of individuals $N$ is constant during the evolution, $Z_A(t)$ can be derived from $Z_{D_1}(t)$ and $Z_{D_2}(t)$. Each active individual resamples at rate 1, and becomes dormant at rate $\epsilon$. When an individual becomes dormant, it gets either colour $1$ with probability $p_1$ or colour $2$ with probability $p_2$, where $p_1,p_2 \in (0,1)$ and $p_1+p_2=1$. For ease of notation, we denote the rate to become dormant with colour $1$ by $\epsilon_1=\epsilon\cdot p_1$ and the rate to become dormant with colour $2$ by $\epsilon_2=\epsilon\cdot p_2$. A dormant individual with colour $1$ becomes active at rate $\delta_1$, a dormant individual with colour $2$ becomes active at rate $\delta_2$. Thus, the transition rates for $(X(t),Y_1(t),Y_2(t),Z_{D_1}(t),Z_{D_2}(t))$ are:
\begin{itemize}
\item 
$(i,j,k,l,m)\rightarrow(i+1,j,k,l,m)$ at rate $(N-l-m-i)\frac{i}{N-l-m}$.
\item 
$(i,j,k,l,m)\rightarrow(i-1,j,k,l,m)$ at rate $i\frac{(N-l-m-i)}{N-l-m}$.
\item 
$(i,j,k,l,m)\rightarrow(i-1,j+1,k,l+1,m)$ at rate $\epsilon_1 i$.
\item 
$(i,j,k,l,m)\rightarrow(i+1,j-1,k,l-1,m)$ at rate $\delta_1 j$.
\item 
$(i,j,k,l,m)\rightarrow(i-1,j,k+1,l,m+1)$ at rate $ \epsilon_2 i$.
\item 
$(i,j,k,l,m)\rightarrow(i+1,j,k-1,l,m-1)$ at rate $\delta_2 k $.
\item 
$(i,j,k,l,m)\rightarrow(i,j,k,l+1,m)$ at rate $\epsilon_1 (N-l-m-i)$.
\item 
$(i,j,k,l,m)\rightarrow(i,j,k,l,m+1)$ at rate $\epsilon_2 (N-l-m-i)$.
\item 
$(i,j,k,l,m)\rightarrow(i,j,k,l-1,m)$ at rate $\delta_1 (l-j)$.
\item 
$(i,j,k,l,m)\rightarrow(i,j,k,l,m-1)$ at rate $\delta_2 (m-k)$.
\end{itemize}
Proceeding in the same way as for the 1-colour seed-bank, we define the scaled variables
\begin{equation}
\label{B1m}
\begin{aligned}
\bar X(t)&=\frac{1}{N}X(N t),\quad \bar Y_1(t)=\frac{1}{N}Y_1(N t),\quad \bar Y_2(t)=\frac{1}{N}Y_2(N t), \\
\bar Z_{D_1}(t)&=\frac{1}{N}Z_{D_1}(N t),\quad \bar Z_{D_2}(t)=\frac{1}{N}Z_{D_1}(N t).
\end{aligned}
\end{equation}
We assume that there exist $c^A_1,c^A_2, c^D_1, c^D_2 \in (0,\infty)$ such that
\begin{equation}
\label{e3252}
N\epsilon_1=c^A_1,  \qquad N\epsilon_2=c^A_2,  \qquad N\delta_1=c^D_1, \qquad N\delta_2=c^D_2, \qquad N\in\mathbb{N},
\end{equation} 
and further assume that
\begin{equation}
\begin{aligned}
&\lim_{N \to \infty}\frac{i}{N}=x, \quad \lim_{N \to \infty}\frac{j}{N}=y_1,\quad \lim_{N \to \infty}\frac{k}{N}=y_2,\\
&\lim_{N \to \infty}\frac{N-l-m}{N}=z_A\quad\lim_{N \to \infty}\frac{N-l-m}{N}=z_{D_2},
\quad \lim_{N \to \infty}\frac{N-l-m}{N}=z_{D_1}.
\end{aligned}
\end{equation}
Using the same method of converging generators as for model 1, we obtain the following continuum limit:
\begin{equation}
\label{mvan}
\begin{aligned}
\d x(t)&=\sqrt{z_A(t)\frac{z_{A}-x(t)}{z_A(t)}\frac{x(t)}{z_A(t)}}\,\d w(t)\\
&\qquad+ \left[c_1^Dy_1(t)-c_1^Ax(t)\right]\d t+\left[c_2^D y_2(t)-c_2^A x(t)\right]\d t,\\
\d y_1(t)&=\left[c_1^A x(t)-c_1^Dy_1(t)\right]\d t,\\
\d y_2(t)&=\left[c_2^A x(t)-c_2^D y_2(t)\right]\d t,\\
\d z_{A}(t)&=\left[c_1^Dz_{D_1}(t)-c_1^A z_A(t)+c_2^D z_{D_2}(t)-c_2^Az_A(t)\right]\d t,\\
\d z_{D_1}(t)&=\left[c_1^Az_{A}(t)-c_1^D z_{D_1}(t)\right]\d t,\\
\d z_{D_2}(t)&= \left[c_2^Az_{A}(t)-c_2^D z_{D_2}(t)\right]\d t.                  
\end{aligned}
\end{equation}
Note that the equation for $z_A(t)=1-z_{D_1}(t)-z_{D_2}(t)$ follows directly from the equations from $z_{D_1}(t)$ and $z_{D_2}(t)$. It is therefore redundant, but we use it for notational reasons. Again, we see that $z(t)=(z_A(t),z_{D_1}(t), z_{D_2}(t))$ is governed by an autonomous system of differential equations. Solving this system, we see that 
\begin{equation}
\label{zlim}
\begin{aligned}
\lim_{t \to \infty}z_A(t)=\frac{1}{1+\frac{c_1^A}{c_1^D}+\frac{c_2^A}{c_2^D}},\quad
\lim_{t\to\infty}z_{D_1}(t)=\frac{\frac{c_1^A}{c_1^D}}{1+\frac{c_1^A}{c_1^D}+\frac{c_2^A}{c_2^D}},\quad
\lim_{t \to \infty}z_{D_2}(t))=\frac{\frac{c_2^A}{c_2^D}}{1+\frac{c_1^A}{c_1^D}+\frac{c_2^A}{c_2^D}}.
\end{aligned}
\end{equation}

To compare the Moran model with a 2-colour seed-bank with the Fisher-Wright model with a 2-colour seed-bank, we look at the variables
\begin{equation}
\label{B16alt}
\begin{aligned}
\bar{x}(t)
&=\left(1+\frac{c_1^A}{c_1^D}+\frac{c_2^A}{c_2^D}\right)x\left(\frac{t}{1+\frac{c_1^A}{c_1^D}
+\frac{c_2^A}{c_2^D}}\right),\\ \bar{y}_1(t)&=\left(1+\frac{c_1^A}{c_1^D}
+\frac{c_2^A}{c_2^D}\right)\left(\frac{c^D_1}{c^A_1}\right)y_1\left(\frac{t}{1+\frac{c_1^A}{c_1^D}
+\frac{c_2^A}{c_2^D}}\right),\\
\bar{y}_2(t)&=\left(1+\frac{c_1^A}{c_1^D}+\frac{c_2^A}{c_2^D}\right)
\left(\frac{c^D_2}{c^A_2}\right)y_2\left(\frac{t}{1+\frac{c_1^A}{c_1^D}+\frac{c_2^A}{c_2^D}}\right). 
\end{aligned}
\end{equation}
Defining
\begin{equation}
K_m=\frac{c_m^A}{c_m^D}, \quad e_m=\frac{c_m^D}{1+\frac{c_1^A}{c_1^{D}}+\frac{c_2^A}{c_2^{D}}},
\qquad m\in\{1,2\},
\end{equation} 
we see that, after a short transition period, the system becomes
\begin{equation}
\label{Moranonealt2}
\begin{aligned}
\d \bar{x}(t)  
&= \sqrt{\bar{x}(t)(1-\bar{x}(t))}\,\d w(t)+K_1e_1\,[\bar{y}_2(t)-\bar{x}(t)]\,\d t
+K_2e_2\,[\bar{y}_1(t)-\bar{x}(t)]\,\d t,\\
\d \bar{y}_1(t) &= e_1\,[\bar{x}(t)-\bar{y}_1(t)]\,\d t,\\
\d \bar{y}_2(t) &= e_2\,[\bar{x}(t)-\bar{y}_2(t)]\,\d t,
\end{aligned}
\end{equation}
which is the single-colony version of \eqref{gh1*}--\eqref{gh2*} with 2 colours and without migration. Note, in particular, that after $z(t)$ reaches the equilibrium point in \eqref{zlim}, we have 
\begin{equation}
K_m=\frac{\text{number of dormant individuals with colour }m}{\text{ number of active individuals}}, \quad m \in \{1,2\}.
\end{equation}

\medskip
It is instructive to show how the above result can also be derived with the help of \emph{duality}. The argument that follows easily extends to an $n$-coloured seed-bank for any $n\in\N$ finite, to be considered in Appendix~\ref{app3c}. Recall from \eqref{mvan} that
\begin{equation}
\label{zev}
\begin{aligned}
\d z_{A}(t)&=\left[c_1^Dz_{D_1}(t)-c_1^A z_A(t)+c_2^D z_{D_2}(t)-c_2^Az_A(t)\right]\d t,\\
\d z_{D_1}(t)&=\left[c_1^Az_{A}(t)-c_1^D z_{D_1}(t)\right]\d t,\\
\d z_{D_2}(t)&= \left[c_2^Az_{A}(t)-c_2^D z_{D_2}(t)\right]\d t.                  
\end{aligned}
\end{equation}
Let 
\begin{equation}
\label{defbar}
\begin{aligned}
\bar{z}_A(t)&=\left(1+\frac{c_1^A}{c_1^D}+\frac{c_2^A}{c_2^D}\right)z_A(t),\\  \bar{z}_{D_1}(t)&=\left(1+\frac{c_1^A}{c_1^D}+\frac{c_2^A}{c_2^D}\right)\left(\frac{c_1^D}{c_1^A}\right)z_{D_1}(t),\\ 
\bar{z}_{D_2}(t)&=\left(1+\frac{c_1^A}{c_1^D}+\frac{c_2^A}{c_2^D}\right)\left(\frac{c_2^D}{c_2^A}\right)z_{D_2}(t).
\end{aligned}
\end{equation}
Substitute \eqref{defbar} into \eqref{zev}, to obtain
\begin{equation}
\begin{aligned}
\d \bar{z}_{A}(t)&=c_1^A\left[\bar{z}_{D_1}(t)- \bar{z}_A(t)\right]+c_2^A \left[\bar{z}_{D_2}(t)-\bar{z}_A(t)\right]\d t,\\
\d \bar{z}_{D_1}(t)&=c_1^D\left[\bar{z}_{A}(t)- \bar{z}_{D_1}(t)\right]\d t,\\
\d \bar{z}_{D_2}(t)&=c_2^D \left[\bar{z}_{A}(t)- \bar{z}_{D_2}(t)\right]\d t.                  
\end{aligned}
\end{equation}
To define a dual for the process $(\bar{z}_{A}(t),\bar{z}_{D_1}(t),\bar{z}_{D_2}(t)))_{t \geq 0}$, let $(M(t))_{t \geq 0}$ be the continuous-time Markov chain on $\{A,D_1, D_2\}$ with transition rates
\begin{equation}
\begin{aligned}
&A \to D_m \text{ at rate } c_m^A, \quad m\in\{1,2\},\\
&D_m \to A \text{ at rate } c_m^D, \quad m\in\{1,2\}.
\end{aligned}
\end{equation}
Consider $l$ independent copies of $(M(t))_{t \geq 0}$, evolving on the same state space $\{A,D_1,D_2\}$. Let $(L(t))_{t\geq 0}=(L_A(t),L_{D_1}(t),L_{D_2}(t))_{t\geq 0}$ be the process that counts how many copies of $M(t)$ are on site $\{A\}$, $\{D_1\}$ and $\{D_2\}$ at time $t$. Let $l=m+n_1+n_2$. Then $(L(t))_{t\geq 0}$ is the Markov process on $\N_0^3$ with transition rates
\begin{equation}
(m,n_1,n_2) \to \begin{cases}
(m-1,n_1+1,n_2) \qquad \text{ at rate } mc_1^A,\\
(m-1,n_1,n_2+1) \qquad \text{ at rate } mc_2^A,\\
(m+1,n_1-1,n_2) \qquad \text{ at rate } n_1c_1^D,\\
(m+1,n_1,n_2-1) \qquad \text{ at rate } n_2c_2^D.
\end{cases}.
\end{equation} 
Note that $L_A(t)+L_{D_1}(t)+L_{D_2}(t)=L_A(0)+L_{D_1}(0)+L_{D_2}(0)=m+n_1+n_2=l$. Define $H\colon\, \R^3\times\N_0^3\to \R$ by
\begin{equation}
H((\bar{z}_A,\bar{z}_{D_1},\bar{z}_{D_2}),(m,n_1,n_2)):=\bar{z}_{A}^m\bar{z}_{D_1}^{n_1}\bar{z}_{D_2}^{n_2}
\end{equation}
Using the generator criterion \cite[Proposition 1.2]{JK14}, we see that, for all $t\geq 0$,
\begin{equation}
\label{m10}
\E\left[H((\bar{z}_A(t),\bar{z}_{D_1}(t),\bar{z}_{D_2}(t)),(m(0),n_1(0),n_2(0)))\right]
=\E\left[H((\bar{z}_A(0),\bar{z}_{D_1}(0),\bar{z}_{D_2}(0)),(m(t),n_1(t),n_2(t)))\right].
\end{equation} 
Therefore $(L(t))_{t\geq 0}$ and $(\bar{z}(t))_{t\geq 0}$ are dual to each other with duality function $H$.

Since $(M(t))_{t \geq 0}$ is a irreducible and recurrent, we can define 
\begin{equation}
\label{m11}
\begin{aligned}
\pi_A = \lim_{t\to\infty} \P(M(t)=A)&=\frac{1}{1+\frac{c_1^A}{c_1^D}+\frac{c_2^A}{c_2^D}},\\ 
\pi_{D_1} = \lim_{t\to\infty} \P(M(t)=D_1)&=\frac{\frac{c_1^A}{c_1^D}}{1+\frac{c_1^A}{c_1^D}+\frac{c_2^A}{c_2^D}},\\ 
\pi_{D_2} = \lim_{t\to\infty} \P(M(t)=D_2)&=\frac{\frac{c_2^A}{c_2^D}}{1+\frac{c_1^A}{c_1^D}+\frac{c_2^A}{c_2^D}}.
\end{aligned}
\end{equation} 
Using the duality relation in \eqref{m10} together with \eqref{m11} and \eqref{defbar}, we find
\begin{equation}
\begin{aligned}
\lim_{t \to \infty}\E[\bar z_A(t)]
&=\pi_A \bar{z}_A(0)+\pi_{D_1} \bar{z}_{D_1}(0)+\pi_{D_2} \bar{z}_{D_2}(0)\\
&=\frac{1}{1+\frac{c_1^A}{c_1^D}+\frac{c_2^A}{c_2^D}}\, \bar{z}_A(0)
+\frac{\frac{c_1^A}{c_1^D}}{1+\frac{c_1^A}{c_1^D}
+\frac{c_2^A}{c_2^D}}\, \bar{z}_{D_1}(0)+\frac{\frac{c_2^A}{c_2^D}}{1+\frac{c_1^A}{c_1^D} 
+ \frac{c_2^A}{c_2^D}}\, \bar{z}_{D_2}(0)\\
&=z_A(0)+z_{D_1}(0)+z_{D_2}(0) = 1.
\end{aligned}
\end{equation}
Using the duality relation in\eqref{m10} once more, we get
\begin{equation}\label{mom1}
\lim_{t \to \infty}\E[\bar z_A(t)] = \lim_{t \to \infty}\E[\bar z_{D_1}(t)]=\lim_{t \to \infty}\E[\bar z_{D_2}(t)]=1.
\end{equation}
Computing the limiting second moment $\lim_{t\to\infty}\E[\bar z_A(t)^2]$ by duality, we obtain
\begin{equation}
\label{mom2}
\lim_{t\to\infty}\E[\bar z_A(t)^2]
=\lim_{t\to\infty}\sum_{\substack{i,j\in\\\{A,D_1,D_2\}}}\P(M_t^1=i)\,\bar{z}_i(0)\,\P(M_t^2=j)\,\bar{z}_j(0)
=\sum_{i\in\{A,D_1,D_2\}}\pi_i\bar{z}_i(0)\sum_{j\in\{A,D_1,D_2\}}\pi_j\bar{z}_j(0) = 1.
\end{equation}
Similarly, we find $\lim_{t\to\infty}\E[\bar z_{D_1}(t)^2]=1$ and $\lim_{t\to\infty}\E[\bar z_{D_2}(t)^2]=1$. Combining \eqref{mom1} and \eqref{mom2}, we find
\begin{equation}
\lim_{t \to \infty}\bar{z}_A(t) = \lim_{t \to \infty}\bar{z}_{D_1}(t) = \lim_{t \to \infty}\bar{z}_{D_2}(t) = 1.
\end{equation}
Hence we conclude that
\begin{equation}
\label{glim}
\begin{aligned}
\lim_{t \to \infty} z_A(t) =\frac{1}{1+\frac{c_1^A}{c_1^D}+\frac{c_2^A}{c_2^D}},\quad
\lim_{t\to\infty} z_{D_1}(t) =\frac{\frac{c_1^A}{c_1^D}}{1+\frac{c_1^A}{c_1^D}+\frac{c_2^A}{c_2^D}},\quad
\lim_{t \to \infty} z_{D_2}(t) =\frac{\frac{c_2^A}{c_2^D}}{1+\frac{c_1^A}{c_1^D}+\frac{c_2^A}{c_2^D}}.
\end{aligned}
\end{equation}
Continuing as in \eqref{B16alt}, we again find the single-colony version of \eqref{gh1*}-\eqref{gh2*} with 2 colours and no migration. 

%%%

\subsection{Alternative for Model 2: Three or more colours}
\label{app3c}

The argument in Appendix~\ref{app2c} can be extended to an $\mathfrak{m}\in\N$-colour seed-bank, by introducing sequences of variables $(Y_m(t))_{m=0}^\mathfrak{m}$ and $(Z_m(t))_{m=0}^\mathfrak{m}$ that count the number of $\heartsuit$-individuals in the colour-$m$ seed-bank at time $t$, respectively, the total number of individuals in the colour-$m$ seed-bank at time $t$. Let $\epsilon>0$ be the total rate at which an active individual becomes dormant, and define a probability vector $(p_m)_{m=0}^\mathfrak{m}$ such that $\epsilon_m=\epsilon p_m$ is the rate at which an active individual becomes dormant with colour $m$. Let $\delta_m$ be the rate at which $m$-dormant individuals become active. Via the same line of argument as in Appendix \ref{app2c}, we see that the equivalent of \eqref{mvan} reads
\begin{equation}
\label{mvan2}
\begin{aligned}
\d x(t) &=\sqrt{z_A(t)\frac{z_{A}-x(t)}{z_A(t)}\frac{x(t)}{z_A(t)}}\,\d w(t)
+ \sum_{m=0}^\mathfrak{m}\left[c_m^Dy_m(t)-c_m^Ax(t)\right]\,\d t,\\
\d y_m(t) &=\left[c_m^A x(t)-c_m^Dy_m(t)\right]\,\d t,\\
\d z_{A}(t) &=\sum_{m=0}^\mathfrak{m}\left[c_m^Dz_{D_m}(t)-c_m^Az_A(t)\right]\,\d t,\\
\d z_{D_m}(t) &=\left[c_m^Az_{A}(t)-c_m^D z_{D_m}(t)\right]\,\d t,\qquad 0\leq m\leq N.                  
\end{aligned}
\end{equation}
Solving the autonomous system describing $z(t)=(z_A(t),(z_{D_m}(t))_{m=0}^N)$ via duality, and subsequently substituting into \eqref{mvan2} the variables 
\begin{equation}
\label{B16}
\begin{aligned}
\bar{x}(t) &=\left(1+\sum_{n=0}^\mathfrak{m}\frac{c_n^A}{c_n^D}\right)
x\left(\frac{t}{1+\sum_{n=0}^\mathfrak{m}\frac{c_n^A}{c_n^D}}\right),\\ 
\bar{y}_m(t) &=\left(1+\sum_{n=0}^\mathfrak{m}\frac{c_n^A}{c_n^D}\right)
\left(\frac{c^D_m}{c^A_m}\right)y_m\left(\frac{t}{1+\sum_{n=0}^\mathfrak{m}\frac{c_n^A}{c_n^D}}\right),
\qquad 0\leq m\leq N,
\end{aligned}
\end{equation}
we find the single-colony version of \eqref{gh1*}--\eqref{gh2*} with $N$-colours and no migration. Migration can be added as  in Appendix \ref{appA}.

It is straightforward to derive the version \eqref{gh1*}--\eqref{gh2*} with $N$-colours and $M$ colonies. Afterwards
we can let $N,M\to\infty$ and use convergence of generators, to find \eqref{gh1*}--\eqref{gh2*}. The limit is unproblematic because we are interested in finite time horizons only.

%%%%%%%%%% APPENDIX C %%%%%%%%%%%%%%%%%%%%%%%

\section{Successful coupling}
\label{appC}

To prove Lemma \ref{C2} we proceed as in \cite{CG94}, with minor adaptations. The notation used in this appendix is the same as in Section \ref{ss.cos}. For model 1 we write down the full proof. The proof holds works for model 2 and 3 by invoking the colours $m\in\N_0$ and the SSDE in \eqref{gh1*}--\eqref{gh2*}, respectively, \eqref{gh1**}--\eqref{gh2**}.   

\paragraph{Proof of Lemma \ref{C2}.}

The proof consists of 5 steps.

\paragraph{\textbf{Step 1}.}
If $z\in E$ with $x_i=0$ and $x_k>0$ for some $k \neq i$, then 
\begin{equation}
\label{ceen}
\P_z\left(\exists\, t^*>0 \text{ such that } x_i(t)=0\ \forall\, t\in [0,t^*]\right)=0.
\end{equation}
		
\begin{proof}
Suppose that $z$ is such that $x_i=0$, but $x_k>0$ for some $i,k\in\G$. By \eqref{gh1}, 
\begin{equation}
x_i(t)=\int_{0}^{t}\sum_{j\in\G} a(i,j)[x_j(s)-x_i(s)]\,\d s+\int_{0}^t Ke[y_i(s)-x_i(s)]\,\d s
+\int_0^t\sqrt{g(x_i(s))}\,\d w_i(s).
\end{equation}
Suppose that there exists a $T>0$ such that $x_i(t)=0$ for all $t\in[0,T]$, and therefore $g(x_i(t))=0$. Then we obtain for all $t\in[0,T]$ that
\begin{equation}
\int_{0}^{t} \sum_{j\in\G} a(i,j) x_j(s)\,\d s +\int_{0}^t Key_i(s)\, \d s=0.
\end{equation}
Hence, by path continuity of $(Z(t))_{t\geq 0}$, we see that $y_i(t)=0$ for all $t\in[0,T]$, as well as $x_j(t)=0$ for all $j\in \G$ such that $a(i,j)>0$. Repeating this argument, we obtain by irreducibility of $a(\cdot,\cdot)$ that $x_k(t)=0$ for all $k\in\G$ and hence $y_k(t)=0$ for all $k\in\G$. By path continuity, this contradicts the assumption that $x_k(0)>0$. We conclude that \eqref{ceen} holds.
\end{proof}

\paragraph{\textbf{Step 2}.} 
If $\bar{z}\in E\times E$ and $g(x^1_i)\neq g(x^2_i)$, then for all $j$, 
\begin{equation}
\hat{\P}_{\bar{z}}\left(\exists\, t^*>0 \text{ such that } \Delta_j(t)=0\  \forall\, t\in [0,t^*]\right)=0.
\end{equation}
		
\begin{proof}
	
Note that the SSDE in \eqref{gh1}--\eqref{gh2} can be rewritten as
\begin{equation}\label{CM1}
\begin{aligned}
&\d z_{(i,R_i)}(t)
= \sum_{(j,R_j)\in\G\times\{A,D\}}b^{(1)}((i,R_i),(j,R_j))[z_{(j,R_j)}(t)-z_{(i,R_i)}(t)]\,\d t
+\sqrt{g(z_{(i,R_i)}(t))}\ 1_{\{R_i=A\}}\,\d w_i(t),\\
&\forall\, (i,R_i)\in\G\times\{A,D\},
\end{aligned}
\end{equation}	
with $b^{(1)}(\cdot,\cdot)$ defined as in \eqref{mrw}.	
	
Suppose that $\bar{z}$ is such that $g(x_i^1)\neq g(x_i^2)$. Suppose there exist a $T>0$ such that $\Delta_j(t)=0$ for all $t\in[0,T]$. Then also $\sqrt{g(x_j^1(t))}-\sqrt{g(x_j^2(t))}=0$ for all $t\in[0,T]$.  Using \eqref{CM1} on $\Delta_j(t)=z^1_{(j,A)}(t)-z^2_{(j,A)}(t)$, we obtain
\begin{equation}
0 = \int_0^t \sum_{(k,R_k)\in\G\times\{A,D\}}b^{(1)}((j,A),(k,R_k))
\left[\left(z^1_{(k,R_k)}(s)-z^2_{(k,R_k)}(s)\right)-\left(z^1_{(j,R_j)}(s)-z^2_{(j,R_j)}(s)\right)\right]\,\d s.
\end{equation}
Hence
\begin{equation}
\sum_{(k,R_k)\in\G\times\{A,D\}}b^{(1)}((j,A),(k,R_k))
\left[\left(z^1_{(k,R_k)}(t)-z^2_{(k,R_k)}(t)\right)-\left(z^1_{(j,R_j)}(t)-z^2_{(j,R_j)}(t)\right)\right] = 0 
\qquad \forall\, t \in [0,T].
\end{equation}
Using \eqref{CM1}, we can write the SDE for
\begin{equation}
\sum_{(k,R_k)\in\G\times\{A,D\}}b^{(1)}((j,A),(k,R_k))
\left[\left(z^1_{(j,R_j)}(t)-z^2_{(j,R_j)}(t)\right)-\left(z^1_{(i,R_i)}(t)-z^2_{(i,R_i)}(t)\right)\right],
\end{equation}
which yields that, for all $t\in[0,T]$,
\begin{equation}
\label{CM2}
\begin{aligned}
-&\int_{0}^{t}\sum_{(k,R_k)\in\G\times\{A,D\}}b^{(1)}((j,A),(k,R_k))
\left(\sqrt{g(z^1_{k,R_k}(s))}-\sqrt{g(z^2_{k,R_k}(s))}\,\right)1_{\{R_k=A\}}\,\d w_k(s)\\
&=\int_{0}^t\sum_{(k,R_k)\in\G\times\{A,D\}}b^{(1),2}((j,A),(l,R_l))
\left[\left(z^1_{(j,R_j)}(t)-z^2_{(j,R_j)}(t)\right)-\left(z^1_{(i,R_i)}(t)-z^2_{(i,R_i)}(t)\right)\right]\,\d s,
\end{aligned}
\end{equation}
where $b^{(1),2}(\cdot,\cdot)$ is the $2$-step kernel of $b^{(1)}(\cdot,\cdot)$.

The two process in the right-hand side form a process of bounded variation, while the process in the left-hand side is a continuous square-integrable martingale, whose quadratic variation is given by 
\begin{equation}
\label{kwadrvar}
\int_0^t \sum_{k\in\G} a(j,k)^2 \left( \sqrt{g(x_k^1(s))}-\sqrt{g(x_k^2(s))}\,\right)^2\, \d s.
\end{equation}
Since a square-integrable martingale of bounded variation is constant, it follows that \eqref{kwadrvar} equals $0$. Hence, for all $k$ such that $a(j,k)>0$, it follows that $g(x^1_k(t))=g(x^2_k(t))$ for all  $t\in[0,T]$. Moreover, the right-hand side of \eqref{CM2} is equal to $0$. Iterating the right-hand side of \eqref{CM2} further, we find by the irreducibility of $a(\cdot,\cdot)$ that $g(x^1_i(t))=g(x^2_i(t))$ for all $t\in[0,T]$, which contradicts the assumption on $\bar{z}$ that $g(x^1_i(0))\neq g(x^2_i(0))$. Hence we find that there does not exist a $T>0$ such that $\Delta_j(t)=0$ for all $t\in[0,T]$. 
\end{proof}
		
\paragraph{\textbf{Step 3}.} 

If $\bar{z}\in E\times E,\ i,k\in\G$ and $g(x_i^1)=g(x_i^2)$ with $\Delta_i<0$ and $\Delta_k>0$ for some $k \neq i$, then 
\begin{equation}
\label{cnegen}
\hat{\P}_{\bar{z}}\left(\exists\, t^*\in[0,\tfrac{1}{2}]\colon\, \Delta_i(t^*)<0,\ 
\Delta_k(t^*)>0 ,\ g(x_i^1(t^*))\neq g(x_i^2(t^*))\right)>0. 
\end{equation}

\begin{proof}
Note that by assumption we have $x_i^1<1$ and $x_k^1>0$. Let $t_0\in[0,\frac{1}{4}]$. If $x_i^1>0$, then set $t_0=0$. Otherwise, by Step 1 and path continuity, we find with probability 1 a $t_0\in[0,\frac{1}{4}]$ such that $x_i^1(t_0)>0$, $\Delta_i(t_0)<0$ and $\Delta_k(t_0)>0$. Let $\tilde z= \bar{z}(t_0)$. By the existence of $t_0$ and the Markov property, it is enough to prove that  
\begin{equation}
\label{ctien}
\hat{\P}_{\tilde{z}}\left(\exists\, t^*\in[0,\tfrac{1}{4}]\colon\, \Delta_i(t^*)<0,\, 
\Delta_k(t^*)>0,\ g(x_i^1(t^*))\neq g(x_i^2(t^*))\right)>0 
\end{equation}
in order to prove \eqref{cnegen}. Define the following two martingales:
\begin{eqnarray}
M_i(t)&=&\int_0^t\sqrt{g(x_i^1(s))}\,\d w_i(s),\\
M_k(t)&=&\int_0^t\left(\sqrt{g(x_k^1(s))}-\sqrt{2g(x_k^2(s))}\,\right)\,\d w_k(s).
\end{eqnarray}
Their corresponding quadratic variation processes are given by 
\begin{eqnarray}
\left\langle M_i(t)\right\rangle &=& \int_0^t g(x_i(s))\,\d s,\\
\left\langle M_k(t)\right\rangle &=& \int_0^t \left(\sqrt{g(x_k^1(s))}-\sqrt{2g(x_k^2(s))}\,\right)^2\,\d s.
\end{eqnarray}
By Knight's theorem (see \cite[Theorem V.1.9 p.183]{RY99}), we can write $M_i(t)$ and $M_k(t)$ as time-transformed Brownian motions:
\begin{eqnarray}
M_i(t) &=& w_i\left(\left\langle M_i(t)\right\rangle\right),\\
M_k(t) &=& w_k\left(\left\langle M_k(t)\right\rangle\right).
\end{eqnarray}
We may assume that $g(\tilde{x}^1_i)=g(\tilde{x}^2_i)$, otherwise we can set $t^*=0$. Recall that $0<\tilde{x}^1_i<1$, $\tilde\Delta_i<0$ and $\tilde \Delta_k>0$, and, since $0<g(\tilde{x}^1_i) = g(\tilde{x}^2_i)$, also $\tilde x_i^2<1$. Choose an $\epsilon\in(0,\frac{1}{15})$ such that $\tilde x_i^1,\tilde x_i^2\in[5 \epsilon,1-5\epsilon]$, $-\tilde \Delta_i>5 \epsilon$ and $\tilde \Delta_k>5\epsilon$. Let $\xi\in(0,\epsilon)$ be such that $g(\xi)<\min\{g(u)\colon\,\epsilon\leq u\leq 1-\epsilon\}$, and set $c_1=\min\{g(u)\colon\,\xi\leq u\leq 1-\xi\}$ and $c_2=\|g\|$. Then we can make the following estimates:
\begin{eqnarray}
\left\langle M_i(t)\right\rangle &\leq& c_2t\qquad \left\langle M_k(t)\right\rangle \leq c_2 t,\,t\geq0,\\
\left\langle M_i(t)\right\rangle &\geq& c_1t\qquad \text{ for } t\geq 0 \text{ such that } x_i(s)\in[\xi,1-\xi]\ \forall\, s\in[0,t].
\end{eqnarray}
Define $c_3=\min\{\frac{\xi}{2Ke},\frac{\xi}{2}\}$. Fix $T\in[0,c_3]$ and define 
\begin{eqnarray}
\Omega_0 &=& \left\{\min_{t\in[0,c_1T]} w_i(t)<-1, \max_{t\in[0,c_2T]} w_i(t)< \epsilon, 
\max_{t\in[0,c_2T]}|w_k(t)|<\epsilon\right\},\\
\Omega_1&=&\left\{\exists t^*\in[0,1] \text{ such that } \Delta_i(t^*)<0,\ 
\Delta_k(t^*)>0,\ g(x_i^1(t^*))=g(x_i^2(t^*))\right\}.
\end{eqnarray}
Note that $\P(\Omega_0)>0$. Therefore it suffices that $\Omega_0\subset\Omega_1$. 

We start by checking the conditions $\Delta_k$. Using \eqref{gh1}, we can write 
\begin{equation}
\begin{aligned}
\Delta_k(t) &=\Delta_k(0)+\int_0^t \sum_{l\in\G} a(k,l)(\Delta_l(s)-\Delta_k(s))\,\d s
+\int_0^t Ke \left[\delta_k(s)-\Delta_k(s)\,\d s\right]\\
&+\int_0^t \left(\sqrt{g(x_k^1(s))}-\sqrt{2g(x_k^2(s))}\right)^2\,\d w_k(s). 
\end{aligned}
\end{equation}
Since $|\Delta_l(t)|\leq 1$, $|\delta_k(t)|\leq 1$ for all $t\geq 0$, and $M_k(t)=w_k(\left\langle M_k(t)\right\rangle)$ for $t\in[0,T]$, we may estimate
\begin{eqnarray}
\Delta_k(t)>5 \epsilon- 2c_3-2Ke c_3- \epsilon =2\epsilon.
\end{eqnarray}
So, on $\Omega_0$, $\Delta_k(t)>0$ for all $t\in[0,T]$. By expanding $x_i^1(t)$, we find
\begin{equation}
\label{C24}
x_i^1(t)=x_i^1(0)+\int_0^t \sum_{l\in\G} a(i,l)(x_l^1(s)-x_i^1(s))\,\d s
+\int_0^t Ke (y_i^1(s)-x_i^1(s))\,\d s+M_i(t),
\end{equation}
so that on $\Omega_0$ we have, for $t\in[0,T]$,
\begin{eqnarray}
\label{C25}
x_i^1(t)<1-10\epsilon+c_3+Ke c_3+\epsilon=1-8\epsilon.
\end{eqnarray}
To check the conditions on $x^1_i(t)$ and $\Delta_i(t)$, we define the following random times:
\begin{equation}
\begin{aligned}
\sigma &= \inf\{t\geq0:x^1_i(t)=\xi\},\\
\tau &= \inf\big\{t>0\colon\,g(x_i^1(t))\neq g(x_i^2(t))\big\}.
\end{aligned}
\end{equation}
We will prove that, on $\Omega_0$, we have $\sigma<\tau$ and $x_i^2(\tau)\geq x_i^1(\tau)+3\epsilon$. To do so, we first prove that $\sigma<T$. Assume the contrary $\sigma\geq T$. Then by \eqref{C25} we have $x^1_i(t)\in[\xi,1-\xi]$ for all $t\in[0,T]$, which implies that $\min_{[0,T]} M_i(t)<-1$. Hence there exists a $\kappa$ such that, by \eqref{C24}, 
\begin{eqnarray}
x_i^1(\kappa)<1-10\epsilon+\epsilon-1<0.
\end{eqnarray}
However, this contradicts the fact that $x_i^1>0$ for all $t\geq 0$. We conclude that $\sigma<T$. Now suppose 
that $\tau>\sigma$. Expanding $\Delta_i$, we get, for $t<\tau$,
\begin{equation}
\Delta_i(t)=\Delta_i(0)+\int_0^t \sum_{l\in\G} a(i,l)(\Delta_l(s)-\Delta_i(s))\,\d s
+ \int_0^t Ke [\delta_i(s)-\Delta_i(s)]\,\d s, 
\end{equation}
which can be rewritten as
\begin{equation}
\label{C30}
x_i^2(t)=x^1_i(t)-x^1_i(0)+x_i^2(0)-\int_0^t\sum_{l\in\G} a(i,l)[\Delta_l(s)-\Delta_i(s)]\,\d s
- \int_0^t Ke [\delta_i(s)-\Delta_i(s)]\,\d s.
\end{equation}
By \eqref{C30}, we obtain, for $t\in[0,\sigma]$,
\begin{equation}
\label{C31}
\begin{aligned}
x_i^2(t) &\leq 1-5\epsilon+ 2 \epsilon +  2\epsilon=1-\epsilon,\\
x_i^2(t) &\geq x_i^1(t)+5 \epsilon -2 \epsilon\geq 3 \epsilon,
\end{aligned}
\end{equation}
so $x_i^2 (t)\in[\epsilon,1-\epsilon]$ for $t\in[0,\sigma]$. But then $g(x_i^1(\sigma))=g(\xi)<g(x_i^2(t))$ by the definition of $\xi$. Hence we obtain a contradiction and conclude that $\tau\leq \sigma$. From \eqref{C31} we obtain that $\Delta_i(t)<0$ for all $t\in[0,\tau]$, which concludes the proof that $\Omega_0\subset\Omega_1$.
\end{proof}

\paragraph{\textbf{Step 4}.}

If $\bar{z}\in E\times E$ and $\Delta_i<0, \Delta_j=0$, $\Delta_k>0$ for some $i,j,k$, then
\begin{equation}
\label{oudlem}
\hat{\P}_{\bar{z}}\left(\exists\, t^*\in[0,1]\colon\,\Delta_i(t^*)<0,\Delta_j(t^*) \neq 0, \Delta_k(t^*)>0\right)>0.
\end{equation}

\begin{proof}
Suppose that  $\bar{z}$ satisfies $\Delta_i<0$, $\Delta_j=0$, $\Delta_k>0$. Define
\begin{equation}
\begin{aligned}
\Gamma_0 &= \{\bar{z}\in E\times E:\Delta_i<0,\Delta_j\neq0, \Delta_k>0\},\\
\Gamma_1 &= \{\bar{z}\in E\times E:\Delta_i<0,\ g(x_i^1)\neq g(x_i^2),\ \Delta_k>0\}.
\end{aligned}
\end{equation}
By Step 3 and path continuity, there exists a $T\in[0,\frac{1}{2}]$ such that $\P^{\bar{z}}\left(\bar{z}(T)\in\Gamma_1\right)>0$. By the Markov property, 
\begin{eqnarray} 
\label{mo1706}
\hat{\P}_{\tilde{z}}\big(\exists\, t^*\in [0,1]\colon\,\bar{z}(t^*)\in \Gamma_0\big)
\geq \int_{\Gamma_1} \hat{\P}_{\bar{z}}(\bar{z}(T)\in d\tilde{z})\,
\hat{\P}_{\tilde{z}}\big(\exists t^*\in[0,\tfrac{1}{2}]\colon\,\bar{z}(t^*)\in \Gamma_0\big).
\end{eqnarray}
By path continuity, we can find for $\bar{z}\in\Gamma_1$ a $t^\prime$ such that, for all $t\leq t^\prime$, $\Delta_i(t)<0$, $\Delta_k(t)>0$ and $g(x_i^1(t))\neq g(x_i^2(t))$. By Step 2 there exists a $t^*<t^\prime$ such that $\bar{z}(t^*)\in\Gamma_0$. Hence both probabilities in the integral on the right-hand side of \eqref{mo1706} are positive.
\end{proof}
		
\paragraph{\textbf{Step 5}.} 

Proof of Lemma \ref{C2}.

\begin{proof}
Suppose that \eqref{eqlem58} holds for the pair $i,j,$ and $a(j,k)>0$, but \eqref{eqlem58} fails for the pair $i,k$. This implies that there exist $\epsilon_0>0$, $\delta_0>0$ and a positive increasing sequence $(t_n)_{n\in\N}$ of times with $t_n\to\infty$, such that 
\begin{equation}
\label{oei}
\lim_{t\to \infty} \hat{\P}_{\bar{z}}\left(\{\Delta_i(t)<\epsilon_0, \Delta_k(t)>\epsilon_0\}
\cup\{\Delta_i(t)>\epsilon_0, \Delta_k(t)<\epsilon_0\}\right)>\delta_0.
\end{equation} 	
By compactness of $E\times E$, there exists a  subsequence $t_{n_k}$ such that $\CL(\bar{z}(t_{n_k}))$ converges and \eqref{oei} holds. Let $\bar{\nu}=\lim_{k\to\infty} \CL(\bar{z}(t_{n_k}))$. Then 
\begin{equation}
\label{consq}
\begin{aligned}
\bar{\nu}\left(\{\Delta_i<\epsilon_0, \Delta_j>\epsilon_0\}
\cup\{\Delta_i>\epsilon_0, \Delta_j<\epsilon_0\}\right)
&=0,\\
\bar{\nu}\left(\{\Delta_j<\epsilon_0, \Delta_k>\epsilon_0\}
\cup\{\Delta_j>\epsilon_0, \Delta_k<\epsilon_0\}\right)
&=0,\\
\bar{\nu}\left(\{\Delta_i<\epsilon_0, \Delta_k>\epsilon_0\}
\cup\{\Delta_i>\epsilon_0, \Delta_k<\epsilon_0\}\right)
&>\delta_0.
\end{aligned}
\end{equation}
Assume without loss of generality that $ \bar{\nu}\left(\{\Delta_i<\epsilon_0, \Delta_k>\epsilon_0\}\right)>0$. Hence, by \eqref{consq},
\begin{eqnarray}
\label{C37}
\bar{\nu}\left(\{\Delta_i<\epsilon_0, \Delta_k>\epsilon_0\}\right)
=\bar{\nu}\left(\{\Delta_i<\epsilon_0, \Delta_j\in(-\epsilon_0,\epsilon_0), \Delta_k>\epsilon_0\}\right)>0.
\end{eqnarray}
For each $\bar{z}\in \{\Delta_i<\epsilon_0, \Delta_j\in(-\epsilon_0,\epsilon_0), \Delta_k>\epsilon_0\}$, Step 4 implies that 
\begin{equation}
\hat{\P}_{\bar{z}}\left(\exists\, t^*\in[0,1]\colon\,\Delta_i(t^*)<0,\Delta_j(t^*) \neq 0, \Delta_k(t^*)>0\right)>0,
\end{equation}
and therefore, by \eqref{C37},
\begin{equation}
\hat{\P}_{\bar{\nu}}\left(\exists\, t^*\in[0,1]\colon\,\Delta_i(t^*)<0,\Delta_j(t^*) \neq 0, \Delta_k(t^*)>0\right)>0.
\end{equation}
By path continuity, we can find $T\in[0,1]$ and $\epsilon>0$ such that 
\begin{equation}\label{nut}
\hat{\P}_{\bar\nu}\left(\Delta_i(T)<-\epsilon,\ \left|\Delta_j(T)\right|,\ \Delta_k(T)>\epsilon \right)>0.
\end{equation}
Let $\bar{\mu}(t_n)=\CL(\bar{z}(t_n))$. Then, by the Markov property and \eqref{nut}, 
\begin{equation}
\begin{aligned}
&\liminf_{n\to\infty} \hat{\P}_{\bar{\mu}(t_n)}\left(\Delta_i(T)<-\epsilon,\ 
\left|\Delta_j(T)\right|>\epsilon,\ \Delta_k(T)>\epsilon \right)\\
&=\liminf_{n\to\infty} \hat{\P}_{\bar{\mu}(0)}\left(\Delta_i(T+t_n)<-\epsilon,\ 
\left|\Delta_j(T+t_n)\right|>\epsilon,\ \Delta_k(T+t_n)>\epsilon \right) > 0.
\end{aligned}
\end{equation}
However, this violates \eqref{eqlem58} for either $i,j$ or $j,k$. We conclude that \eqref{oei} fails and that \eqref{oei} holds for $i,k$. By irreducibility, \eqref{oei} holds for all $i,k\in\G$. 
\end{proof} 

%%%%%%% APPENDIX D %%%%%%%%%%%%%%%%%%%%%%%%%%%%

\section{Bounded derivative of Lyapunov function}
\label{appe}

Recall from Section \ref{ss.cos} that 
\begin{equation}
h(t)= 2 \sum_{j \in \G} a(i,j)\,\hat\E\left[|\Delta_j(t)|\,
1_{\{\sign\Delta_i(t)\,\neq\,\sign\Delta_j(t)\}}\right] + 2Ke\,\hat\E\left[\big(|\Delta_i(t)| + |\delta_i(t)|\big)\,
1_{\{\sign\Delta_i(t)\,\neq\,\sign\delta_i(t)\}}\right].
\end{equation}
In this section we show that $h^\prime(t)$ exists for all $t>0$ and is bounded. To do so, we need to get rid of the indicator in the expectations.

Let 
\begin{equation}
h_{1,j}(t)=\hat\E\left[|\Delta_j(t)|\,
1_{\{\sign\Delta_i(t)\,\neq\,\sign\Delta_j(t)\}}\right] 
\end{equation}
and
\begin{equation}
h_2(t)=2Ke\,\hat\E\left[\big(|\Delta_i(t)| + |\delta_i(t)|\big)\,
1_{\{\sign\Delta_i(t)\,\neq\,\sign\delta_i(t)\}}\right].
\end{equation}
Then $h(t)=2\sum_{j\in\G}a(i,j)h_{1,j}(t)+h_2(t)$. We show that $h_{1,j}(t)$ is differentiable with bounded derivative for $j\in\G$. The proof of the differentiability of $h_2(t)$ is similar. Fix $t\geq 0$. Note that
\begin{equation}
\begin{aligned}
&\hat\E\left[|\Delta_j(t)|\,
1_{\{\sign\Delta_i(t)\,\neq\,\sign\Delta_j(t)\}}\right]\\
&=\hat\E\left[|\Delta_j(t)|\,1_{\{\sign\Delta_i(t)\,\neq\,\sign\Delta_j(t)\}} \mid |\Delta_i(t)|
\neq 0, |\Delta_i(t)|\neq 0 \right]\P\left(|\Delta_i(t)|\neq 0,|\Delta_j(t)|\neq  0\right)\\ 
&\qquad +\hat\E\left[|\Delta_j(t)|\,1_{\{\sign\Delta_i(t)\,\neq\,\sign\Delta_j(t)\}} \mid |\Delta_i(t)|= 0 
\text{ or } |\Delta_j(t)|= 0 \right]  \P\left(|\Delta_i(t)|= 0\text{ or  }|\Delta_j(t)|= 0 \right).
\end{aligned}
\end{equation}
Since $\Delta_i(t)$ and $\Delta_j(t)$ have zero local time, the second term vanishes and $\P(|\Delta_i(t)|\neq 0,|\Delta_j(t)|\neq  0)=1$. By continuity of $\Delta_i(t)$ and $\Delta_j(t)$, we can define sets
\begin{equation}
B_{n} = \big\{|\Delta_i(r)|>0 \text{ and } |\Delta_j(r)|>0, \forall r \in \mathcal{B}(t,\tfrac{1}{n})\big\}.
\end{equation}
Then
\begin{equation}
\cdots\subset B_n\subset B_{n+1}\subset B_{n+2}\subset \cdots,
\end{equation}
so
\begin{equation}
B_n=\bigcup_{i=0}^n B_i
\end{equation}
and we define 
\begin{equation}
B:=\bigcup_{i=0}^\infty B_n=\lim_{n\to\infty} B_n.
\end{equation}
Since $\P(|\Delta_i(t)|\neq 0,|\Delta_j(t)|\neq  0)=1$, it follows that $\P(B)=1$.

For each $B_n$, we have
\begin{equation}
\begin{aligned}
B_n=C_n\cup C_n^c, \qquad C_n = \big\{\omega\in B_n\colon\,  
1_{\{\sign\Delta_i(r)\,\neq\,\sign\Delta_j(r)\}}=1,\ \forall r\in\CB(t,\tfrac{1}{n})\big\},
\end{aligned}
\end{equation}
and, by the definition of $B_n$,
\begin{equation}
\cdots\subset C_n\subset C_{n+1}\subset C_{n+2}\subset \cdots \qquad \qquad
\cdots\subset C^c_n\subset C^c_{n+1}\subset C^c_{n+2}\subset \cdots
\end{equation}
Let $C=\bigcup_{i=0}^\infty C_i $ and $C^c=\bigcup_{i=0}^\infty C^c_i $ be such that $B=C\cup C^c$. Using \eqref{Ito1}, we obtain
\begin{equation}
\begin{aligned}
\frac{1}{s} \left(h_{1,j}(t+s)-h_{1,j}(t)\right)&=\frac{1}{s}\left[\hat\E\left[|\Delta_j(t)|\,
1_{\{\sign\Delta_i(t+s)\,\neq\,\sign\Delta_j(t+s)\}}\right]-\hat\E\left[|\Delta_j(t)|\,
1_{\{\sign\Delta_i(t)\,\neq\,\sign\Delta_j(t)\}}\right]\right]\\
&= \frac{1}{s}\left[\hat\E\left[|\Delta_j(t+s)|\,
1_{\{\sign\Delta_i(t+s)\,\neq\,\sign\Delta_j(t+s)\}}-|\Delta_j(t)|\,
1_{\{\sign\Delta_i(t)\,\neq\,\sign\Delta_j(t)\}} \mid B\right]\right]\\
&= \frac{1}{s}\left[\hat\E\left[|\Delta_j(t+s)|-|\Delta_j(t)| \mid C\right]\right]\P(C)\\
&= \frac{1}{s}\hat\E\left[\sum_{j\in\G} a(i,j) \int_{t}^{t+s}\sign(\Delta_i(r))
[\Delta_j(r)-\Delta_i(r)]\,\d r \mid  C\right]\,\P(C)\\
&\qquad + \frac{1}{s}\hat\E\left[\int_{t}^{t+s}\sign(\Delta_i(r))\left[\sqrt{g(x^1_i(r))} 
- \sqrt{g(x^2_i(r))}\,\right]\, \d w_i(r) \mid C\right]\,\P(C)\\
&\qquad +\frac{1}{s}\hat\E\left[ Ke\,\int_{t}^{t+s}\sign(\Delta_i(r))\big[\delta_i(r) 
- \Delta_i(r)\big]\,\d r \mid C\right]\,\P(C)\\
&= \sum_{j\in\G} a(i,j)\hat\E\left[\frac{1}{s} \int_{t}^{t+s}\sign(\Delta_i(r))
[\Delta_j(r)-\Delta_i(r)]\,\d r \mid C\right]\,\P(C)\\
&\qquad + \frac{1}{s}\hat\E\left[\int_{t}^{t+s}\sign(\Delta_i(r))\left[\sqrt{g(x^1_i(r))} 
- \sqrt{g(x^2_i(r))}\,\right]\, \d w_i(r) \mid C\right]\,\P(C)\\
&\qquad +\hat\E\left[ Ke\,\frac{1}{s}\int_{t}^{t+s}\sign(\Delta_i(r))\big[\delta_i(r) - \Delta_i(r)\big]\,
\d r \mid C\right]\,\P(C).
\end{aligned} 
\end{equation}
In the last equality, the first and third term are bounded, because $\Delta_i(t),\delta_i(t)$ and $\Delta_j(t)$ are continuous functions of $t$, and $\sign(\Delta_i)$ is constant since we conditioned on the set $C$. Therefore, letting $s\to 0$, it follows from the fundamental theorem of calculus that these terms are bounded. The second term is more involved. Since, on the set $C$,
\begin{equation}
\sign(\Delta_i(r))\left[\sqrt{g(x^1_i(r))} - \sqrt{g(x^2_i(r))}\,\right]
\end{equation} 
is a continuous function, we can rewrite the stochastic integral as a time-transformed Brownian motion: 
\begin{equation}
\begin{aligned}
\frac{1}{s}
&\hat\E\left[\int_{t}^{t+s}\sign(\Delta_i(r))\left[\sqrt{g(x^1_i(r))} - \sqrt{g(x^2_i(r))}\,\right]\, \d w_i(r) \mid C\right]\\
&=\frac{1}{s}\hat\E\left[W\left(\int_0^{t+s} \left[\sqrt{g(x^1_i(r))} - \sqrt{g(x^2_i(r))}\,\right]^2\d r\right)
-W\left(\int_0^t \left[\sqrt{g(x^1_i(r))} - \sqrt{g(x^2_i(r))}\,\right]^2\d r\right) \mid C\right].
\end{aligned}
\end{equation}
Since the normal distribution is differentiable with respect to its variance, we are done. 

%%%%%%%%%%%%%%%%%% REFERENCES %%%%%%%%%

\bibliography{seedbank}
\bibliographystyle{alpha}

%%%%%%%%%%%%%%%%%%%%%%%%%%%%%%%

\end{document}